\documentclass[10pt]{amsart}
\usepackage{latexsym,amsmath,amscd,amssymb,epsfig,graphicx,tabls,verbatim}
\usepackage{epstopdf}

\numberwithin{equation}{section}
\newtheorem{theorem}{Theorem}[section]
\newtheorem{proposition}[theorem]{Proposition}
\newtheorem{corollary}[theorem]{Corollary}
\newtheorem{lemma}[theorem]{Lemma}
\newtheorem{conjecture}[theorem]{Conjecture}

\newtheorem{problem}[theorem]{Problem}
\newtheorem{remark}[theorem]{Remark}
\newtheorem{defn}[theorem]{Definition}
\theoremstyle{definition}
\newtheorem{example}[theorem]{Example}

\newcommand{\Park}{{{\sf Park}}}
\newcommand{\Cat}{{\mathrm{Cat}}}
\newcommand{\Nar}{{\mathrm{Nar}}}
\newcommand{\Kirk}{{\mathrm{Kirk}}}

\newcommand{\codim}{{\mathrm{codim}}}
\newcommand{\rank}{{\mathrm{rank}}}

\newcommand{\Sym}{{\mathrm{Sym}}}

\newcommand{\Shi}{{{\sf Shi}}}
\newcommand{\Cox}{{{\sf Cox}}}

\newcommand{\Ind}{{\mathrm{Ind}}}
\newcommand{\Res}{{\mathrm{Res}}}
\newcommand{\Hilb}{{\mathrm{Hilb}}}
\newcommand{\Hom}{{\mathrm{Hom}}}

\newcommand{\triv}{{\mathbf{1}}}
\newcommand{\1}{{\mathbf 1}}

\newcommand{\init}{{\mathrm{in}}}

\newcommand{\symm}{{\mathfrak{S}}}
\newcommand{\mm}{{\mathfrak{m}}}
\newcommand{\gr}{{\mathfrak{gr}}}

\newcommand{\CC}{{\mathbb {C}}}
\newcommand{\ZZ}{{\mathbb {Z}}}

\newcommand{\RR}{{\mathbb {R}}}
\newcommand{\NN}{{\mathbb {N}}}

\newcommand{\xx}{{\mathbf{x}}}

\newcommand{\LLL}{{\mathcal{L}}}

\newcommand\qbin[3]{\left[\begin{matrix} #1 \\ #2 \end{matrix} \right]_{#3}}

\begin{document}


\title[Parking spaces]
{Parking spaces}

\author{Drew Armstrong}
\address{Dept. of Mathematics\\University of Miami\\
Coral Gables, FL 33146}
\email{d.armstrong@math.miami.edu}

\author{Victor Reiner}
\address{School of Mathematics\\
University of Minnesota\\
Minneapolis, MN 55455}
\email{reiner@math.umn.edu}

\author{Brendon Rhoades}
\address{Dept. of Mathematics\\
University of California - San Diego\\
La Jolla, CA 92093}
\email{bprhoades@math.ucsd.edu}

\thanks{First author partially supported by NSF grant DMS-1001825. Second author partially supported by NSF grant DMS-1001933.
Third author partially supported by NSF grant DMS-1068861}



\begin{abstract}
Let $W$ be a Weyl group with root lattice $Q$ and Coxeter number $h$. The elements of the finite torus $Q/(h+1)Q$ are 
called the $W$-{\sf parking functions}, and we call the 
permutation representation of $W$ on the set of $W$-parking functions 
the (standard) $W$-{\sf parking space}. Parking spaces have interesting connections to enumerative combinatorics, diagonal harmonics, and rational Cherednik algebras. In this paper we define two new $W$-parking spaces, called the {\sf noncrossing parking space} and the {\sf algebraic parking space}, with the following features:
\begin{itemize}
\item They are defined more generally for real reflection groups.
\item They carry not just $W$-actions, but $W\times C$-actions, where $C$ is the cyclic subgroup of $W$ generated by a Coxeter element.
\item In the crystallographic case, both are isomorphic to the standard $W$-parking space.
\end{itemize}
Our Main Conjecture is that the two new parking spaces are isomorphic to each other as permutation representations of $W\times C$. This conjecture ties together several threads in the Catalan combinatorics of finite reflection groups. 
Even the weakest form of the Main Conjecture has interesting combinatorial consequences, and
this weak form is proven in all types except $E_7$ and $E_8$.
We provide evidence for the stronger forms of the conjecture, 
including proofs in some cases, and suggest further directions for the theory.
\end{abstract}

\maketitle

\tableofcontents

\section{Introduction}
\label{intro-section}

Let $W$ be a finite Coxeter group (finite real reflection group). (We refer to the standard references \cite{BBCoxeter, Humphreys}.) The main goal of this paper is to define two new objects that deserve to be called ``parking spaces" for $W$. First we will describe the origin of the term ``parking space".

\subsection{Classical parking functions and spaces}
A classical {\sf parking function} is a map $f:[n]:=\{1,2,\ldots,n\}\to\NN$ for which the increasing rearrangement $(b_1\leq b_2\leq\cdots\leq b_n)$ of the sequence $(f(1),f(2),\ldots,f(n))$ satisfies $b_i\leq i$.\footnote{The name ``parking function" comes from a combinatorial interpretation due to Konheim and Weiss \cite{KonheimWeiss}. Suppose $n$ cars want to parking in $n$ linearly-ordered parking spaces and that car $i$ wants to park in spot $f(i)$. At step $i$, car $i$ tries to park in space $f(i)$. If the spot is full then $i$ parks in the first available spot $\geq f(i)$. If no such spot exists then $i$ leaves the parking lot. The function $f$ is called a parking function if every car is able to park.} Let $\Park_n$ denote the set of parking functions, which carries a permutation representation of the symmetric group $W=\symm_n$ via $(w.f)(i):=f(w^{-1}(i))$. We will call this the {\sf classical parking space}. We note that $\Park_n$ has size $(n+1)^{n-1}$ and its $W$-orbits are counted by the Catalan number $\frac{1}{n+1}\binom{2n}{n}$. Each orbit is represented by an {\sf increasing parking function}, but one could also parametrize the orbits by other Catalan objects, for example the {\em nonnesting partitions} defined in Definition~\ref{nonnesting-partitions-definition} 
below.

\begin{example}
\label{S_3-parking-example}
Below we exhibit the $(3+1)^{3-1}=4^2=16$ elements of  $\Park_3$, grouped into $5=\frac{1}{4}\binom{2 \cdot 3}{3}$ $\symm_3$-orbits by rows. The increasing orbit representatives are on the left.
$$
\begin{tabular}{|c|ccccc|}\hline
111&   &   &   &   & \\\hline
112&121&211&   &   & \\\hline
113&131&311&   &   & \\\hline
122&212&221&   &   & \\\hline
123&132&213&231&312&321\\\hline
\end{tabular}
$$
The permutation action of $\symm_3$ on $\Park_3$ decomposes over the orbits. Furthermore, note that the action of $\symm_3$ on the orbit $\{122,212,221\}$ is isomorphic to the action on cosets of a subgroup of type $\symm_2\times\symm_1$, which has Frobenius characteristic given by the complete homogeneous symmetric function $h_{2,1}=h_2h_1=(\sum_{1\leq i_1\leq i_2} x_{i_1}x_{i_2})(\sum_{1\leq i} x_i)$. Hence the action of $\symm_3$ on $\Park_3$ has Frobenius characteristic $$h_3+3\,h_{2,1}+h_{1,1,1}.$$
\end{example}

\subsection{Weyl group and real reflection group parking spaces}
Now we review how the action of $\symm_n$ on $\Park_n$ can be generalized to other finite {\sf Weyl groups} $W$ (crystallographic real reflection groups). Let $W$ act irreducibly on $V\cong\RR^n$. Since $W$ is crystallographic one can choose
\begin{equation}
\label{crystallographic-set-up}
\Delta\subseteq\Phi^+\subseteq\Phi\subseteq Q\subseteq V,
\end{equation}
where $\Delta,\Phi^+,\Phi,Q$ are the simple roots, positive roots, root system, and root lattice, respectively. Given $\Delta=\{\alpha_1,\ldots,\alpha_n\}$, let $s_i$ denote the reflection through the hyperplane $H_i=\alpha_i^\perp$. This endows $W$ with set of Coxeter generators $S=\{s_1,\ldots,s_n\}$. The product $s_1s_2\cdots s_n$ (taken in any order) is called a {\sf standard Coxeter element} 
of $W$. 
A {\sf Coxeter element} of $W$ is any $W$-conjugate of a standard Coxeter element.
The  Coxeter elements 
 form a single conjugacy class \cite[\S 3.16]{Humphreys} and the multiplicative order $h$ of any Coxeter element is called the {\sf Coxeter number} of $W$.

The role of $W$-{\sf parking functions} is played by the quotient $Q/(h+1)Q$,
having cardinality $(h+1)^n$, of the two nested rank $n$ lattices $(h+1)Q \subset Q$.
The $W$-action on $V$ induces a permutation $W$-action on $Q/(h+1)Q$ which we call the {\sf standard $W$-parking space}; see Haiman \cite[\S 2.4 and \S 7.3]{Haiman}. Recall that there is a standard partial order on positive roots $\Phi^+$ defined by setting $\alpha\leq\beta$ if and only if $\beta-\alpha\in\NN\Phi^+$. Then the $W$-orbits on $Q/(h+1)Q$ can be parametrized by antichains in the root poset $(\Phi^+,\leq)$ (also called {\sf $W$-nonnesting partitions}); see  Cellini and Papi \cite[\S 4]{CelliniPapi}, Sommers \cite[\S 5]{Sommers-B-stable}, and Shi \cite{Shi-plus-sign-types}. Finally, we note that the number of $W$-nonnesting partitions equals the {\sf $W$-Catalan number},
$$
\Cat(W)=\prod_{i=1}^n \frac{h+d_i}{d_i},
$$
where $\{d_1,d_2,\ldots,d_n\}$ is the multiset of {\sf degrees} for $W$. By definition, these are the degrees of any set of homogeneous algebraically-independent generators $f_1,\ldots,f_n$ for the invariant subalgebra $\RR[V]^W=\RR[f_1,\ldots,f_n]$ when $W$ acts on the algebra of polynomial functions $\RR[V]=\Sym(V^*)\cong\RR[x_1,\ldots,x_n]$ (see \cite{ST}).

\begin{example}
\label{type-A-example-1}
In type $A_{n-1}$ one has $W=\symm_n$ generated by $S=\{s_1,\ldots,s_{n-1}\}$, where the adjacent transposition $s_i=(i,i+1)$ corresponds to reflection through the hyperplane $H_{i,i+1}=\{x\in\RR^n: \langle x,e_i-e_{i+1}\rangle =0\}\subseteq\RR^n$. Thus $\symm_n$ acts irreducibly on the subspace $V=\1^\perp\subseteq\RR^n$ perpendicular to $\1=(1,1,\ldots,1)$. A Coxeter element $c=s_1s_2\cdots s_{n-1}$ is an $n$-cycle, with order $h=n$.

Here the root lattice $Q$ is the 
$\ZZ$-sublattice $\{ x \in \ZZ^n: \sum_i x_i=0\}$ inside $\ZZ^n$.  Hence
$
Q/(n+1)Q 
\quad \cong \quad
\left\{ 
x \in \ZZ_{n+1}^n: \sum_{i=1}^n x_i \equiv 0\bmod{n+1} 
\right\},
$
using the abbreviation $\ZZ_{n+1}:=\ZZ/(n+1)\ZZ$.
Haiman \cite[Prop 2.6.1]{Haiman} noted that 
the parking functions $\Park_n \subseteq\ZZ^n$ 
descend to give coset representatives for $\ZZ_{n+1}^n/\ZZ_{n+1}\1$.
We claim that this implies the $\symm_n$-actions 
on the sets $\Park_n$ and on $Q/(n+1)Q$ are isomorphic:
the finite abelian groups $\ZZ_{n+1}^n/\ZZ_{n+1}\1$ and 
$\left\{ 
x \in \ZZ_{n+1}^n: \sum_{i=1}^n x_i \equiv 0\bmod{n+1} 
\right\}$
are naturally {\it Pontrjagin dual}, so
that they carry {\it contragredient} $\symm_n$-representations, and permutation
representations are always self-contragredient.
\end{example}

In this paper we will propose two new {\sf $W$-parking spaces}. The first space is defined combinatorially, in terms of the {\sf $W$-noncrossing partitions}, and the second space is defined algebraically, as a quotient of the polynomial ring. These new spaces have the following features:
\begin{itemize}
\item For crystallographic $W$, both are isomorphic as $W$-representations to the standard $W$-parking space $Q/(h+1)Q$. However,
\item both are defined more generally for {\em non-crystallographic} finite reflection groups, and
\item both carry an additional $W\times C$-action, where $C=\langle c\rangle$ is the cyclic subgroup of $W$ generated by a Coxeter element $c$.
\end{itemize}
Our Main Conjecture states that the two new parking spaces are {\em isomorphic to each other} as $W\times C$-permutation representations, which has several consequences for the Catalan combinatorics of finite reflection groups. In the next section we will define the new parking spaces, state the Main Conjecture, and explore some of its consequences. After that we will give evidence for the conjecture, prove some special cases, and suggest problems for the future.

\section{Definitions and Main Conjecture}
\label{definitions-section}

First we will review the concepts of $W$-noncrossing and $W$-nonnesting partitions. For a full treatment, see \cite{Armstrong}.

Let $W$ be a finite Coxeter group (finite real reflection group) acting irreducibly on $V\cong\RR^n$ and orthogonally with respect to the inner product $\langle\cdot,\cdot\rangle$. Since $W$ may not be crystallographic there is no root lattice $Q$. However, one can still discuss the (noncrystallographic) {\sf root system} $\Phi=\{\pm \alpha\}$ of unit normals to the reflecting hyperplanes $H_\alpha$ for the set of all reflections $T=\{t_\alpha\}\subseteq W$. Let 
$\Cox(\Phi)=\{H_\alpha\}_{\alpha \in \Phi}$ 
denote the arrangement of reflecting hyperplanes and consider its lattice $\LLL$ of intersection subspaces (called {\em flats} $X$), ordered by reverse-inclusion.

\begin{example}
Again, consider the irreducible action of $W=\symm_n$ on the codimension-one subspace $V=\1^\perp\subseteq\RR^n$. The reflections of $\symm_n$ are precisely the transpositions $T=\{t_{i,j}=(i,j)\}$, where $t_{i,j}$ switches the $i$ and $j$ coordinates in $\RR^n$ and hence reflects in the hyperplane $H_{i,j}=\{x\in V :\langle x,e_i-e_j\rangle =0\}$. A typical intersection flat $X$ is defined by setting several blocks of coordinates equal, and corresponds to a set partition $\pi=\{B_1,\ldots,B_\ell\}$ of $[n]$ for which one has $H_{i,j}\subseteq X$ if and only if $i$ and $j$ occur in the same block of $\pi$. For example, in the case $W=\symm_9$, the flat $X$ defined by equations
$$
\{ x_1=x_3=x_6=x_7, \,\, x_4=x_5, \,\, x_8=x_9\}
$$ 
corresponds to the set partition 
$$
\pi= \{B_1,B_2,B_3,B_4\} =\{\{1,3,6,7\},\{2\},\{4,5\},\{8,9\}\}.
$$
\end{example}

\subsection{Noncrossing partitions for $W$}
The reflecting hyperplanes $\Cox(\Phi)$ decompose $V$ into connected components called {\sf chambers}, on which $W$ acts simply-transitively. Choose a {\sf fundamental chamber} $c_0$ for this action and for each reflecting hyperplane say that $c_0$ lies in the ``positive" half-space. This induces a choice of {\sf positive roots} $\Phi^+$ --- the ``positive'' normal vectors --- and {\sf simple roots} $\Delta=\{\alpha_1,\ldots,\alpha_n\}\subseteq\Phi^+$ --- corresponding to the hyperplanes that bound $c_0$. The reflections $S=\{s_1,\ldots,s_n\}$ defined by the simple roots $\Delta$ endow $W$ with a {\sf Coxeter system} $(W,S)$. Fix once and for all a {\sf Coxeter element} $c=s_1s_2\cdots s_n$ by ordering $S$. The {\sf Coxeter number} $h$ is the multiplicative order of $c$, and it turns out to equal the maximum of the fundamental degrees $d_1,\ldots,d_n$ mentioned earlier; \cite[\S 3.17]{Humphreys}.

For all $w\in W$, let $\ell_T(w)$ denote the minimal value of $k$ such that $w=t_1\cdots t_k$ with $t_i\in T$ for all $i$. One calls this the {\sf reflection length} on $W$. It turns out (see Carter \cite{Carter}) that $\ell_T(w)$ is equal to the codimension of the $W$-fixed space $V^w\subseteq V$, that is, $\ell_T(w)=\codim V^w=\dim V-\dim V^w$. Then one defines the {\sf absolute order} $\leq_T$ on $W$ by setting
\begin{equation}
u \leq_T v \Longleftrightarrow \ell_T(v) = \ell_T(u) + \ell_T(u^{-1}v)
\end{equation}
for all $u, v \in W$.  This poset is graded with rank function $\ell_T$. It has a unique minimal element $1\in W$ but in general many maximal elements (the elements with trivial fixed space), among which is the conjugacy class of Coxeter elements. Following Brady and Watt \cite[\S 4]{BradyWatt} 
and Bessis \cite[Defn. 2.1.3]{Bessis},
make the following definition.

\begin{defn}
\label{noncrossing-partitions-definition}
Define the poset $NC(W)$ of {\sf $W$-noncrossing partitions} as the interval $[1,c]_T$ in absolute order. Brady and Watt \cite[Theorem 2]{BradyWatt2} showed that this poset embeds into the partition lattice 
$$
\begin{array}{rcl}
NC(W)& \hookrightarrow &\LLL \\
w & \mapsto & V^w
\end{array}
$$
We will sometimes identify $NC(W)$ with its image under this embedding, and refer to the elements of $NC(W)$ as the {\sf noncrossing flats} $X\in\LLL$.
\end{defn}

Because Coxeter elements form a conjugacy class, and because conjugation by $w\in W$ is a poset isomorphism $[1,c]_T\cong [1,wcw^{-1}]_T$, we do not include $c$ in the notation for $NC(W)$. Furthermore, the map $w\mapsto c^{d}wc^{-d}$ is a poset automorphism, which corresponds to the action $X\mapsto c^dX$ on noncrossing flats.  

\begin{example}
For $W=\symm_n$ one can choose as Coxeter generators the set of adjacent transpositions $S=\{s_1,\ldots,s_{n-1}\}$, where $s_i=(i,i+1)$. A natural choice of Coxeter element is the $n$-cycle $c=s_1 s_2\cdots s_{n-1}=(1,2,\ldots,n)$. In this case the noncrossing flats $X\in\LLL$ correspond to partitions $\pi=\{B_1,\ldots,B_\ell\}$ that are {\em noncrossing} in the sense of 
Kreweras
 \cite{Kreweras} and Poupard \cite{Poupard}. That is, if one draws the numbers $[n]=\{1,2,\ldots,n\}$ clockwise around a circle then the convex hulls of its blocks $B_i$ are pairwise-disjoint (i.e. ``noncrossing''). For example the set partition
$\pi=\{\{1,3,6,7\},\{2\},\{4, 5\},\{8,9\}\}$ is noncrossing, as shown here:

\begin{center}
\includegraphics[scale=1]{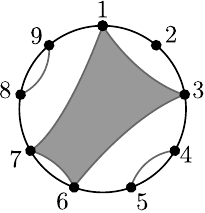}
\end{center}

\noindent Given such a noncrossing partition/flat $\pi=\{B_1,\ldots,B_\ell\}$, one recovers a permutation $w\in NC(W)$ by converting each block $B_i$ into a cycle, oriented clockwise, and then multiplying these. For example the noncrossing partition $\pi$ shown above corresponds to $w=(1,3,6,7)(2)(4,5)(8,9)$.
\end{example}

\subsection{Nonnesting partitions for $W$} The concept of nonnesting partitions is more recent than that of noncrossing partitions, and their original definition was completely general. The following definition is due to Postnikov \cite[Remark 2]{Reiner}.
\begin{defn}
\label{nonnesting-partitions-definition}
Let $W$ be a Weyl group,\footnote{Note that the same definition can be made for noncrystallographic types, but that the resulting structures do not have nice properties. It is an open problem to generalize ``nonnesting partitions" to this case.} so that there exists a 
root poset $(\Phi^+,\leq)$, defined by setting $\alpha\leq\beta$ if and only if $\beta-\alpha\in\NN\Phi^+$. The set $NN(W)$ of {\sf $W$-nonnesting partitions} is the collection of antichains (sets of pairwise-incomparable elements) in $(\Phi^+,\leq)$. Athanasiadis and Reiner \cite[Cor. 6.2]{AthanasiadisR} showed that there is an embedding
$$
\begin{array}{rcl}
NN(W)& \hookrightarrow &\LLL \\
A & \mapsto & \cap_{\alpha \in A} H_\alpha,
\end{array}
$$
which defines a partial order on $NN(W)$. We will sometimes identify $NN(W)$ with its image under this embedding, and speak about {\sf nonnesting flats} $X\in\LLL$.
\end{defn}

\begin{example}
For $W=\symm_n$, one can choose the root system and positive roots as follows:
$$
\begin{aligned}
\Phi&:=\{e_i-e_j: 1 \leq i \neq j \leq n\}, \\
\Phi^+&:=\{e_i-e_j: 1 \leq i < j \leq n\}. \\
\end{aligned}
$$
Draw the numbers $[n]=\{1,\ldots,n\}$ in a line. Then one has $\alpha=e_i -e_j \leq \beta=e_r - e_s$ in the root poset if and only if the semicircular arc connecting $i$ and $j$ is enclosed (i.e. ``nested'') inside the arc connecting $r$ and $s$. Thus an antichain $A$ in $\Phi^+$ gives rise to a set of arcs whose transitive closure is a {\sf nonnesting set partition} of $[n]$. For example, let $W=\symm_6$. The antichain $A=\{e_1-e_3, e_2-e_5, e_3-e_6\}$ corresponds to the nonnesting flat
$X=\cap_{\alpha \in A} H_\alpha=\{x_1=x_3=x_6,\,\, x_2=x_5\}$ and the nonnesting set partition $\pi=\{\{1,3,6\},\{2,5\},\{4\}\}$. Below we show the antichain $A$ (where $ij$ stands for $e_i-e_j$). The corresponding nonnesting partition is on the right.

\medskip
\begin{center}
\includegraphics[scale=1]{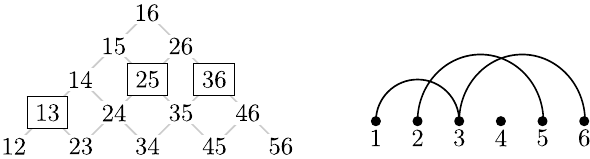}
\end{center}
\end{example}

\subsection{Noncrossing and nonnesting parking functions} Now we define a family of $W$-modules, two of which are the standard $W$-parking space, and the new $W$-noncrossing parking space.

Define an equivalence relation on the set of ordered pairs 
$$
W\times\LLL=\{(w,X):w\in W, X\in\LLL\}
$$ 
by setting $(w,X) \sim (w',X')$ when one has both
\begin{itemize}
\item $X=X'$, that is, the flats are equal, and
\item $wW_X = w'W_X$ where $W_X$ is the pointwise $W$-stabilizer of the flat $X$.
\end{itemize}
Let $[w,X]$ denote the equivalence class of $(w,X)$, and note that the left-regular action of $W$ on itself in the first coordinate descends to a $W$-action on equivalence classes:
$
v.[w,X]:=[vw,X].
$

\begin{defn}
Define the {\sf $W$-nonnesting} and {\sf $W$-noncrossing parking functions} as the following $W$-stable subsets of $(W\times\LLL)/\sim$:
$$
\begin{aligned}
\Park^{NN}_W&:=\{[w,X]: w \in W \text{ and } X \in NN(W)\}\\
\Park^{NC}_W&:=\{[w,X]: w \in W \text{ and } X \in NC(W)\}.
\end{aligned}
$$
\noindent
One could alternately phrase\footnote{Thanks to an anonymous referee for suggesting this phrasing.} 
$\Park^{NN}_W$ (resp. $\Park^{NC}_W$) as the set of
pairs $(wW_X,X)$ with $X$ in $NN(W)$ (resp. in $NC(W)$) 
and $wW_X$ in $W/W_X$ a coset modulo $W_X$.

Both of these subsets inherit the $W$-action $v.[w,X]=[vw,X]$, but the second set $\Park^{NC}_W$ also has a $W\times C$-action, defined by letting $C$ act on the right:
$$
(v,c^d).[w,X]:=[vwc^{-d} \,\, , \,\, c^d(X)]
$$
\end{defn}

\begin{example}
\label{type-A-pictorial-example}
One can think of type $A_{n-1}$ noncrossing parking functions pictorially. Consider $[w,X]\in\Park_{\symm_n}^{NC}$. Then the noncrossing flat $X$ corresponds to a noncrossing partition $\pi=\{B_1,\ldots,B_\ell\}$ of $[n]$ and the permutation $w\in\symm_n$ is considered only up to its coset $wW_X$ for the Young subgroup $W_X=\symm_{B_1}\times\cdots\times\symm_{B_\ell}$. Thus one can think of $w$ as a function that assigns to each block a set of labels $w(B_i)=\{w(j)\}_{j\in B_i}$. Visualize this by labeling each block $B_i$ in the noncrossing partition diagram by the set $w(B_i)$. For example, the leftmost picture in the figure below shows  the noncrossing parking function $[w,X]\in\Park_{\symm_9}^{NC}$, where $X$ corresponds to $\pi = \{\{ 1, 3, 7 \},\{ 2\} \{ 4, 5 , 6\}, \{ 8, 9\} \}$, and
$$
\begin{aligned}
wW_X&=\left(
\begin{matrix}
1 & 2 & 3 & 4 & 5 & 6 & 7 & 8 & 9 \\
1 & 8 & 3 & 2 & 4 & 6 & 9 & 5 & 7  
\end{matrix}
\right)W_X  \\
&=\left(
\begin{array}{ccc|c|ccc|cc}
1 & 3 & 7 & 2 &4 & 5 & 6 & 8 & 9 \\
1 & 3 & 9 & 8 &2 & 4 & 6 & 5 & 7
\end{array}
\right)W_X.
\end{aligned}
$$

\begin{center}
\includegraphics[scale=1]{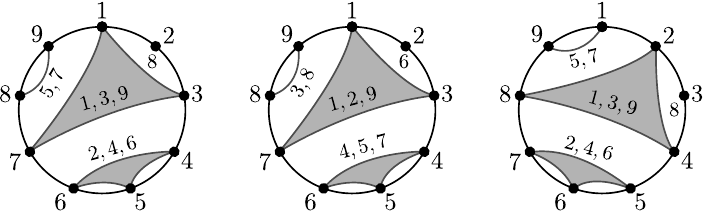}
\end{center}
\medskip

The $W\times C$-action is also easy to visualize. If one chooses the standard $n$-cycle for our Coxeter element, $c=s_1s_2\cdots s_{n-1}=(1,2,\ldots,n)$, then the action $[w,X]\mapsto [vwc^{-d},c^d(X)]$ of the element $(v,c^d)\in \symm_n \times C$ just permutes the labels by $v\in\symm_n$ and rotates the picture by $2\pi d/n$ clockwise.  For example, let $v=(1)(2,7,3)(4)(5,8,6)(9)\in\symm_9$. The middle and right pictures in the figure show $[vx,X]$ and $[wc^{-1},c(X)]$, respectively.
\end{example}

\subsection{The coincidence of $W$-representations}
\label{coincidence-of-perm-reps-section}
We consider several {\it $W$-sets} (and $W \times C$-sets), that is, sets with a $W$-action, such as $\Park^{NN}_W, \Park^{NC}_W$, or the set of cosets 
$W/W_X=\{wW_X: w \in W\}$ under left-translation.

\begin{defn}
Given a finite group $G$ and two $G$-sets
$A_1, A_2$, write $A_1 \cong_G A_2$ to mean there is 
a $G$-equivariant bijection $A_1 \rightarrow A_2$.
Let $\CC[A]$ denote 
the $G$-permutation representation associated to the 
$G$-set $A$.  
Given two $G$-representations $V_1,V_2$, write
$V_1 \cong_{\CC[G]} V_2$ if they are isomorphic 
as $\CC[G]$-modules, or equivalently if they have the same character.
\end{defn}

Note that $A_1 \cong_G A_2$ implies $\CC[A_1] \cong_{\CC[G]} \CC[A_2]$,
but the converse need not 
hold\footnote{Although see Proposition~\ref{W-sets-converse} for a situation
where the converse does hold.}.
As an example of the notations, 
$\CC[W/W_X] \cong_{\CC[G]} \Ind_{W_X}^W \triv_{W_X}$, and
by definition 
\begin{equation}
\label{definitional-permutation-reps}
\begin{aligned}
\Park^{NN}_W &\cong_W \bigsqcup_{X \in NN(W)} W/W_X, \\
\Park^{NC}_W &\cong_W \bigsqcup_{X \in NC(W)} W/W_X.  \\
\end{aligned}
\end{equation}
Recall that Shi \cite{Shi-plus-sign-types} and Cellini and Papi \cite[\S 4]{CelliniPapi} established a bijection between antichains $NN(W)$ in the root poset and $W$-orbits on the finite torus $Q/(h+1)Q$. It turns out that this bijection sends a nonnesting flat $X\in NN(W)$ to a $W$-orbit whose stabilizer is conjugate to $W_X$; see Athanasiadis \cite[Lemma 4.1, Theorem 4.2]{Athanasiadis2}. Hence one has an isomorphism of $W$-sets:
$$
Q/(h+1)Q \cong_W \Park^{NN}_W.
$$
Furthermore, it was checked {\bf case-by-case} in \cite[Theorem 6.3]{AthanasiadisR} that $NN(W)$ and $NC(W)$ contain the same number of flats in each $W$-orbit on the partition lattice $\LLL$. (This generalizes an observation of Stanley in type $A$; see \cite[Proposition 2.4]{StanleyParking}.) Hence one also has this isomorphism of $W$-sets:
\begin{equation}
\label{case-by-case-AthR-check}
\Park^{NN}_W \cong_W \Park^{NC}_W.
\end{equation}

\begin{example}
In the case of the symmetric group $W=\symm_n$ with $n\leq 3$, {\bf every} partition of $[n]$ is both noncrossing and nonnesting, hence in these cases one has an equality $\Park_{\symm_n}^{NN}=\Park_{\symm_n}^{NC}$. For $n=4$ there is exactly one crossing set partition, namely $\pi_1=\{\{1,3\},\{2,4\} \}$, and exactly one nesting set partition, namely $\pi_2=\{\{1,4\},\{2,3\}\}$. However, note that $\pi_1$ and $\pi_2$ correspond to flats $X_1,X_2$ in the same $W$-orbit on $\LLL$, hence one still has an isomorphism because
$$
\Park^{NN}_{\symm_4}
\cong_W \bigsqcup_{X \in \LLL - \{X_2\}} W/W_X
\cong_W \bigsqcup_{X \in \LLL - \{X_1\}} W/W_X
\cong_W \Park^{NC}_{\symm_4}.
$$

\end{example}

\subsection{The algebraic $W$-parking space}
\label{park-alg-defn-section}
Next we will define a new $W$-parking space {\em algebraically}. Extend scalars from $V\cong\RR^n$ to $\CC^n$ and let $W$ act on the polynomial algebra
isomorphic to the symmetric algebra of the dual space $V^*$
$$
\CC[V]=\Sym(V^*)=\CC[x_1,\ldots,x_n],
$$
where $x_1,\ldots,x_n$ is a $\CC$-basis for $V^*$.
It is a subtle consequence of the representation 
theory of {\sf rational Cherednik algebras} 
(see Berest, Etingof and Ginzburg \cite{BEG}, Gordon \cite{Gordon}, and
Etingof \cite{Etingof} \cite[\S 4]{EtingofMa}),
that there exist {\sf homogeneous systems of parameters}
$\Theta=(\theta_1,\ldots,\theta_n)$ of degree $h+1$ inside
$\CC[V]$, with the following property:
\begin{quote}
The $\CC$-linear
isomorphism defined by
\begin{equation}
\label{strong-hsop-setup}
\begin{array}{rcl}
V^* &\longrightarrow & \CC \theta_1 + \cdots + \CC \theta_n \\
x_i &\longmapsto & \theta_i
\end{array}
\end{equation}
is $W$-equivariant.  In particular, the linear span
$\CC \theta_1 + \cdots +\CC \theta_n$ carries a copy of the
dual\footnote{Since the reflection groups considered
in this paper are real, one always has a $\CC[W]$-module isomorphism 
$V \cong_{\CC[W]} V^*$. However, we use $V^*$ to be consistent with the case of complex 
reflection groups.  We hope that some of our conjectures can be extended to complex reflection groups, at least in the well-generated case considered by Bessis and Reiner
\cite{BessisR}, and perhaps even in the arbitrary case considered 
by Gordon and Griffeth \cite{GordonGriffeth}.} reflection representation $V^*$.
\end{quote}

In some cases (for example in types $B/C$ and $D$, and in cases of rank $\leq 2$) one can choose the coordinate functionals $x_1,\ldots,x_n$ in $V^*$ such that $(\theta_1,\ldots,\theta_n)=(x_1^{h+1},\ldots,x_n^{h+1})$; see Sections \ref{type-BC-proof-section} and \ref{type-D-proof-section} below. However, already in type $A$ the construction of such an hsop is somewhat tricky --- see Haiman \cite[Proposition 2.5.4]{Haiman} or Chmutova and Etingof \cite[\S 3]{ChmutovaEtingof} --- and for the exceptional real reflection 
groups we know of no simple construction.

When one has such a $\Theta$ it is natural to consider the quotient ring $\CC[V]/(\Theta)$. For example, this quotient occurs in the rational Cherednik theory \cite{Gordon, BEG, Etingof, EtingofMa} as the finite-dimensional irreducible module $L_{(h+1)/h}(\triv_W)$ corresponding to the trivial representation $\triv_W$ of $W$, and taken at the rational parameter $\frac{h+1}{h}$. In the crystallographic case it is known (see \cite[\S 5]{Gordon}, \cite[eqn. (5.5)]{BessisR}) that $\CC[V]/(\Theta)$ is isomorphic to the permutation representation $\CC[Q/(h+1)Q] \cong_{\CC[W]} \CC[\Park_W^{NN}]$. We wish to {\bf deform} the quotient $\CC[V]/(\Theta)$ somewhat, considering instead the following.

\begin{defn}
Let $W$ be a real irreducible reflection group $W$, with $\Theta=(\theta_1,\ldots,\theta_n)$ and $(x_1,\ldots,x_n)$ chosen as in \eqref{strong-hsop-setup}. Consider the ideal 
\begin{equation}
\label{deformed-ideal-definition}
(\Theta-\xx) :=(\theta_1 - x_1 , \ldots, \theta_n-x_n)
\end{equation}
and define the {\sf algebraic $W$-parking space} as the quotient
ring 
$$
\Park^{alg}_W := \CC[V] / (\Theta-\xx).
$$
This quotient has the structure of a $W\times C$-representation since the ideal $(\Theta-\xx)$ is stable under the following two commuting actions on $\CC[V]=\CC[x_1,\ldots,x_n]$:
\begin{itemize}
\item the action of $W$ by linear substitutions, and
\item the action of $C=\langle c \rangle$ 
by scalar substitutions $c^d(x_i) = \omega^{-d} x_i$, with $\omega:=e^{\frac{2\pi i}{h}}$.
\end{itemize}
\end{defn}

Note that we have not included the choice of $\Theta$ in the 
notation $\Park^{alg}_W$.  This is justified by the following
proposition, which we will prove in Section~\ref{deformation-section} below.

\begin{proposition}
\label{deformation-isomorphism}
For every irreducible real reflection group $W$,
and for {\bf any} choice of $\Theta$ satisfying
\eqref{strong-hsop-setup}, one has an isomorphism 
of $W \times C$-representations
$$
\Park^{alg}_W :=
 \CC[V]/(\Theta-\xx) 
\quad \cong_{\CC[W \times C]} \quad
\CC[V]/(\Theta).
$$
\end{proposition}

While it is conceivable that the ring structures of $\CC[V]/(\Theta)$ and $\CC[V]/(\Theta - \xx)$ may depend on $\Theta$,
we will only be interested in the $\CC[W \times C]$-module structures of these objects, and hence denote them both
by $\Park^{alg}_W$, suppressing reference to $\Theta$.
This proposition has several important consequences. For example, forgetting the $C$-action, one obtains the following string of $\CC[W]$-module isomorphisms:
\begin{equation}
\label{string-of-isomorphisms}
\begin{aligned}
\Park^{alg}_W &:= \CC[V]/(\Theta-\xx) \\
&\cong_{\CC[W]} \CC[V]/(\Theta) \\
&\cong_{\CC[W]} \CC[Q/(h+1)Q] \\
&\cong_{\CC[W]} \CC[\Park^{NN}_W] \\
&\cong_{\CC[W]} \CC[\Park^{NC}_W].
\end{aligned}
\end{equation}

\subsection{The Main Conjecture}
\label{main-conjecture-subsection}
Our Main Conjecture is a stronger and more direct geometric connection 
between the first and last terms in \eqref{string-of-isomorphisms}, eliminating
the crystallographic hypothesis.  Here we consider the
subvariety/zero locus $V^\Theta$ cut out in $V$ by the ideal $(\Theta-\xx)$.
The notation $V^\Theta$ is meant to be suggestive of the fact that
this subvariety is the subset of $V$ fixed pointwise 
under the map $\Theta: V \longrightarrow V$ that sends the element
of $V$ having coordinates $(x_1,\ldots,x_n)$ to the element of $V$ having
coordinates $(\theta_1,\ldots,\theta_n)$.  Note that 
Proposition~\ref{deformation-isomorphism} implies
$\CC[V]/(\Theta-\xx)$ is a $\CC$-vector space of dimension $(h+1)^n$,
since $\CC[V]/(\Theta)$ is.  In particular,
this subvariety $V^\Theta \subseteq V$ can contain at most $(h+1)^n$
points, and contains exactly this many points when counted with multiplicity.
Note that the $W \times C$-action on $V$, in which an element $c^d$ in $C$ 
scales $V$ by the root-of-unity $\omega^d$, restricts to one on $V^\Theta$.
\vskip.1in
\noindent
{\bf Main Conjecture.} 
{\it
Let $W$ be an irreducible real reflection group.
\newline 
\begin{tabular}{ll}
{\sf (strong version) }& {\bf For all choices} of $\Theta$ as in \eqref{strong-hsop-setup}, one has that ...\\
{\sf (intermediate version) } &
{\bf There exists a} choice of $\Theta$ as in \eqref{strong-hsop-setup} 
such that ...
\end{tabular}
\newline
... the subvariety $V^{\Theta}$ inside $V$ consists of $(h+1)^n$
distinct points, that have a $W \times C$-equivariant bijection
to the set $\Park^{NC}_W$, that is, $V^{\Theta} \cong_{W \times C} \Park^{NC}_W$.
}
\vskip.1in
\noindent
In other words, the ideal $(\Theta)$ has its zero locus supported
at the origin in $V$, with scheme structure as a fat point
of multiplicity $(h+1)^n$, while $(\Theta-\xx)$ cuts out a
subvariety $V^{\Theta}$ that deforms this fat point at the origin,
blowing it apart into $(h+1)^n$ reduced points, 
carrying a $W \times C$-representation identifiable as\footnote{Actually,
this would show that the two  $W \times C$-representations are
contragredient, since $\CC[V]/(\Theta-\xx)$ is the space of
functions on $V^\Theta$.  But permutation representations are
self-contragredient, so this does not affect the $W \times C$-isomorphism
statement.}
that of the noncrossing parking functions $\Park^{NC}_W$.

The appendix to this paper contains a uniform
argument communicated to the authors by Etingof
\cite{EtingofComm} which implies that for any irreducible real reflection group $W$, there exists
an hsop $\Theta$ as in \eqref{strong-hsop-setup} such that $V^{\Theta}$ has
$(h+1)^n$ distinct points.  Unfortunately, the hsop $\Theta$ involved arises from deep 
results in the theory of rational Cherednik algebras and determining the
$W \times C$-set structure of $V^{\Theta}$ has been so far intractable.

\begin{example}
To gain intuition, we check the strong version of the Main Conjecture
in type $A_1$, that is, when $W$ has rank $1$.  Here $W = \{1, s\}$ acts on 
$V = \CC^1$ by $s.v = -v$.  Hence $W$ acts on the basis element
$x$ for $V^*$, as well as on $\CC[V] = \Sym(V^*)=\CC[x]$,
via $s(x)=-x$.  The only choice of a Coxeter element
is $c = s$, with Coxeter number $h = 2$.  An hsop $\theta$ of degree $h+1=3$
must take the form $\theta=\alpha x^3$ for some $\alpha$ in $\CC^\times$.
Any such $\theta$ has the map $x \mapsto \alpha x^3$ 
being $W$-equivariant as required by \eqref{strong-hsop-setup}.

As depicted schematically here
\begin{center}
\epsfxsize=30mm
\epsfbox{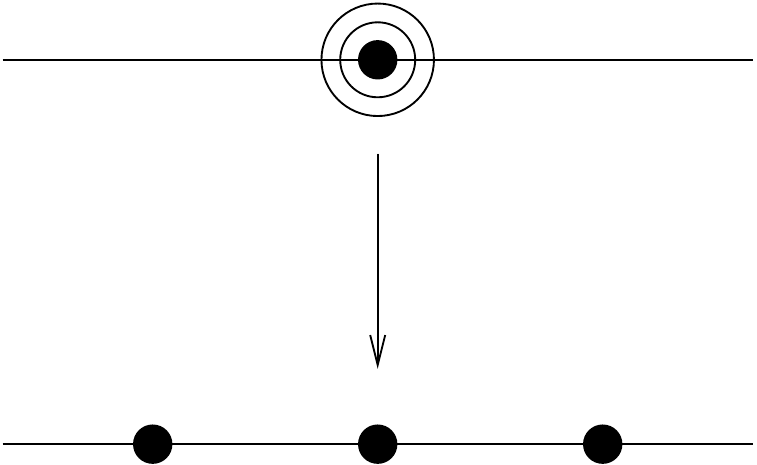}
\end{center}
the subvariety of $V=\CC^1$ cut out by $\theta=\alpha x^3$ is the origin $\{0\}$,
cut out with multiplicity $3(=h+1)$.  Meanwhile, the subvariety $V^\theta$ cut out by
$$
\theta-x=\alpha x^3-x=\alpha x(x- \alpha^{-\frac{1}{2}})(x+\alpha^{-\frac{1}{2}})
$$
consists of $3$ {\sf distinct} points $\{0,\pm  \alpha^{-\frac{1}{2}} \}$.
Furthermore, one can check that
the $W \times C$-permutation module structure
on these $3$ points of $\Park^{alg}_W=V^\theta$
matches that of $\Park^{NC}_W$.  
In this case, there are only two flats $X$ in the intersection
lattice $\LLL$, namely the origin $\{0\}$ and the whole space $V$,
both of which are noncrossing flats, since $c=s$ has $V^s=\{0\}$
and $1$ has $V^1=V$.  Thus $NC(W)=[1,s] \cong \{ \{0\}, V \}$.
\begin{enumerate}
\item[$\bullet$]The singleton $W \times C$-orbit $\{0\}$ within
$V^\theta$ matches the singleton $W \times C$-orbit
$\{ \,\, [1,\{0\}] \,\, \}$ in $\Park^{NC}_W$, both carrying trivial $W \times C$-action.
\item[$\bullet$]
The two-element $W \times C$-orbit $\{\pm \alpha^{-\frac{1}{2}} \}$
within $V^\theta$ matches the two-element $W \times C$-orbit 
$\{ [1,V], [s,V] \}$ in $\Park^{NC}_W$: in both cases either
element in the orbit has $W \times C$-stabilizer subgroup 
equal to $\{ (1,1), (s,s) \}$.
\end{enumerate}
\end{example}

It can be shown that certain $W \times C$-subsets of $V^{\Theta}$ carry actions 
which provide evidence for the Main Conjecture.
Define the {\sf dimension} of a point $p \in V^{\Theta}$ to
be the minimal dimension of a flat $X \in \LLL$ such that
$p \in X$.  Then one has a disjoint union decomposition
\begin{equation*}
V^{\Theta} = V^{\Theta}(0) \uplus V^{\Theta}(1) \uplus \dots \uplus V^{\Theta}(n),
\end{equation*}
where $n = \dim(V)$ is the rank of $W$ and $V^{\Theta}(d)$ denotes the set
of $d$-dimensional points in $V^\Theta$.  Moreover, each of the sets $V^{\Theta}(d)$
is $W \times C$-stable.

Note that since $NC(W)=[1,c]_T$ has unique
bottom and top elements $1,c$, inside
$\Park^{NC}_W$ there are two corresponding
$W \times C$-orbits:
the {\it $W$-trivial orbit} $\{ [1,\{0\}] \}$, 
carrying the trivial $W \times C$-representation, and 
the {\it $W$-regular orbit} $\{ [w,V] \}_{w \in W}$
carrying the coset representation 
$\left( W \times C \right)/ \langle (c,c) \rangle$.
Section~\ref{extreme-orbits-section} proves the following
result on the counterparts\footnote{See
Haiman \cite[Cor. 7.4.1]{Haiman}
for the counterpart to (iii) for $\Park^{NN}_W$, in
the guise of $Q/(h+1)Q$.}
in $V^\Theta$, showing that  $V^\Theta(d)$ for $d \in \{0,1,n\}$ are
all as described in the Main Conjecture.
\begin{proposition}
\label{extreme-orbits-proposition}
For $W$ an irreducible real reflection group of rank $n$,
and for all choices of $\Theta$ as in \eqref{strong-hsop-setup}, 
one has the following.
\begin{enumerate}
\item[(i)] The set $V^{\Theta}(0) = \{ 0 \}$ of $0$-dimensional points in $V^{\Theta}$
is the unique $W \times C$-orbit in $V^{\Theta}$ carrying the trivial 
$W \times C$-representation.
\item[(ii)] There exists a $W \times C$-equivariant injection
$V^{\Theta}(1) \hookrightarrow \Park^{NC}_W$ 
whose image is precisely the set of 
noncrossing parking functions of the form $[w,X]$ with $\dim(X) = 1$.
\item[(iii)] The set $V^{\Theta}(n)$ of $n$-dimensional points is the unique $W$-regular
orbit of points in $V^{\Theta}$.
\end{enumerate}
Furthermore, every point in the subsets $V^{\Theta}(0)$, $V^{\Theta}(1)$, and $V^{\Theta}(n)$
of $V^{\Theta}$ is cut out by the ideal $(\Theta - \xx)$ in a reduced fashion, that is, 
with multiplicity one.
\end{proposition}

When $n \leq 2$, the sets $V^{\Theta}(0)$, $V^{\Theta}(1)$, and $V^{\Theta}(n)$
exhaust the variety $V^{\Theta}$, and Proposition  \ref{extreme-orbits-proposition} implies that
$V^{\Theta}$ consists of $(h+1)^n$ distinct points with the same $W \times C$-action as
$\Park^{NC}_W$.
This completes the proof of the strong version of the Main Conjecture 
 in rank $\leq 2$.  We remark that it is possible to prove the strong version of the Main
Conjecture in rank 2 directly by explicitly computing $V^{\Theta}$ for all relevant hsops 
$\Theta$.

\begin{corollary}
\label{cor:ranksmallstrong}
The strong version of the Main Conjecture holds in rank $\leq 2$.
\end{corollary}

The
intermediate version is verified for the Weyl groups of types $B/C$
in Section~\ref{type-BC-proof-section} and type $D$
in  Section~\ref{type-D-proof-section},
using representations of $W$ via signed permutation matrices, and
picking the simple hsop $\Theta = (x_1^{h+1},\ldots,x_n^{h+1})$.

In fact, all the important consequences of the Main Conjecture
are even implied by the following weakest version,
on the level of $W \times C$-characters,
generalizing the $W$-isomorphism 
\eqref{case-by-case-AthR-check}.
This weak form is just shy of being a theorem;
it has been verified for all irreducible
real reflection groups except $E_7, E_8$ (where the computations became
too big for our computing power).

\vskip.1in
\noindent
{\bf Main Conjecture.} {\sf (weak version) }\\
{\it
Let $W$ be an irreducible real reflection group and consider the cyclic subgroup $C:=\langle c\rangle\leq W$ generated by a Coxeter element $c\in W$. Then one has an isomorphism of $W\times C$-representations
\begin{equation}
\label{weak-form-equation}
\CC[\Park^{alg}_W] \quad \cong_{\CC[W \times C]} \quad \CC[\Park^{NC}_W].
\end{equation}
}
\vskip.1in
\noindent
This weakest version of the conjecture is approachable
via a simple explicit formula for the
$W \times C$-character for $\Park^{alg}_W$, explained next.
Proposition~\ref{deformation-isomorphism} tells us
that $\Park^{alg}_W$ has the same $W \times C$-character
as that of the graded vector space $\CC[V]/(\Theta)$.
The latter space has $c$ acting by the scalar 
$c^d$ on the $d^{th}$ homogeneous component $\left( \CC[V]/(\Theta)\right)_d$.  
Thus taking advantage of the known\footnote{Calculated already in 
\cite[Theorem 1.11]{BEG}, \cite[\S 4]{BessisR}, \cite[\S 5]{Gordon}
using a Koszul resolution for $(\Theta)$ and a famous
result of Solomon \cite{Solomon} on $W$-invariant differential
forms with polynomial coefficients.}
{\it graded $W$-character} on $\CC[V]/(\Theta)$
\begin{equation}
\label{park-graded-W-character}
\sum_d \chi_{\left( \CC[V]/(\Theta)_d \right)}(w) q^d =
\frac{ \det(1-q^{h+1}w) }{  \det(1 - qw)    } \\
\end{equation}
one immediately deduces the following.

\begin{proposition}
\label{algebraic-character-prop}
For any element $w$ in an irreducible real 
reflection group $W$, and any $d$ in $\ZZ$, one can evaluate
the $W \times C$-character of $\Park^{alg}_W$ explicitly as
$$
\chi_{\Park^{alg}_W}(w,c^d)
  = \lim_{q \rightarrow \omega^d} \frac{ \det(1-q^{h+1}w) } 
                                                 {  \det(1 - qw)    } \\
  =(h+1)^{\mathrm{mult}_w(\omega^d)}
$$
where $\mathrm{mult}_w(\omega^d)$ denotes the multiplicity of the eigenvalue
$\omega^d$ when $w$ acts on the (complexified) reflection representation $V$.
\end{proposition}

\noindent
One can then approach the weak version of the Main Conjecture
by comparing this formula to
an explicit computation of the 
$W \times C$-character $\chi_{\Park^{NC}_W}$.
Section~\ref{type-A-proof-section} verifies the weak version of
the Main Conjecture in type $A$
by such a combinatorial argument.
The authors also thank Christian Stump 
for writing software in SAGE 
that verified this character equality
for the irreducible exceptional real reflection groups
$H_3, H_4, F_4, E_6$. Types $E_7$ and $E_8$ are still unverified.

The status of the various versions of the Main Conjecture is summarized in 
Table~\ref{table:truth}, together with the locations of the corresponding proofs.
\footnote{While this paper was under review, the third author was able to prove the intermediate form
in type A$_{n-1}$, as well as the strong form in type A$_3$.  The writeup is in preparation.}
The reader may wonder why we were able to prove the intermediate form of the Main Conjecture
in types BCD but only the weak form in type A.  This is because in types BCD there exist
simple hsops satisfying the conditions of the intermediate form which make the locus 
$V^{\Theta}$ easy to work with;  indeed, when one lets these groups act on $V = \CC^n$ by their 
standard matrix representations, one can let 
$(\theta_1, \dots, \theta_n)$ be given by $\theta_i = x_i^{h+1}$, where $x_i$ is the standard
coordinate function.  
With this choice of hsop, we have that 
\begin{equation*}
V^{\Theta} = \{ (v_1, \dots, v_n) \in \CC^n \,:\, \text{$v_i = 0$ or $v_i^h = 1$ for $1 \leq i \leq n$} \}.
\end{equation*}
Analyzing the action of $W$ on this set is straightforward.
Unfortunately, constructing hsops in type A is substantially more difficult.  
While explicit hsops have been written down in type A, their complicated forms make
it difficult to get a handle on the loci $V^{\Theta}$.  This is a somewhat anomalous case where
the theory in type A is more difficult than in the other classical types.

\begin{table}
\caption{Status of the Main Conjecture}
\centering
\begin{tabular}{|c| c| }
\hline
Reflection group $W$ & Strongest version of the Main Conjecture proven for $W$ \\
\hline
\hline
rank $\leq 2$ & Strong; Corollary~\ref{cor:ranksmallstrong} \\ \hline
type $A_{n-1}$ & Weak; Section~\ref{type-A-proof-section} \\ \hline
type $B_n/C_n$ & Intermediate; Section~\ref{type-BC-proof-section} \\ \hline
type $D_n$ & Intermediate; Section~\ref{type-D-proof-section} \\ \hline
type $H_3, H_4, F_4, E_6$ & Weak; computer verification \\ \hline
type $E_7, E_8$ & Open \\
\hline
\end{tabular}
\label{table:truth}
\end{table}

\section{Consequences of the Main Conjecture}
\label{consequences section}

\subsection{First consequence:  the $W$-action gives
$\Park^{NN}_W \cong_W \Park^{NC}_W$}

As mentioned earlier, the $W$-set isomorphism
$\Park^{NN}_W \cong_W \Park^{NC}_W$ in
\eqref{case-by-case-AthR-check} was checked case-by-case.  However,
this would follow immediately from 
even the weak form of the Main Conjecture by forgetting the
$C$-action, and using only the
$\CC[W]$-module isomorphism 
$\CC[\Park^{NN}_W] \cong_{\CC[W]} \CC[\Park^{NC}_W]$, via the following proposition.

\begin{proposition}
\label{W-sets-converse}
For real reflection groups $W$ and
finite $W$-sets $A_1, A_2$ whose $W$-orbits are all
of the form $W/W_X$ for varying reflection subgroups $W_X$,
one has $\CC[A_1] \cong_{\CC[W]} \CC[A_2]$ if and only if $A_1 \cong_W A_2$.
\end{proposition}
\begin{proof}
Only the forward direction is nontrivial, so 
assume $\CC[A_1] \cong_{\CC[W]} \CC[A_2]$.
One must show for each $W$-conjugacy class of reflection subgroups $W_X$,
or equivalently, for each $W$-orbit $W.X$ of intersection subspaces $X$, that
the number of $W$-orbits in $A_i$ having stabilizers conjugate to $W_X$
is the same for $i=1,2$.  This follows after showing
that the characters $\Ind_{W_X}^W \triv_{W_X}$ corresponding to different
$W$-orbits $W.X$ are
linearly independent, via a
triangularity\footnote{The authors thank A.R. Miller for suggesting this 
argument, similar to \cite[Proof of Thm. 1]{Miller}.} argument:  
picking for each $W$-orbit $W.X$ an element 
$w_X$ in $W$ whose fixed space is $X$, the character value
$\chi_{\Ind_{W_X}^W \triv_{W_X}} (w_Y) \neq 0$ exactly
when orbits $W.X, W.Y$ have a choice of nested representatives 
$X \subseteq Y$.
\end{proof}

\subsection{Second consequence: the 
$C$-action is a cyclic sieving phenomenon}

Another consequence of even the weak form of the Main Conjecture
comes from comparing the residual $C$-representations
carried by the $W$-invariant subspaces in the two sides of
the $W \times C$-isomorphism \eqref{weak-form-equation}.

The definition of $\CC[\Park^{NC}_W]$ shows that
its $W$-invariant subspace carries the usual
permutation representation of $C=\langle c \rangle$ acting on 
the noncrossing flats $NC(W)$, that is, for any noncrossing flat $X=V^w$
with $w$ in $[1,c]_T$, it has image $c^d(X)=V^{c^d w c^{-d}}$, another
noncrossing flat.  In particular, the trace of $c^d$ in this
representation is the number of noncrossing flats $X$ in $NC(W)$ with
$c^d(X)=X$.

On the other hand, the $C$-representation structure
on the $W$-invariant subspace of $\Park^{alg}_W$ can be deduced
from the fact that $c$ acts via the scalar $\omega^m$ on the $m^{th}$
graded component of the graded vector space $\left( \CC[V]/(\Theta) \right)^W$.
This means that $c^d$ will act with trace on $\Park^{alg}_W$
given by substituting $q=\omega^d$ into the Hilbert series in $q$
for  $\left( \CC[V]/(\Theta) \right)^W$, which is known to be the 
{\it $q$-Catalan number for $W$} (see \cite{BEG, BessisR, Gordon}):
\begin{equation}
\label{q-Catalan-formula}
\Cat(W,q) = \prod_{i=1}^n \frac{1-q^{h+d_i}}{1-q^{d_i}}.
\end{equation}
Thus whenever the weak form the Main Conjecture is true, it re-proves
(the {\it real} reflection group case of) this 
main result from \cite{BessisR}:

\begin{theorem}\cite[Theorem 1.1]{BessisR}
\label{CSP-theorem}
For irreducible real reflection groups $W$,
$$
(\,\, NC(W)\,\, , \,\, \Cat(W,q) \,\, ,\,\, C\,\, )
$$
exhibits the {\bf cyclic sieving phenomenon} defined in \cite{RStantonWhite}:
the number of elements $X$ in $NC(W)$ with
$c^d(X)=X$ is the evaluation of $\Cat(W,q)$ at $q=\omega^d$.
\end{theorem}

We remark that, modulo the weak form of the Main Conjecture, our proof of 
Theorem~\ref{CSP-theorem} is both uniform and follows from a character 
computation rather than direct enumeration.  The proof of 
\cite[Theorem 1.1]{BessisR} had to rely on certain facts that
had been checked case-by-case, and used a counting argument.

\subsection{Third consequence: Kirkman and Narayana numbers for $W$}
\label{third-consequence-section}

Recall that ignoring the $C$-action in the
the weak form of the Main Conjecture gives 
the $\CC[W]$-module isomorphism that follows from
\eqref{case-by-case-AthR-check},
which has previously only been checked case-by-case in \cite[Thm. 6.3]{AthanasiadisR}.  
This $\CC[W]$-module isomorphism already has an interesting corollary for the
{\it Narayana and Kirkman polynomials for $W$}.  These can be defined by  
\begin{equation}
\label{Narayana-polynomial-definition}
\begin{array}{rl}
\Nar_W(t) &:=\sum_{X \in NC(W)} t^{\dim_\CC(X)} \\
          &=\sum_{w \in [1,c]_T} t^{\dim V^w} \\
 & \\
\Kirk_W(t) &:=\Nar_W(t+1) \\
           &:= \sum_{A} t^{n-|A|}
\end{array}
\end{equation}
where in the last sum, $A$ ranges over all {\it clusters} in the
{\it cluster complex of finite type} associated to $W$
by Fomin and Zelevinsky \cite{FominZelevinsky}
in the crystallographic case, or over all subsets $A$ of rays
that form the cones of the 
{\it Cambrian fan} associated to $W$ by Reading \cite{Reading}
in the real reflection group case.
When $W$ is crystallographic, 
it is also known (see \cite[Theorem 5.9]{FominReading-ParkCity})
that 
$\Kirk_W(t)=\sum_F t^{n-\dim(F)}$
where $F$ ranges over faces in the interior of the 
dominant region of the Shi arrangement for $W$, discussed in Section~\ref{Shi-chamber-section}
below.

Let $V$ be the geometric representation of $W$ with dimension $n$. Recall that the exterior powers $\wedge^k V$ for $k\in\{0,\ldots,n\}$ are irreducible and pairwise inequivalent, with $\wedge^0 V$ equal to the trivial representation and $\wedge^n V$ equal to the determinant representation \cite[Theorem 5.1.4]{GeckPfeiffer}. In the case of the symmetric group $\mathfrak{S}_m$ the irreducible representation $\wedge^k V$ corresponds to the hook-shaped partition $(m-k,1^k)\vdash m$. Denote by
$\Park(W)$ either of the
equivalent $\CC[W]$-representations 
$\CC[\Park^{NC}_W] \cong_{\CC[W]} \CC[\Park^{alg}_W]$.

\begin{corollary}
\label{Narayana-Kirkman-corollary}
The Kirkman numbers are the multiplicities of the exterior powers $\wedge^k V$ in the irreducible decomposition of the parking space $\Park(W)$. That is, for $W$ irreducible one has
$$
\Kirk_W(t)= \sum_{k=0}^n \langle \,\, \chi_{\wedge^k V} \,\ 
              , \,\, \chi_{\Park(W)} \,\, \rangle_W \cdot t^k.
$$
\end{corollary}
The type $A_{n-1}$ special case of this corollary
is essentially an observation of Pak and Postnikov \cite[eqn (1-6)]{PakPostnikov}.
Section~\ref{Narayana-Kirkman-section} deduces Corollary~\ref{Narayana-Kirkman-corollary} from the isomorphism \eqref{case-by-case-AthR-check}.  It also explains
how calculations of Gyoja, Nishiyama, and Shimura\footnote{The authors thank Eric Sommers for pointing them to this reference} \cite{GNS} can be used to give
explicit formulas for  natural $q$-analogues of the coefficients of $\Kirk_W(t)$,
that we call {\it $q$-Kirkman numbers}, in types $A, B/C, D$.

Section~\ref{Shi-chamber-section} explains how, 
for crystallographic $W$, the set of nonnesting parking 
functions $\Park^{NN}_W$ labels in a natural way the regions of the 
{\it Shi arrangement} of affine hyperplanes, giving a labeling
closely related to one in the type $A$ case defined by
Athanasiadis and Linusson \cite{ALShi}.

Section~\ref{open-problems-section} closes with some open problems.

\section{Proof of Proposition~\ref{deformation-isomorphism}}
\label{deformation-section}

Recall the statement here.

\vskip.1in
\noindent
{\bf Proposition~\ref{deformation-isomorphism}.}
{\it
For every irreducible real reflection group $W$,
and for {\bf any} choice of $\Theta$ satisfying
\eqref{strong-hsop-setup}, one has an isomorphism 
of $W \times C$-representations
$$
\Park^{alg}_W :=
 \CC[V]/(\Theta-\xx) 
\quad \cong_{\CC[W \times C]} \quad
\CC[V]/(\Theta).
$$
}
\vskip.1in
\noindent
In fact, this is really a commutative algebra statement having little to
do with the $W \times C$-action, as we now explain.  

Consider the polynomial ring $S:=k[x_1,\ldots,x_n]$ with its standard
grading in which each $x_i$ has degree $1$, so a 
monomial $\xx^\alpha=x_1^{\alpha_1} \cdots x_n^{\alpha_n}$ has 
degree $|\alpha|=\sum_i \alpha_i$.  
For any not necessarily homogeneous 
polynomial $f=\sum_\alpha c_\alpha \xx^\alpha$ of top degree $d$, let 
$\init(f) := \sum_{\alpha: |\alpha|=d} c_\alpha \xx^\alpha$
be its initial form with respect to the degree.
Given a finite set of polynomials $\{f_1,\ldots,f_\ell\}$ in $S$, 
denote by 
$$
\begin{aligned}
J&=(f_1,\ldots,f_\ell)\\
I&=(\init(f_1),\ldots,\init(f_\ell))
\end{aligned}
$$
the ideals $J,I$ in $S$ generated by $\{f_i\}_{i=1}^\ell$,
and by their initial forms $\{\init(f_i)\}_{i=1}^\ell$. 
Thus the quotient $S/I$ inherits the standard grading,
but $R=S/J$ generally does not.
We wish to compare them via the increasing filtration of $k$-subspaces
\begin{equation}
\label{nongraded-ring-filtration}
\{0\} \subseteq R_{(0)} \subseteq  R_{(1)} \subseteq R_{(2)} \subseteq \cdots \subseteq R
\end{equation}
where $R_{(d)}$ is the image under the composite map
$S_{\leq d} \hookrightarrow S \twoheadrightarrow S/J$ of
the subspace $S_{\leq d}:=\bigoplus_{i \leq d} S_i$;
here $S_d$ is the $d^{th}$ graded component of $S$.
As $\bigcup_{d=0}^\infty R_{(d)}=R$, one has a $k$-vector space isomorphism
\begin{equation}
\label{isomorphism-from-nongraded-quotient}
R \cong \gr(R):=\bigoplus_{d \geq 0} R_{(d)}/R_{(d+1)}.
\end{equation}

\begin{lemma}
\label{filtration-factor-isomorphism-lemma}
With the above notations, 
assume that $\init(f_1),\ldots,\init(f_\ell)$ form an $S$-regular sequence.
Then the following composite map $\varphi_d$ is surjective, 
$$
S_d \hookrightarrow S_{\leq d} \rightarrow R^{(d)} \rightarrow R^{(d)}/R^{(d-1)}
$$
with $\ker(\varphi_d)=I_d:=I \cap S_d$, and hence induces 
$k$-vector space isomorphisms
$$
\begin{array}{rcl}
\left( S/I \right)_d &\cong& R^{(d)}/R^{(d-1)}, \\
S/I &\cong& \gr(R). \\
\end{array}
$$
\end{lemma}

\noindent
Assuming Lemma~\ref{filtration-factor-isomorphism-lemma}
for the moment, we explain how it implies 
Proposition~\ref{deformation-isomorphism}.
Combining the last assertion of the 
lemma with \eqref{isomorphism-from-nongraded-quotient}
gives a $k$-vector space isomorphism $S/I \rightarrow R$.  
Furthermore, note that if the ideals $I,J$ happen to both
be $G$-stable for a finite group $G$ of
grade-preserving automorphisms of $S$, 
then all of the maps involved will be 
$G$-equivariant.  This holds in
the set-up of Proposition~\ref{deformation-isomorphism},
where $S=\CC[V]=\CC[x_1,\ldots,x_n]$, with
$G=W \times C$, acting as before, taking
$$
\begin{array}{rclcrcl}
f_i&=&\theta_i-x_i,&\text{ so that }&R&=&\CC[V]/(\theta-\xx),\\ 
\init(f_i)&=&\theta_i,&\text{ so that }&S/I&=&\CC[V]/(\theta).
\end{array}
$$
Lastly, note that $\theta_1,\ldots,\theta_\ell$ form an $S$-regular
sequence because they are a system of parameters and $S$ is a Cohen-Macaulay ring.  Hence the hypotheses of
Lemma~\ref{filtration-factor-isomorphism-lemma} are satisfied,
and this would complete the proof of Proposition~\ref{deformation-isomorphism}.

\begin{proof}[Proof of Lemma~\ref{filtration-factor-isomorphism-lemma}]
The map $S_d \overset{\varphi_d}{\longrightarrow} R_{(d)}/R_{(d-1)}$ 
surjects since  
$
S_{\leq d} =S_d \oplus S_{\leq d-1} 
$
and $S_{\leq d-1}$ is annihilated by the composite surjection
$
S_{\leq d} \rightarrow R_{(d)} \rightarrow R_{(d)}/R_{(d-1)}.
$

Any $f$ in $I_d$ will lie in the kernel of $\varphi_d$, as 
there exist homogeneous $h_i$ expressing
$$
f=\sum_{i=1}^\ell h_i \init(f_i)
=\sum_{i=1}^\ell h_i f_i -
\sum_{i=1}^\ell h_i (f_i - \init(f_i)).
$$
The first sum is in $\ker(\varphi_d)$ as it is in $J$, 
so it maps to zero in $R=S/J$ and hence maps to zero in $R_{(d)}$.  The
second sum lies in $S_{\leq d-1} \subset \ker(\varphi_d)$.
Hence $I_d \subset \ker(\varphi_d)$.

To show $\ker(\varphi_d)=I_d$,
it remains to show $S_d \cap (S_{\leq d-1} + J) \subseteq I_d$.
Assume $f$ lies in 
$S_d \cap (S_{\leq d-1} + J)$,
so $\deg(f)=d$, and one can find an expression 
\begin{equation}
\label{kernel-expression}
f = g + \sum_{i=1}^\ell g_i f_i 
\end{equation}
where $g$ in $S_{\leq d-1},$ and $g_i$ in $S$.  Choose this expression
in such a way that
the quantity $d_0:=\max\{ \deg(g_i f_i)\}_{i=1}^\ell$ is 
{\it minimized}.  

Note that $\deg(f)=d$ forces $d_0 \geq d$.
If $d_0=d$, then taking the degree $d_0$ component
in \eqref{kernel-expression} gives the expression 
$f=\sum_{i=1}^\ell \init(g_i) \init(f_i)$, 
showing that $f$ lies in $I_d$, and we are done.

On the other hand, if $d_0 > d$, then taking the degree $d_0$ component in \eqref{kernel-expression} gives
$0=\sum_{i=1}^\ell \init(g_i) \init(f_i)$.  In other words,
when one creates the free graded $S$-module $S^\ell$ having
ordered $S$-basis elements $(e_1,\ldots,e_\ell)$
with $\deg(e_i):=\deg(f_i)$, the vector
$\sum_{i=1}^\ell \init(g_i) e_i$ is in the kernel of
the $S$-module map defined $S$-linearly by
$$
\begin{array}{rcl}
S^\ell &\longrightarrow& S \\
e_i &\longmapsto & \init(f_i)
\end{array}
$$ 
Because the $\{\init(f_i)\}_{i=1}^\ell$ are a regular sequence,
this vector must be an $S$-linear combination 
of {\it Koszul syzygies}\footnote{We only use this small part
of the regular sequence hypothesis, that first syzygies are Koszul.}, that is, one
can write
\begin{equation}
\label{Koszul-syzygy-expression}
\sum_{i=1}^\ell \init(g_i) e_i
 = 
\sum_{1 \leq j \leq k \leq \ell} 
  \gamma_{jk} \left( \init(f_k) e_j - \init(f_j) e_k \right).  
\end{equation}
for some homogeneous $\gamma_{jk}$ in $S$ that
satisfy
\begin{equation}
\label{gamma-degrees}
\deg(g_j) +\deg(f_j)  =
\deg(\gamma_{jk}) +\deg(f_j)+ \deg(f_k) 
\end{equation}
Now apply to \eqref{Koszul-syzygy-expression}
the different $S$-module map $S^\ell \rightarrow S$ that 
sends $e_i \mapsto f_i$, giving
$$
\sum_{i=1}^\ell \init(g_i) f_i
=\sum_{1 \leq j < k \leq \ell} 
          \gamma_{jk}(\init(f_k) f_j - \init(f_j) f_k).
$$
This allows one to rewrite \eqref{kernel-expression} as follows:
$$
\begin{array}{rccclcl}
f&=& g & + &\displaystyle \sum_{i=1}^\ell g_i f_i & & \\
 &=& g & + &\displaystyle\sum_{i=1}^\ell \init(g_i) f_i &+& \displaystyle\sum_{i=1}^\ell (g_i-\init(g_i))f_i \\
 &=& g & + &\displaystyle\sum_{1 \leq j < k \leq \ell} 
          \gamma_{jk}(\init(f_k) f_j - \init(f_j) f_k)
           & + & \displaystyle\sum_{i=1}^\ell (g_i-\init(g_i)) f_i.\\
 &=& g & + &\displaystyle\sum_{1 \leq j < k \leq \ell} 
          \gamma_{jk}\left((\init(f_k)-f_k) f_j - (\init(f_j)-f_j) f_k\right)
           & + & \displaystyle\sum_{i=1}^\ell (g_i-\init(g_i)) f_i\\
\end{array}
$$
in which the last line added $0=(-f_k)f_j -(-f_j)f_k$ for 
$1 \leq j < k \leq \ell$ to the previous line.
Using \eqref{gamma-degrees} and the fact that
$\deg( \init(f_i)-f_i ) <\deg(f_i)$, one finds that this
last expression for $f$ contradicts the minimality of $d_0$
in \eqref{kernel-expression}.
\end{proof}
\section{Proof of Proposition~\ref{extreme-orbits-proposition}}
\label{extreme-orbits-section}

Recall the statement here.

\vskip.1in
\noindent
{\bf Proposition~\ref{extreme-orbits-proposition}.}
{\it
For $W$ an irreducible real reflection group of rank $n$,
and for all choices of $\Theta$ as in \eqref{strong-hsop-setup}, 
one has that
\begin{enumerate}
\item[(i)] the set $V^{\Theta}(0) = \{ 0 \}$ of $0$-dimensional points in $V^{\Theta}$
is the unique $W \times C$-orbit in $V^{\Theta}$ carrying the trivial 
$W \times C$-representation,
\item[(ii)] there exists a $W \times C$-equivariant injection 
$V^{\Theta}(1) \hookrightarrow \Park^{NC}_W$ whose image is precisely the set of 
noncrossing parking functions of the form $[w,X]$ with $\dim(X) = 1$, and
\item[(iii)] the set $V^{\Theta}(n)$ of $n$-dimensional points is the unique $W$-regular
orbit of points in $V^{\Theta}$.
\end{enumerate}
Furthermore, every point in the subsets $V^{\Theta}(0)$, $V^{\Theta}(1)$, and $V^{\Theta}(n)$
of $V^{\Theta}$ is cut out by the ideal $(\Theta - \xx)$ in a reduced fashion, that is, 
with multiplicity one.
}

\vskip.1in
\noindent
The proof of the proposition will show that the
injection $V^{\Theta}(1) \hookrightarrow \Park^{NC}_W$ 
in part (ii) 
can be made more explicit as follows. 
Lemma~\ref{conjugate-to-noncrossing} below will show that
one can choose $W \times C$ orbit representatives 
$\{\beta_i\}_{i=1}^r$ in $V^\Theta(1)$ in such a way that
for each $i=1,2,\ldots,r$, the unique smallest subspace 
$X_i$ in $\LLL$ containing $\beta_i$ is a noncrossing line.
The injection $V^{\Theta}(1) \hookrightarrow \Park^{NC}_W$ 
is then determined by sending 
$\beta_i$ to $[1,X_i]$.

The proof of the three parts (i),(ii),(iii) have differing levels of difficulty
and different flavors.

\subsection{Proof of Proposition~\ref{extreme-orbits-proposition}(i)}

Note that the origin $0$ is contained in $V^\Theta$, so that $V^{\Theta}(0) = \{0\}$.
Also, since $W$ acts irreducibly on $V$, the origin
is the {\it only} $W$-fixed point of $V$, or of $V^\Theta$.
For the multiplicity one assertion,
note that the Jacobian matrix $J$
for $(\Theta-\xx)=(\theta_1-x_1,\ldots,\theta_n-x_n)$
looks like
$$
J=\left[ \frac{\partial (\theta_i - x_i)}{\partial x_j} 
     \right]_{\substack{i=1,2,\ldots,n\\j=1,2,\ldots,n}} \\
  =\left[ \frac{\partial \theta_i}{\partial x_j} 
     \right] - I_{n \times n}
$$
and will be nonsingular when evaluated $\xx=0$:
each $\theta_i$ has degree $h + 1 \geq 2$,  so the
above expression shows $J|_{\xx=0}=-I_{n \times n}$.

This completes the proof of Proposition~\ref{extreme-orbits-proposition}(i).

\subsection{Proof of Proposition~\ref{extreme-orbits-proposition}(iii)}

Observe that a point in $V^{\Theta}$ is in $V^{\Theta}(n)$ if and only if it is $W$-regular.
Therefore, it suffices to show that there exists a unique $W$-regular orbit of points 
in $V^{\Theta}$ and that the points in this orbit have multiplicity one.

We first exhibit a $W$-regular orbit of points inside $V^{\Theta}$.
It is known that any Coxeter element $c$ in $W$ has a $W$-regular eigenvector
$v$ in $V$, spanning a simple $\omega$-eigenspace for $c$,
where $\omega=e^{\frac{2\pi i}{h}}$, as usual;
see, e.g., Humphreys \cite[\S 3.19]{Humphreys}.
Starting with the eigenvector equation for $v$, and
applying the $W$-equivariant map $\Theta: V \longmapsto V$ discussed in the 
Section~\ref{main-conjecture-subsection}, one obtains the following:
$$
\begin{aligned}
c(v) &= \omega v \\
\Theta(c(v)) &= \Theta(\omega v) \\
c \Theta(v) &= \omega^{h+1} \Theta(v)
= \omega \Theta(v)\\
\end{aligned}
$$
where the third equality used the fact that $\Theta$ is homogeneous
of degree $h+1$.  Due to the simplicity of the $\omega$-eigenspace for $c$, one
concludes that $\Theta(v) = \lambda v$ for some $\lambda$ in $\CC$.
Furthermore, $\lambda \neq 0$, else the hsop $(\Theta)$ would vanish on the
line spanned by the nonzero vector $v$.  This allows one to rescale $v$ to
$v_0:= \lambda^{-\frac{1}{h}} v$, giving another $W$-regular $\omega$-eigenvector
for $c$, but which now lies in $V^\Theta$:
$$
\begin{aligned}
\Theta( v_0 ) &= \Theta(  \lambda^{-\frac{1}{h}} v ) \\
              &= \lambda^{-\frac{h+1}{h}} \Theta( v ) \\
              &= \lambda^{-\frac{h+1}{h}} \lambda v 
              =  \lambda^{-\frac{1}{h}} v 
              = v_0.
\end{aligned}
$$
Thus the entire regular $W$-orbit of $v_0$ lies inside the $W$-stable
subvariety $V^\Theta$, as desired.
As $c(v_0)=\omega v_0$, one also knows that $v_0$ has $W \times C$-stabilizer
$\langle (c,c) \rangle$.

To see $v_0$ and its $W$-images are the {\it only}
$W$-regular points in $V^\Theta$, and cut out by
the ideal $(\Theta-\xx)$ with multiplicity one, we rely on
the following, proven below.

\begin{lemma}
\label{Park-has-det-once-lemma}
For an irreducible real reflection group $W$,
there is exactly {\it one} copy of the $W$-alternating character
$\det$ in the $W$-representation $\CC[V]/(\Theta)$.
\end{lemma}

\noindent
By Proposition~\ref{deformation-isomorphism},
this implies $\det$ occurs only once in $\CC[V]/(\Theta-\xx)$.
The idea is to use the copy of $\det$ inside the regular representation of $W$
to exhibit multiple copies of $\det$ in the $\CC[V]/(\Theta-\xx)$
if there were more than one $W$-regular orbit in $V^\Theta$
or if there were a $W$-regular orbit with points of multiplicity at least two.

To make this rigorous, consider the {\it primary decomposition} 
$$
(\Theta-\xx) = \bigcap_{v \in V^\Theta} I_v
$$
where $I_v$ is the $\mm_v$-primary component of $(\Theta-\xx)$
for the maximal ideal 
$$
\mm_v= (x_1-x_1(v), \ldots,x_n-x_n(v))
$$
of $\CC[V]$ corresponding to $v$.  
Because $\CC[V]/(\Theta-\xx)$ is Artinian and Noetherian,
the Chinese Remainder Theorem \cite[\S 2.4]{Eisenbud} gives a ring isomorphism
\begin{equation}
\label{Chinese-remainder-isomorphism}
\CC[V]/(\Theta-\xx) \overset{\varphi}{\longrightarrow} \bigoplus_{v \in V^\Theta} \CC[V]/I_v:=A.
\end{equation}
By definition, $(\Theta-\xx)$ cuts out the point $v$ in $V^\Theta$
with multiplicity $\dim_\CC \CC[V]/I_v$.

Note that the isomorphism $\varphi$ also respects $W$-actions, if one lets $w \in W$
map the summands of $A$ via isomorphisms $w: \CC[V]/I_v \rightarrow  \CC[V]/I_{w(v)}$,
induced from the action of $w$ on $\CC[V]$ (sending $I_v$ isomorphically
to $I_{w(v)}$).  As $\varphi$ is a $W$-isomorphism, 
the $W$-alternating ($\det$-isotypic) component $A^{W,\det}$ is $1$-dimensional.

Now given any $W$-regular vector $v_0$ in $V^\Theta$, 
any nonzero element $f_{v_0}$ in $\CC[V]/I_{v_0}$
can be $W$-antisymmetrized to give a (nonzero) element
$$
\delta_{v_0}:=\sum_{w \in W} \det(w) w(f_{v_0})
$$
lying in the $1$-dimensional space $A^{W,\det}$.

If $v_0, v_0'$ are $W$-regular vectors lying in {\it different} $W$-orbits,
then these two elements $\delta_{v_0}, \delta_{v'_0}$ would be $\CC$-linearly independent
in $A^{W,\det}$, because they are supported in different summands of $A$;  contradiction.

Similarly if the multiplicity $\dim_\CC \CC[V]/I_{v_0} \geq 2$, then one could pick 
$f_{v_0}, f'_{v_0}$ which are $\CC$-linearly independent
in $\CC[V]/I_{v_0}$, and obtain $\delta_{v_0}, \delta'_{v_0}$ that are
$\CC$-linearly independent in $A^{W,\det}$; again, a contradiction.  

This completes the proof of Proposition~\ref{extreme-orbits-proposition}(iii), except for
the Proof of Lemma~\ref{Park-has-det-once-lemma}.

\begin{proof}[Proof of Lemma~\ref{Park-has-det-once-lemma}.]
Note that \eqref{park-graded-W-character} or Proposition~\ref{algebraic-character-prop}
implies that $\CC[V] / (\Theta)$ has $W$-character $\chi(w)=(h+1)^{\dim V^w}$.
Thus the multiplicity of $\det$ in $\CC[V]/(\Theta)$ is
\begin{equation}
\label{det-inner-product}
\langle \det , \chi \rangle_W 
=\frac{1}{|W|} \sum_{w \in W} \det(w) (h+1)^{\dim V^w}.
\end{equation}
On the other hand, well-known results of Orlik and Solomon \cite[\S 1, \S4]{OrlikSolomon}
show that for complex reflection groups,
\begin{equation}
\label{Orlik-Solomon-formula}
\sum_{w \in W} \det(w) t^{\dim V^w}
= \prod_{i=1}^n (t-e_i) 
\end{equation}
where $e_1,\ldots,e_n$ are the {\it exponents} of $W$.
Combining \eqref{det-inner-product} and \eqref{Orlik-Solomon-formula} gives
$$
\langle \det , \chi \rangle_W 
=\frac{1}{|W|} \prod_{i=1}^n (h+1-e_i) 
=\frac{1}{|W|} \prod_{i=1}^n (e_i+1) 
=1
$$
where the second equality uses the {\it exponent duality}
$e_i + e_{n-1-i}=h$ known for real reflection groups
\cite[\S 5]{OrlikSolomon}, and the
last equality uses the Shephard and Todd formula $|W|=\prod_{i=1}^n(e_i+1)$;
cf. the proof of \cite[Cor. 7.4.1]{Haiman}.
\end{proof}

\subsection{Some noncrossing geometry}

Before delving into the proof of 
Proposition~\ref{extreme-orbits-proposition}(ii) in the next subsection,
we collect here three results (Lemmas~\ref{conjugate-to-noncrossing}, 
~\ref{W-C-orbit-lemma},
~\ref{one-dimensional-noncrossings-not-perpendicular} below)
on noncrossing subspaces, particularly noncrossing lines,
needed in the proof, and of possible interest in their own right.

\begin{lemma}
\label{conjugate-to-noncrossing}
Every flat in $\LLL$ is a $W$-translate of a noncrossing flat, so the map
$$
\begin{array}{rcl}
C\backslash NC(W) &\overset{\phi}{\longrightarrow} &W\backslash \LLL \\
C.X & \longmapsto & W.X
\end{array}
$$
is a well-defined surjection.
\end{lemma}

\begin{proof}
Let $X$ in $\LLL$ be a flat and let $W_X=\{w\in W: wx=x\text{ for all } x\in X\}$ be the isotropy subgroup of $W$ corresponding to $X$. Then $W_X$ is a parabolic subgroup of $W$; i.e. there exists $w\in W$ and $I\subseteq S$ such that $W_X=wW_Iw^{-1}$. It follows that $X=V^{W_I}$. Then by Carter \cite[Lemma 2]{Carter} one has that $V^{W_I}=V^x$, where $x$ is the product of the generators $I$ taken in any order. Note that this $x\in W$ is noncrossing with respect to the Coxeter element $c=xy$, where $y$ is the product of the generators $S-I$ in any order. Since all Coxeter elements are conjugate in $W$, we are done.
\end{proof}

One can be much more precise about the map $\phi$ 
in Lemma~\ref{conjugate-to-noncrossing} when one restricts to $1$-dimensional 
flats $X$, that is, when $X$ is a line.  For the remainder of this section, a {\it line} 
will always mean a $1$-dimensional flat in the intersection lattice $\LLL$ for $W$.  

\begin{lemma}
\label{W-C-orbit-lemma}
Let $X$ in $\LLL$ be a noncrossing line.  Then the set of noncrossing lines  
in the $W$-orbit of $X$ is precisely the $C$-orbit of $X$, or in other words, 
the surjection $\phi$ from Lemma~\ref{conjugate-to-noncrossing} restricts to a bijection
$$
C\backslash \{\text{ noncrossing lines }\} \overset{\phi}{\longrightarrow} 
W\backslash \{\text{ lines }\}.
$$
\end{lemma}

\begin{proof}
The map $\phi$ occurs as the top horizontal map in the following larger diagram of maps,
explained below, involving the orbits of various groups acting on  
the sets of lines, noncrossing lines,
half-lines, reflections $T$, roots $\Phi$, and simple roots $\Pi$:
\begin{equation}
\label{large-diagram-of-bijections-surjections}
\begin{matrix}
C \backslash \{\text{ noncrossing lines }\}
  & \overset{\phi}{\longrightarrow} & 
    W \backslash \{\text{ lines } \} \\
 & & \\
{f_1}\uparrow & & \uparrow{g_1} \\
 & & \\
C \backslash T 
   & 
    & \langle W,-1  \rangle \backslash \{ \text{ half-lines } \} \\
 & & \\
{f_2}\uparrow & & \uparrow{g_2} & & \\
 & & \\
\langle C, -1 \rangle \backslash \Phi 
  &\overset{\psi}{\longleftarrow} 
    & \langle -w_0 \rangle \backslash \Pi 
\end{matrix}
\end{equation}
We have seen in Lemma~\ref{conjugate-to-noncrossing} that
$\phi$ is well-defined and surjective.  Our strategy will be to show that 
\begin{enumerate}
\item[$\bullet$]
each of the vertical maps $f_1,f_2,g_1,g_2$  is bijective, and 
\item[$\bullet$]
the other horizontal map $\psi$ is surjective.  
\end{enumerate}
Cardinality considerations will then imply that both $\phi, \psi$ are also bijective.  

\vskip.1in
\noindent
{\sf The bijection $f_1$.}
The map $f_1$ is induced by the map from the set of {\it reflections} $T$ in $W$ to the
set of noncrossing lines
$$
\begin{array}{rcl}
T & \longrightarrow &\{ \text{ noncrossing lines }\} \\
t & \longmapsto &V^{ct} \\
\end{array}
$$
which is bijective by work of Bessis \cite[\S 2.4]{Bessis}, and 
equivariant for the action of $C$ via conjugation on $T$ and via translation on 
noncrossing lines:
$$
c^d t c^{-d} \longmapsto 
V^{c \cdot c^{d} t c^{-d}} 
= V^{c^{d} \cdot ct \cdot c^{-d}} 
= c^d(V^{ct}).
$$

\vskip.1in
\noindent
{\sf The bijection $f_2$.}
The map $f_2$ is induced from map $\Phi \longrightarrow T$ that sends a root 
$\alpha$ to the reflection $t_\alpha$ through the hyperplane perpendicular to $\alpha$.
The latter map is two-to-one, with fibers $\{\pm \alpha\}$, inducing
a bijection $\langle -1 \rangle \backslash \Phi \longrightarrow T$.
It is also equivariant for the $W$-action as well as the $C$-action by restriction:
$$
c(\pm \alpha) \longmapsto t_{\pm c(\alpha)} = c t_{\alpha} c^{-1}.
$$
Hence $f_2$ is a bijection.

\vskip.1in
\noindent
{\sf The bijection $g_1$.}
A {\it half-line} is either of the two rays emanating from the origin
inside any of the lines for $W$.  
There is a two-to-one map $\{ \text{ half-lines } \} \longrightarrow \{ \text{ lines }\}$
that sends a half-line to the line that it spans, inducing the bijection $g_1$.

\vskip.1in
\noindent
{\sf The bijection $g_2$.}
Normalize the set of simple roots $\Pi=\{\alpha_1,\ldots,\alpha_n\}$,
to have unit length, and denote by 
$\{\delta_1,\ldots,\delta_n\}$ their dual basis, defined by 
\begin{equation}
\label{fundamental-dominant-weight-definition}
(\alpha_i,\delta_j) = \begin{cases}
1 & \text{ if }i=j,\\
0 & \text{ if }i \neq j.
\end{cases}
\end{equation}
Then the map
$
\Pi  \longrightarrow  \{ \text{ half-lines } \}
$ 
sending $\alpha_i \longmapsto \RR_+ \delta_i
$
gives a well-defined map 
$$
\langle -w_0 \rangle \backslash \Pi \overset{g_2}{\longrightarrow} 
\langle W,-1 \rangle \backslash \{ \text{ half-lines } \}
$$
where $w_0$ is the unique {\it longest element} of $W$,
because $-w_0$ lies in $\langle W,-1 \rangle$.  Note that $-w_0$ always takes
a simple root $\alpha_i$ to another simple root $\alpha_j$, corresponding
to the fact that conjugation by $w_0$ sends the simple reflection $s_i$ to
another simple reflection $w_0 s_i w_0 = s_j$, via a 
{\it diagram automorphism}\footnote{Although not needed in the sequel, we 
remark that the diagram \eqref{large-diagram-of-bijections-surjections}
will also show the following:
both sets of orbits 
$
C\backslash \{\text{ noncrossing lines }\}
$
and
$
W\backslash \{\text{ lines }\}
$
biject with the orbits of the simple reflections $S$ 
under the Coxeter diagram automorphism
$s \mapsto w_0 s w_0$.  
Thus their cardinality is at most $n=|S|$, with equality 
if and only if $-1$ lies in $W$.};
see \cite[\S 2.3]{BBCoxeter}.

The usual geometry of {\it Tits cones} and 
{\it Coxeter complexes} \cite[\S 1.15, 5.13]{Humphreys} says 
\begin{enumerate}
\item[$\bullet$]
each half-line is in the $W$-orbit of {\it exactly one} half-line
$\RR_+ \delta_1,\ldots,\RR_+ \delta_n$, and 
\item[$\bullet$]
these are the only half-lines lying within the
{\it dominant cone} 
$$
F:= \{ x \in V: (\alpha_i,x) \geq 0 \text{ for }i=1,2,\ldots, n\}
=\RR_+ \delta_1 + \cdots +\RR_+ \delta_n
$$
\end{enumerate}
In particular, this shows that $g_2$ is surjective.  
It remains to show that $g_2$ is also injective,
that is, assuming that $\RR_+ \delta_j$ lies in
$\langle W, -1 \rangle. \RR_+ \delta_i$,
we must show that this forces $\alpha_j \in \langle -w_0 \rangle. \alpha_i$.
If $\RR_+ \delta_j = w \left( \RR_+ \delta_i \right)$ for some $w$ in $W$, then
the above-mentioned Tits cone geometry implies $i = j$ and hence $\alpha_j = \alpha_i$. 
So assume without loss of generality that 
$\RR_+ \delta_j= -w \left(\RR_+ \delta_i\right)$ 
for some $w \in W$.  Note that both $-1$ and $w_0$ send the dominant cone $F$
to its negative $-F$, and therefore $-w_0(F)=F$.  Hence
$-w_0 \left( \RR_+ \delta_j \right) = w_0 w \left( \RR_+ \delta_i\right)$ 
lies in $F$.  The above-mentioned Tits cone geometry
forces $-w_0 \left( \RR_+  \delta_j \right) = \RR_+ \delta_i$.
Since $-w_0$ is an isometry, this forces $-w_0 \delta_j = \delta_i$, and 
also $-w_0 \alpha_j = \alpha_i$, as desired.

\vskip.1in
\noindent
{\sf The surjection $\psi$.}
We will show that the inclusion $\Pi \hookrightarrow \Phi$
induces a well-defined surjection 
$\langle -w_0 \rangle \backslash \Pi 
\overset{\psi}{\longrightarrow} 
\langle C, -1 \rangle \backslash \Phi$,
whenever the Coxeter element $c$ is chosen to be {\it bipartite},
that is, one has properly 2-colored the Coxeter diagram, 
giving a decomposition of the simple reflections into two mutually commuting sets
$$
S=\{s_1,s_2,\ldots,s_m\} \sqcup \{s_{m+1},s_{m+2},\ldots,s_n\}
$$
so that
\begin{equation}
\label{bipartite-factorization}
c=c_+ c_- 
=\underbrace{s_1 s_2 \ldots s_m}_{c_+} \cdot 
   \underbrace{s_{m+1} s_{m+2} \ldots s_n}_{c_-}.
\end{equation}
It is well-known 
(see \cite[Chap. V \S 6, Exer. 2]{Bourbaki}, \cite[\S 3.19 Exer. 2]{Humphreys}) 
that in this setting, the longest element $w_0$ can be expressed as
$$
w_0 = \begin{cases}
c^{\frac{h}{2}} &\text{ if }h \text{ is even,}\\
c^{\frac{h-1}{2}} c_+& \text{ if }h \text{ is odd.}
\end{cases}
$$

We first show that $\psi$ is well defined,
that is, for any $\alpha_i$ in $\Pi$, the root
$-w_0 (\alpha_i)$ lies in $\langle C, -1 \rangle.\alpha_i$.
If $h$ is even, then $w_0 = c^{\frac{h}{2}}$ and this is obvious,
so assume without loss of generality that $h$ is odd. Then 
$-w_0 (\alpha_i) = - c^{\frac{h-1}{2}} c_+ (\alpha_i)$.  
If the subscript $i$ in $\alpha_i$ lies in $\{1,2,\ldots,m\}$, so that
$s_i$ appears in $c_+$, then $c_+ (\alpha_i) = -\alpha_i$,
and hence $-w_0 (\alpha_i) = c^{\frac{h-1}{2}} (\alpha_i)$ lies in 
$\langle C, -1 \rangle. \alpha_i$.  
On the other hand, if the subscript $i$ lies in $\{m+1,m+2,\ldots,n\}$
so that $s_i$ appears in $c_-$, then
$c_- (\alpha_i) = -\alpha_i$ and 
$$
-w_0 (\alpha_i) = -  c^{\frac{h-1}{2}} c_+ (\alpha_i) 
              =  c^{\frac{h-1}{2}} c_+ c_- (\alpha_i) 
              = c^{\frac{h+1}{2}} (\alpha_i)
$$
which lies in $\langle C, -1 \rangle.\alpha_i$.  
This proves that $\psi$ is well defined.

To show that $\psi$ is surjective, it suffices to show that 
every root in $\Phi$ is $C$-conjugate to plus or minus a simple root.
Starting with the factorization $c=c_1 c_2 \ldots c_n$ it is known that
one can define a sequence of roots
$$
\theta_i := s_n s_{n-1} \cdots s_{i+1}(\alpha_i)
$$
for $i=1,2,\ldots,n$, giving a complete set of representatives
for the $C$-orbits on $\Phi$; see \cite[Chap VI, \S1, No. 11, Prop. 33]{Bourbaki}.
Thus it suffices to check that for each $i = 1, 2, \dots, n$, 
one has $\theta_i$ lying in $\pm C.\alpha_i$.

If $i$ lies in $\{m+1,m+2,\ldots,n\}$, so $s_i$ appears in $c_-$, 
then $s_n, s_{n-1}, \dots, s_{i+1}$ all fix $\alpha_i$.  Therefore 
$
\theta_i = s_n s_{n-1} \cdots s_{i+1} \alpha_i=\alpha_i
$ 
and we are done.
If $i$ lies in $\{1,2,\ldots, m\}$, so $s_i$ appears in $c_+$, then
$s_1,s_2,\ldots,s_{i-1}$ all fix $\alpha_i$. Therefore
$$
c (\theta_i) = s_1 s_2 \cdots s_n \cdot s_n s_{n-1} \cdots s_{i+1} (\alpha_i) 
= s_1 s_2 \cdots s_i (\alpha_i) 
= - s_1 s_2 \cdots s_{i-1} (\alpha_i) 
= -\alpha_i.
$$
Hence $\theta_i=-c^{-1}(\alpha)$ lies in $-C.\alpha_i$, as desired.
\end{proof}

We remark that Lemma~\ref{W-C-orbit-lemma} was originally verified  
case-by-case along these lines: 
\begin{enumerate}
\item[$\bullet$]
It is trivial to check it in the dihedral types $I_2(m)$, 
\item[$\bullet$]
it can be checked by computer in the exceptional types, and
\item[$\bullet$]
it is easy to argue directly in
the classical types $A,B/C,D$.  
\end{enumerate}
For example, in type $A_{n-1}$, the $W$-orbits of noncrossing
lines correspond to set partitions with precisely two blocks of fixed sizes $k$ and $n-k$.  
The set of noncrossing partitions with block sizes $k$ and $n-k$ forms a single $C$-orbit.
The analysis in the other classical types $B/C$ and $D$ is similar.  

\vskip.1in

The last lemma  needed for the proof of 
Proposition~\ref{extreme-orbits-proposition}(ii)
pertains to the {\it Coxeter plane} $P \subset V$ 
corresponding to a bipartite Coxeter element $c = c_+ c_-$.
This $P$ is a $2$-dimensional subspace of $V$, acted on by the dihedral group 
$\langle c_+, c_- \rangle$.  The subgroup of this dihedral group generated by $c$
acts on $P$ by $h$-fold rotation; see \cite[Chapter 3]{Humphreys}.  

\begin{lemma}
\label{one-dimensional-noncrossings-not-perpendicular}
A noncrossing line cannot be 
orthogonal to the Coxeter plane $P$.
\end{lemma}
\begin{proof}[Proof of Lemma~\ref{one-dimensional-noncrossings-not-perpendicular}]
Suppose that $X$ in $\LLL$ is a noncrossing line, spanned over $\RR$ by a vector $v$, and that $v$ is perpendicular to $P$.  Since $X$ is noncrossing, there exists a reflection
$t$ in $T$ such that $X = V^{tc}$.  
Let $H$ be the reflecting hyperplane for $t$.

One has that $t(v) = t(tc(v)) = c(v)$, and therefore $v \neq t(v)$ (else
$V^c \neq \{ 0 \}$).  Moreover, as $c(P) = P$ and $c$ is an
orthogonal transformation, the vector
$t(v) = c(v)$ is also perpendicular to $P$.  It follows that the difference
$v - t(v)$ is perpendicular to $P$ and nonzero.  The vector $v - t(v)$ is perpendicular 
to $H$, which shows that $P \subseteq H$.  But by the discussion in
\cite[Section 3.18]{Humphreys}, no reflecting hyperplane contains the Coxeter plane $P$,
so this is a contradiction.
\end{proof}

\subsection{Proof of Proposition~\ref{extreme-orbits-proposition}(ii)}
\label{noncrossing-lines-proof-section}

We will assume, without loss of generality, that our Coxeter element $c$ is {\sf bipartite}. 
To prove (ii), we begin by analyzing the set $V^{\Theta}(1)$.  Consider the 
map $\Theta: V \rightarrow V$ discussed in Section~\ref{main-conjecture-subsection}.
Because any flat $X \in \LLL$ has the form $X = V^w$ for some $w \in W$ and 
$\Theta$ is $W$-equivariant, one has that $\Theta(X) \subseteq X$ for all $X \in \LLL$ and
the restriction $\Theta|_X$ is a well-defined map $X \rightarrow X$ given by
homogeneous polynomials of degree $h+1$.

If the flat $X$ is $1$-dimensional, choosing a coordinate $x$ for $X$,
so that its coordinate ring is $\CC[x]$,
one can express $(\Theta-\xx)|_X=\alpha x^{h+1} - x$ for some
$\alpha \in \CC^{\times}$.
Its zero set $X^\Theta$ consists of the $h+1$ points
\begin{equation*}
\{ 0 \} \sqcup 
\underbrace{ 
 \{ \alpha^{-\frac{1}{h}}, \omega \alpha^{-\frac{1}{h}}, \dots, \omega^{h-1} \alpha^{-\frac{1}{h}} \} 
}_{{V^\Theta}(1) \cap X},
\end{equation*}
where $\omega$ is a primitive $h^{th}$ root of unity, and each point occurs with multiplicity
one.  Recall that the cyclic group $C$ acts via scaling by $\omega$, and so its
action on these $h+1$ points decomposes into the two orbits $\{0\}, V^\Theta(1) \cap X$
shown above.
Since $X$ is $1$-dimensional and $W$ is a real reflection group, it follows that any
element $w$ in $W$ which stabilizes $X$ must act as $1$ or $-1$ on $X$.  Since 
$C$ acts by scaling by $\omega$, the $W \times C$-stabilizer of 
the typical point $\alpha^{-\frac{1}{h}}$ in the orbit $V^{\Theta}(1) \cap X$ equals
\begin{equation}
\label{line-WxC-stabilizer}
\left\{ 
(w,c^d) \,:\, \begin{matrix} \text{$w$ acts as $1$ on $X$ and $c^d = 1$, or}\\
                               \text{$w$ acts as $-1$ on $X$ and $d \equiv  \frac{h}{2} \bmod{h}$}
                 \end{matrix}.
\right\}
\end{equation}

Suppose in addition that the flat $X$ is noncrossing.  
It remains to show that the $W \times C$-stabilizer of 
the noncrossing parking function $[1, X] \in \Park^{NC}_W$ equals the same
subgroup as in \eqref{line-WxC-stabilizer}.  
It follows directly from the definition of $\Park^{NC}_W$ that this stabilizer equals
\begin{equation*}
\{ (w, c^d) \,:\, \text{$c^d$ stabilizes $X$ and $c^d w^{-1}$ acts as $1$ on $X$} \}.
\end{equation*}
Let $\pi : V \twoheadrightarrow P$ be orthogonal projection onto the Coxeter plane.  
If $c^d$ stabilizes $X$ for some $d \geq 0$, since $c^d$ acts orthogonally, one has that
$c^d$ stabilizes the subspace $\pi(X)$ of $P$, as well, and acts by the same
scalar on both $X$ and $\pi(X)$.  
By Lemma~\ref{one-dimensional-noncrossings-not-perpendicular}, this subspace
$\pi(X)$ is a line inside $P$, that is, $\pi(X) \neq 0$.   
Since $c$ acts on $P$ by $h$-fold rotation, this forces
either
\begin{enumerate}
\item[$\bullet$] $d \equiv 0 \bmod{h}$, so that $c^d$ fixes the line $\pi(X)$ (and $X$) 
pointwise, or 
\item[$\bullet$] $d \equiv \frac{h}{2} \bmod{h}$, so that 
$c^d$ acts as $-1$ on the line $\pi(X)$ (and $X$).
\end{enumerate}
It follows that the $W \times C$-stabilizer of 
$[1,X] \in \Park^{NC}_W$ is as described in \eqref{line-WxC-stabilizer}.

Let $\beta_1, \dots, \beta_r$ be a complete set of representatives for the $W \times C$-orbits
in $V^{\Theta}(1)$.  By Lemma~\ref{conjugate-to-noncrossing}, without loss of generality
we can assume that $\beta_i \in X_i$ for all $i$, where $X_1, \dots, X_r \in \LLL$ are 
noncrossing and are a complete set of representatives for the $W$-orbits of lines
in $\LLL$.
Since the $W \times C$-stabilizers of $\beta_i$ and $[1, X_i]$ are equal
for all $i$, we get a well defined $W \times C$-equivariant map $V^{\Theta}(1) \rightarrow 
\Park^{NC}_W$ induced by $\beta_i \mapsto [1, X_i]$.  

We claim that the map $V^{\Theta}(1) \rightarrow \Park^{NC}_W$ described above is
a bijection
onto $\Park^{NC}_W(1) := \{ [w, X] \in \Park^{NC}_W \,:\, \dim(X) = 1 \}$.  By stabilizer equality, this 
map restricts to a bijection from the $W \times C$-orbit of $\beta_i$ to the 
$W \times C$-orbit of $[1, X_i]$ for all $i$.  Moreover, the image of the 
$W \times C$-orbit of $\beta_i$ is
$\{ [w, c^d X_i] \,:\, \text{$w \in W$, $c^d \in C$} \}$.  Since the $W \times C$-orbit of $\beta_i$ inside
$V^{\Theta}(1)$ is $V^{\Theta}(1) \cap \{wX_i \,:\, w \in W \}$ and the $W$-orbit of 
$X_i$ contains the $C$-orbit of $X_i$, the sets $\{ [w, c^d X_i] \,:\, \text{$w \in W$, $c^d \in C$} \}$
are disjoint for distinct flats $X_i$ and the map
$V^{\Theta}(1) \rightarrow \Park^{NC}_W(1)$ is injective.  
Lemma~\ref{W-C-orbit-lemma} shows that this map is also surjective.

This completes the proof of Proposition~\ref{extreme-orbits-proposition}(ii).

\section{Type $B/C$}
\label{type-BC-proof-section}

Before proceeding with the proof, we review from \cite{Reiner} 
how to visualize the elements of $NC(W)$ for $W$ of type $B/C$
and extend this to similarly visualize the elements of $\Park^{NC}_W$,
analogous to the description for type $A$ given 
in Example~\ref{type-A-pictorial-example}.

\subsection{Visualizing type $B/C$}

The Weyl group $W$ of type $B_n$ or $C_n$ is the {\it hyperoctahedral group}
$\symm^{\pm}_n$ of $n \times n$ signed permutation matrices, that is, matrices
having one nonzero entry equal to $\pm 1$ in each row and column.
These signed permutation matrices act on $V = \CC^n$, but
we will also often think of $w$ in $W$ as a permutation of the set 
$$
\pm [n]:=\{1,2,\dots,n,-1,-2,\dots,-n\}
$$
satisfying $w(-i) = - w(i)$.
The reflections $T \subseteq W=\symm^{\pm}_n$ are tabulated here: 
\medskip
\begin{center}
\begin{tabular}{|c|c|}\hline
reflection & reflecting hyperplane\\\hline\hline
$t^+_{ij}=(i,j)(-i,-j)$ & $x_i=+x_j$ \\\hline
$t^-_{ij}=(i,-j)(-i,j)$ & $x_i=-x_j$ \\\hline
$t^-_{ii}=(+i,-i)$ & $x_i=0(=-x_i)$ \\\hline
\end{tabular}
\end{center}
\medskip
\noindent
A typical intersection flat $X$ has several blocks of coordinates
set equal to each other or their negatives, and possibly one block
(called the {\it zero block}, if present) where the coordinates are
all zero.  This corresponds to a set partition
$\pi=\{B_1,\ldots,B_\ell\}$ of the set $\pm [n]$ with  
$\pi=-\pi$, in the sense that for each block $B_i$ there exists some block $B_j=-B_i$,
and at most one zero block $B_k=-B_k$.
For example, when $n=9$, the flat $X$ defined by the equations
$$
\begin{aligned}
x_1&=-x_6=-x_9\\
x_2&=x_5=0\\
x_3&=x_4\\
x_7&=x_8
\end{aligned}
$$ 
has zero coordinates $x_2,x_5$, corresponding to the
zero block $\{\pm 2,\pm 5\}$ in the associated set partition:
\begin{equation}
\label{type-B9-set-partition-example}
\begin{aligned}
\pi= &\{\{+1,-6,-9\},\{-1,+6,+9\},\{\pm 2,\pm 5\}, \\
      &\quad  \{+3,+4\},\{-3,-4\},\{+7,+8\},\{-7,-8\}\}.
\end{aligned}
\end{equation}

One can take as Coxeter generators for $W$ 
the set $S = \{s_0, s_1, . . . , s_{n-1} \}$,
where 
$$
\begin{aligned}
s_i&=t^+_{i,i+1}=(i,i+1) \text{ for }i=1,2,\ldots,n-1,\text{ and }\\
s_0&=t^-_{ii}=(-1,+1).
\end{aligned}
$$
The Coxeter element $c=s_0s_1 \cdots s_{n-1}$
acts as the $2n$-cycle 
$$
c=(+1,+2,\dots,+n,-1,-2,\dots,-n)
$$ 
on $\pm [n]$, and hence has order $h=2n$.  Drawing these values in
the $2n$-cycle clockwise around a circle, 
one can depict set partitions of $\pm [n]$ as before, by drawing
the polygonal convex hulls of their blocks.   The flats $X$ in $\LLL$ 
have the extra property that they are {\it centrally symmetric},
and at most one block is sent to itself by the central symmetry.
Furthermore a flat $X$ in $\LLL$ lies in $NC(W)$ if and only if its 
blocks are noncrossing.  
For example, the element $X$ or partition $\pi$ of $\pm [9]$ from
\eqref{type-B9-set-partition-example} is noncrossing, and is
depicted here:

\begin{center}
\includegraphics[scale=.5]{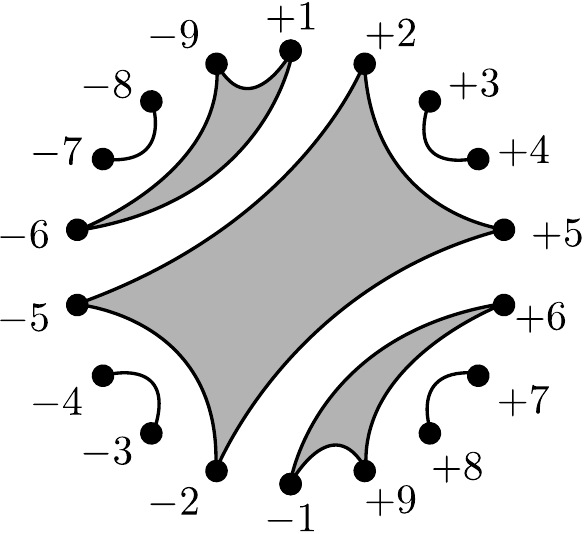}
\end{center}

As in type $A$, such a noncrossing partition $\pi$ or flat $X$ bijects
with the signed permutation $w$ whose associated permutation of 
$\pm [n]$ is obtained by orienting the edges of each block $B_i$ counterclockwise.
One can check that the conjugation action of $C=\langle c \rangle$
on $NC(W)=[1,c]_T$ again corresponds to clockwise rotation 
through $\frac{2 \pi}{h}=\frac{2 \pi}{2n}$ in the picture.

Also as in type $A$ (see Example~\ref{type-A-pictorial-example}),
when $W=\symm^{\pm}_n$ is of type $B_n/C_n$, 
an element $[w,X]$ of $\Park^{NC}_W$ is an 
equivalence class where $X$ is a noncrossing
flat, corresponding to some noncrossing partition $\pi$ of $\pm [n]$,
say with nonzero blocks $\{B_1,-B_1,\ldots,B_\ell,-B_\ell\}$,
and $w$ is a signed permutation in $\symm^{\pm}_n$ considered only up to its
coset $wW_X$ for the parabolic subgroup 
$$
W_X = \symm_{B_1} \times \cdots \times \symm_{B_\ell} \times \symm^{\pm}_{B_0}.
$$
Here $\symm_{B_i}$ is isomorphic to a symmetric group on $|B_i|$ letters
that permutes the elements of $B_i$, and 
$\symm^{\pm}_{B_0}$ is a hyperoctahedral group acting on the coordinates
in the zero block $B_0$, if present.
One can think of the coset $wW_X$ as representing a
function $w$ assigning to each block $B_i$ the (unordered) subset 
$w(B_i):=\{w(j)\}_{j \in B_i} \subseteq \pm [n]$.
It is convenient to visualize $[w,X]$ as labeling the blocks $B_i$
by the sets $w(B_i)$ in the noncrossing partition diagram for $\pi$.

For example, when $n=9$, consider the element $[w,X]$ of $\Park^{NC}_W$
in which $X$ is represented by the partition $\pi$
in \eqref{type-B9-set-partition-example}, with
\begin{equation}
\label{typeB-parking-function-example}
\begin{aligned}
W_X&=\symm_{\{+1,-6,-9\}} \times \symm_{\{+3,+4\}} \times \symm_{\{+7,+8\}} \times\symm^{\pm}_{\{2,5\}} \\
wW_X&=
\left(
\begin{matrix}
+1 & +2 & +3 & +4 & +5 & +6 & +7 & +8 & +9 \\
-6 & -3 & +5 & -9 & +2 & -8 & -1 & -4 & +7  
\end{matrix}
\right)W_X \\
&=
\left(
\begin{array}{ccc|cc|cc|cc}
+1 & -6 & -9 & +7 &+8 & +3 & +4 & +2 & +5 \\
-6 & +8 & -7 & -1 & -4 & +5 & -9 & -3 & +2
\end{array}
\right)W_X
\end{aligned}.
\end{equation}
This $[w,X]$ in $\Park^{NC}_W$ is depicted here:

\begin{center}
\includegraphics[scale=.5]{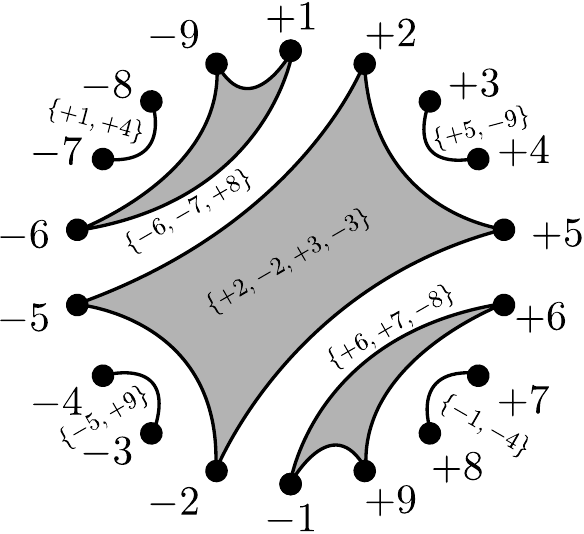}
\end{center}

The $W \times C$-action on $\Park^{NC}_W$ has
$W=\symm^{\pm}_n$ permuting the values within the sets that
label the blocks of the partition $\pi$, and
has the generator $c$ for $C$ rotating
the picture $\frac{2\pi}{2n}$ clockwise.

\subsection{Proof of Main Conjecture (intermediate version) 
in Type B/C}

When the hyperoctahedral group $W=\symm^\pm_n$, acts on $V = \CC^n$,
with $x_1, \dots, x_n$ the standard coordinate functionals in $V^*$,
we choose 
$$
(\Theta)=(\theta_1,\ldots,\theta_n):=(x_1^{2n+1}, \dots, x_n^{2n+1}).
$$ 
It is easily seen that this is an hsop of 
degree $h+1$ in $\CC[V]$, and that the map
$\Theta$ sending $x_i$ to $\theta_i$ is 
$W$-equivariant, so that \eqref{strong-hsop-setup} is satisfied.
In this setting, the subvariety $V^\Theta$ cut out by the ideal
$$
\begin{aligned}
(\Theta-\xx)&=(x_1^{2n+1} - x_1, \dots, x_n^{2n+1} - x_n) \\
            &=(x_1(x_1^{2n} - 1), \dots, x_n(x_n^{2n} - 1)) \\
\end{aligned}
$$ 
consists exactly of the points 
$(v_1, \dots, v_n) \in \CC^n$ such that each 
$v_i$ is either $0$ or a $2n^{th}$ root of unity,
that is, a power of $\omega := e^{\frac{2 \pi i}{2n}}$.
We now describe a $W \times C$-equivariant bijection 
$$
\begin{array}{rcl}
\Park^{NC}_W & \overset{f}{\longrightarrow} & V^\Theta\\
{} [w,X] &\longmapsto &v
\end{array}
$$
followed by a description of its inverse bijection.
\vskip.1in
\noindent
{\sf The forward bijection $f$.}

Given $[w,X]$, one starts with
the noncrossing partition $\pi$ of $\pm [n]$ corresponding
to $X$, and re-encodes it via a bijection from \cite[Proposition 6]{Reiner}
as a periodic parenthesization of the doubly infinite string
\begin{equation}
\label{infinite-cyclic-order}
\cdots -(n-1), -n, +1, +2, \cdots, +(n-1), +n, -1, -2, \cdots, -n, 
+1, +2, \dots
\end{equation}
For example, when $[w,X]$ and $\pi$ are
as in \eqref{typeB-parking-function-example},
this parenthesization is
\begin{equation}
\label{type-B-parenthesization-example}
\cdots -8) -9
+1) +2 (+3 +4) +5 (+6 (+7 +8) +9
-1) -2 (-3 -4) -5 (-6 (-7 -8) -9
+1) \cdots
\end{equation}
The idea of the encoding from \cite{Reiner}
is that each nonzero block $B$ of $\pi$ corresponds to a pair of parentheses,
having its left parenthesis appearing just to the left of some value $j$ in $\pm [n]$
that we will call the {\it opener} of block $B$.
The innermost
parentheses pairs enclose nonzero blocks that nest no other blocks in $\pi$,
whose elements can then be removed, and then one continues to
pair off the nonzero blocks enclosed with parentheses, inductively. 

The map $f$ then sends $[w,X]$ to the vector $v$ whose
$k^{th}$ coordinate $v_k$ depends upon the unique block $B_{i(k)}$ of $\pi$
for which $+k$ lies in $w(B_{i(k)})$:  
\begin{enumerate}
\item[$\bullet$]
if $B_{i(k)}$ is the zero block of $\pi$, then $v_k=0$, 
\item[$\bullet$]
if $B_{i(k)}$ is a nonzero block, with positive opener $+j$,
then $v_k=+\omega^j$, 
\item[$\bullet$]
if $B_{i(k)}$ is a nonzero block, with negative opener $-j$,
then $v_k=-\omega^j$.
\end{enumerate}
Continuing the example of $[w,X]$ and $\pi$ from
\eqref{typeB-parking-function-example}, one determines the coordinates
of $v=f([w,X])$ as follows.  
The zero block $B_0=\{\pm 2,\pm 5\}$ is labeled $w(B_0)=\{\pm 2,\pm 3\}$,
so $v_2=v_3=0$.  The nonzero blocks $B_i$ have openers and images 
$w(B_i)$ tabulated here
\begin{center}
\begin{tabular}{|c|c|c|}\hline
nonzero block $B_i$ & opener & image $w(B_i)$ \\\hline\hline
\{+1,-6,-9\} & -6 & \{-6,+8,-7\} \\\hline
\{-1,+6,+9\} & +6 & \{+6,-8,+7\} \\\hline
\{+3,+4\} & +3 & \{+5,-9\} \\\hline
\{-3,-4\} & -3 & \{-5,+9\} \\\hline
\{+7,+8\} & +7 & \{-1,-4\} \\\hline
\{-7,-8\} & -7 & \{+1,+4\} \\\hline
\end{tabular}
\end{center}
and from this one concludes that
\begin{equation}
\label{f-image-example}
\begin{array}{rcccccccccl}
v=(&v_1,      &v_2,&v_3,&v_4,     &v_5,      &v_6,      &v_7,      &v_8,      &v_9        &) \\
  (&-\omega^7,&0,  &0   &-\omega^7,&+\omega^3,&+\omega^6,&+\omega^6,&-\omega^6,&-\omega^3&). \\
\end{array}
\end{equation}
We leave it to the reader to check that $f$ 
is $W \times C$-equivariant. 

\vskip.1in
\noindent
{\sf The inverse bijection $f^{-1}$.}

Given $v$ in $V^\Theta$, one can first determine the noncrossing
partition $\pi$ of $\pm [n]$, via  its infinite parenthesization,
as follows.  There will be left parentheses located
just to the left of both $+j, -j$ if and only if either
of $\pm \omega^j$ occurs at least once among the coordinates $v_k$.  
Furthermore, the sizes of the blocks having openers 
$+j$ and $-j$ are both equal to the multiplicity $m_j$
with which the two values $\pm \omega^j$ 
occur among the coordinates $v_k$.  Given any such sequence
of multiplicities $(m_1,\ldots,m_n)$ having $\sum_{j=1}^n m_j\leq n$,
proceeding from the smallest $m_j$ to the largest, one
can recover the location of the right parentheses
closing the block opened by $+j$, for each $j$.   The zero block of 
$\pi$ then contains the unparenthesized letters left in the infinite
string. 

For example, given $n=9$ and the vector $v$ in $V^\Theta$ from
\eqref{f-image-example}, then $v$ has multiplicities 
$$
\begin{aligned}
m_3&=2\\
m_6&=3\\
m_7&=2\\
m_1&=m_2=m_4=m_5=m_8=m_9=0.
\end{aligned}
$$
One concludes that the infinite parenthesization has openers at
$\pm 3,\pm 6, \pm 7$, enclosing blocks of sizes $2,3,2$, respectively.
From this one can recover the parenthesization in 
\eqref{type-B-parenthesization-example} by first closing
off the blocks of size $2$ around $(+3, +4), (-3,-4)$, and the
blocks of size $2$ around $(+7,+8), (-7,-8)$.  Then one removes
the values $\pm 3,\pm 4,\pm 7, \pm 8$ from consideration, and
next closes off the blocks of size $3$ opened by $\pm 6$, 
namely $(+6 +9 -1), (-6 -9 +1)$.

Once one has recovered the parenthesization, and hence $\pi$, 
one determines the set $w(B_i)$ labeling the block $B_i$ opened by $+j$ as follows: 
$\pm k$ lies in the set $w(B_i)$ if and only if $v_k = \pm \omega^j$.
This gives $w$ up to its coset $wW_X$, determining $[w,X]=f^{-1}(v)$.

\section{Type $D$}
\label{type-D-proof-section}

We begin with a review from \cite{AthanasiadisR} 
on visualizing the elements of $NC(W)$ for $W$ of type $D$,
and extend this to similarly visualize the elements of $\Park^{NC}_W$,
analogous to the description of types $A$ and $B/C$.

\subsection{Visualizing type $D$}

The Weyl group $W$ of type $D_n$ is the subgroup of the hyperoctahedral group
$\symm_n^{\pm}$ consisting of $n \times n$ signed permutation matrices having
an even number of $-1$ entries, acting on $V = \CC^n$.  
Its reflections $T$ are these:
$$
\begin{aligned}
t^+_{ij}&=(i,j)(-i,-j),\\
t^-_{ij}&=(i,-j)(-i,j).
\end{aligned}
$$
The intersection subspaces $X$ in $\LLL$ of type $D_n$
correspond, as in type $B_n$, to set partitions
$\pi=\{B_1,\ldots,B_\ell\}$ of the set $\pm [n]$ having 
$\pi=-\pi$, except for an additional restriction:
if the zero block is present, not only is it unique, but 
it must involve at least two different coordinates
$x_i=x_j=0$.  In other words, a zero block $B_0=-B_0$, if
present in $\pi$, must contain at least four elements $\{+i,+j,-i,-j\}$ of $\pm [n]$.

One can take as Coxeter generators for $W$ 
the set $S = \{s_1, . . . , s_{n-1},s_n \}$,
where 
$$
\begin{aligned}
s_i&=t^+_{i,i+1}=(i,i+1) \text{ for }i=1,2,\ldots,n-1,\text{ and }\\
s_n&=t^-_{n-1,n}=(+(n-1),-n,-(n-1),+n).
\end{aligned}
$$
The Coxeter element 
$
c=s_1 \cdots s_{n-1}s_n
$
permutes the set $\pm [n]$ as the following simultaneous
$2(n-1)$-cycle and transposition
$$
c=(+1,+2,\dots,+(n-1),-1,-2,\dots,-(n-1)) \,\, (+n,-n)
$$ 
having order $h=2(n-1)$.  Drawing the labels
$$
(+1,+2,\dots,+(n-1),-1,-2,\dots,-(n-1))
$$ 
clockwise around a circle, 
and drawing two vertices labeled $+n,-n$ at the center of this circle,
one can depict set partitions of $\pm [n]$ as before by drawing
the polygonal convex hulls of their blocks.  These pictures
will always be centrally symmetric.  Furthermore, it turns out
(see \cite[\S3]{AthanasiadisR}) that the noncrossing flats $X$ 
in $NC(W)$ again biject with set partitions $\pi$ 
for which the polygonal convex hulls of two different
blocks have no intersection of their relative interiors.
In particular, this implies that if $\pi$ has a zero block
$B_0=-B_0$, then $\{+n, -n\} \subseteq B_0$.
One bijects such a type $D_n$ noncrossing partition $\pi$ or flat $X$ 
with the element $w$ in $NC(W)=[1,c]_T$, whose associated permutation of 
$\pm [n]$ has cycles 
\begin{enumerate}
\item[$\bullet$]
coming from each {\it nonzero} block $B$ by orienting the edges of the 
polygonal convex hull of $B$ in the cyclic order 
\eqref{infinite-cyclic-order}, and
\item[$\bullet$]
if there is a zero block $B_0$ present, then
including the $2$-cycle $(+n,-n)$, as well as the
$2(|B_0|-1)$-cycle containing the elements of
$B_0 \setminus \{\pm n\}$ oriented in the
cyclic order \eqref{infinite-cyclic-order}.
\end{enumerate}
Such an element $X$ has pointwise stabilizer 
$$
W_X = \symm_{B_1} \times \cdots \times \symm_{B_\ell} \times D_{B_0}
$$
where $D_{B_0}$ is present if and only if the partition $\pi$ of $\pm [n]$
corresponding to $X$ has a zero block.
For example, these three elements $X_1, X_2, X_3$ of $NC(W)$ for $W=D_7$,
$$
\begin{aligned}
X_1&=\{ x_1=x_2=-x_5, \,\, x_3=x_4 \}\\
X_2&=\{ x_1=x_2=-x_5=-x_7, \,\, x_3=x_4 \}\\
X_3&=\{ x_1=x_2=-x_5=x_7=0, \,\, x_3=x_4 \}
\end{aligned}
$$
are equal to $V^{w_1},V^{w_2},V^{w_3}$ for these elements $w_1, w_2, w_3$ in $[1,c]_T$:
$$
\begin{aligned}
w_1&=(+1,+2,-5)(-1,-2,+5)(+3,+4)(-3,-4),\\
w_2&=(+1,+2,-7,-5)(-1,-2,+7,+5)(+3,+4)(-3,-4),\\
w_3&=(+1,+2,+5,-1,-2,-5)(+7,-7)(+3,+4)(-3,-4).
\end{aligned}
$$
Their partitions $\pi_1,\pi_2,\pi_3$ of $\pm [n]$ are depicted here
\begin{center}
\includegraphics[scale=.35]{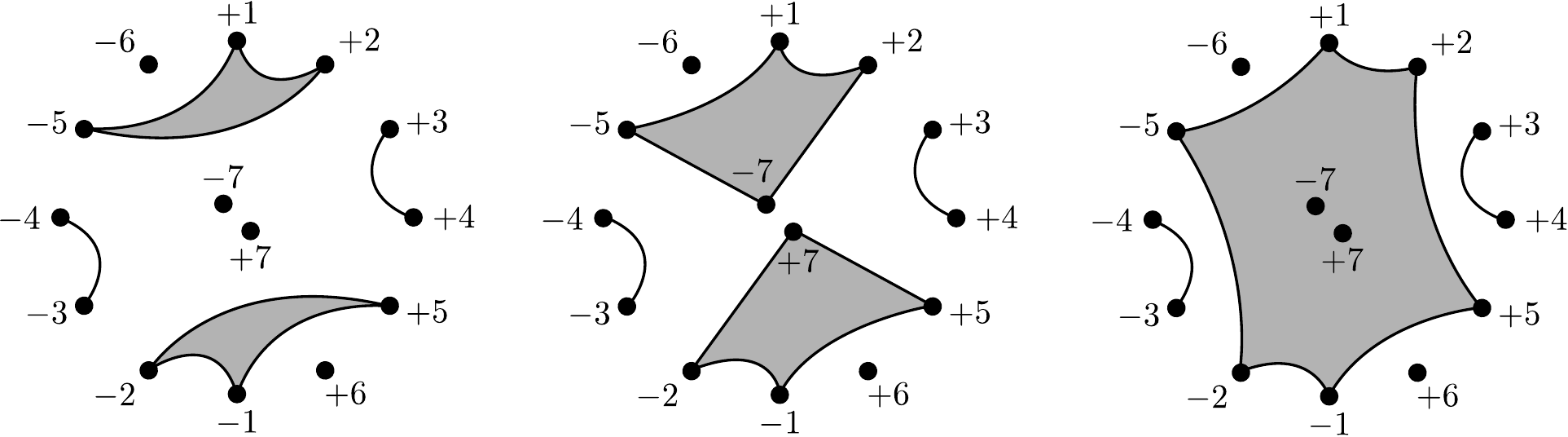}
\end{center}
and have pointwise stabilizers
$$
\begin{aligned}
W_{X_1}&=\symm_{\{+1,+2,-5\}} \times \symm_{\{+3,+4\}}\\
W_{X_2}&=\symm_{\{+1,+2,-5,-7\}} \times \symm_{\{+3,+4\}}\\
W_{X_3}&=\symm_{\{+3,+4\}} \times D_{\{+1,+2,-5,-7\}}.\\
\end{aligned}
$$
The action of $C=\langle c \rangle$
on $NC(W)=[1,c]_T$ corresponds to a clockwise rotation 
through $\frac{2 \pi}{2(n-1)}$ along with a simultaneous
swap of the labels $+n,-n$.  For example, here are the
pictures for $c(X_i)$ for $X_i$ with $i=1,2,3$ shown above:
\begin{center}
\includegraphics[scale=.35]{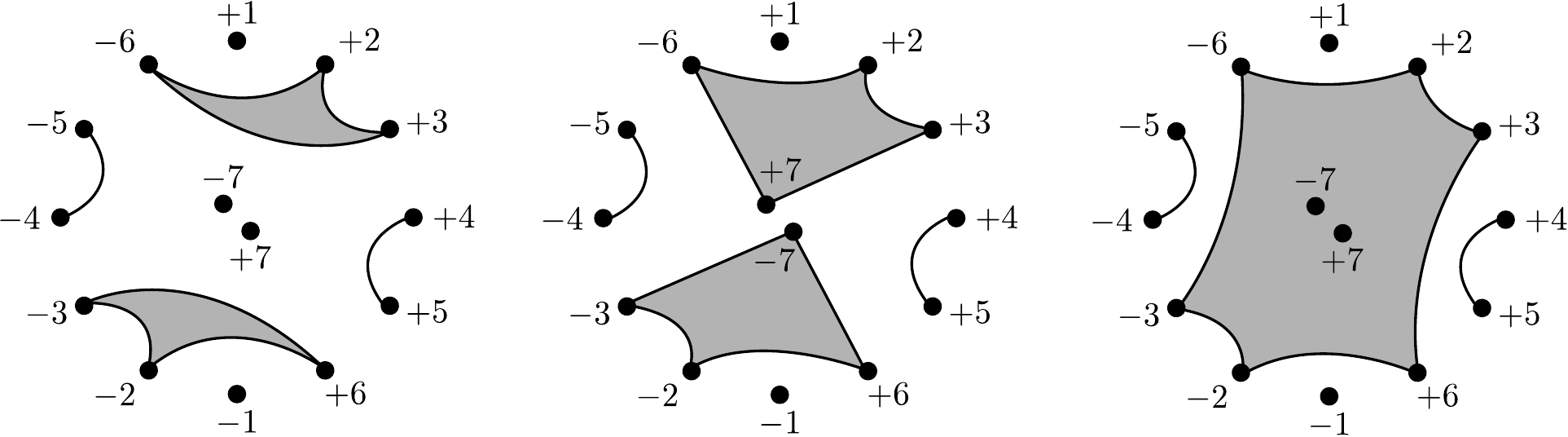}
\end{center}

As in type $B/C$, an element $[w,X]$ of $\Park^{NC}_W$ is an 
equivalence class where $X$ is a noncrossing
flat, and $w$ represents a coset $wW_X$.
One can again visualize this by labeling the blocks in the drawing
of $X$ by their images under $w$.

For example, the element $w$ in $D_7$ given by
$$
w=\left(
\begin{matrix}
+1 & +2 & +3 & +4 & +5 & +6 & +7 \\
-1 & -5 & +2 & -7 & -6 & -4 & -3 
\end{matrix}
\right)
$$
gives rise to three elements $[w,X_1], [w,X_2], [x,X_3]$ of $\Park^{NC}_{W}$
shown here:
\begin{center}
\includegraphics[scale=.35]{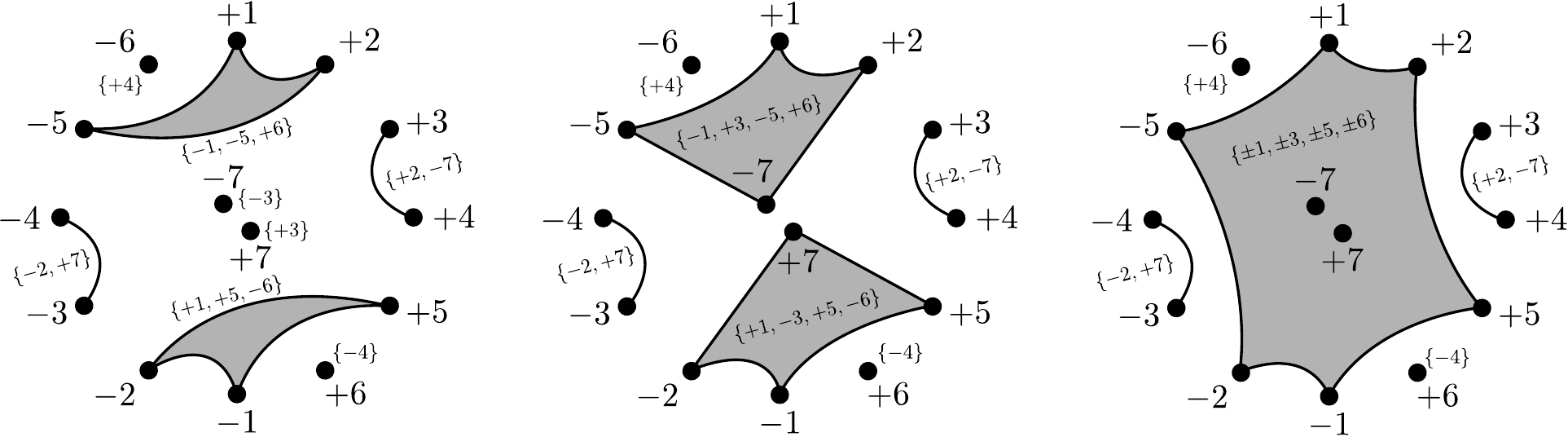}
\end{center}

\subsection{Proof of Main Conjecture (intermediate version) 
in Type D}

When $W=D_n$ acts on $V = \CC^n$,
with $x_1, \dots, x_n$ the standard coordinate functions in $V^*$,
we will choose 
$$
(\Theta)=(\theta_1,\ldots,\theta_n):=(x_1^{2n-1}, \dots, x_n^{2n-1}).
$$ 
It is easily seen that this is an hsop of 
degree $h+1$ in $\CC[V]$, and that the map
$\Theta$ sending $x_i$ to $\theta_i$ is 
$W$-equivariant, so that \eqref{strong-hsop-setup} is satisfied.
In this setting, the subvariety $V^\Theta$ cut out by the ideal
$$
\begin{aligned}
(\Theta-\xx)&=(x_1^{2n-1} - x_1, \dots, x_n^{2n-1} - x_n) \\
            &=(x_1(x_1^{2(n-1)} - 1), \dots, x_n(x_n^{2(n-1)} - 1)) \\
\end{aligned}
$$ 
consists of the points 
$(v_1, \dots, v_n) \in \CC^n$ such that each 
$v_i$ is either $0$ or a $2(n-1)^{st}$ root of unity,
that is, a power of $\omega := e^{\frac{2 \pi i}{2(n-1)}}$.
We now describe a $W \times C$-equivariant bijection 
$$
\begin{array}{rcl}
\Park^{NC}_W & \overset{f}{\longrightarrow} & V^\Theta\\
{} [w,X] &\longmapsto &v
\end{array}
$$
followed by a description of its inverse bijection.
\vskip.1in
\noindent
{\sf The forward bijection $f$.}

Given $[w,X]$, start with
the noncrossing partition $\pi$ of $\pm [n]$ corresponding
to $X$, and ignore the elements $+n,-n$ to obtain a
type $B_{n-1}$ noncrossing partition of $\pm [n-1]$.
Encode $X$ with $\pm n$ removed 
via its parenthesization of the string
$$
\cdots, -(n-1), +1, +2, \cdots, +(n-2), +(n-1), -1, -2, \cdots -(n-1), +1, +2, \cdots
$$
As before, every nonzero block will have an opener, $+j$ or $-j$,
and note that this opener will lie in the set $\pm [n-1]$.

The map $f$ then sends $[w,X]$ to the vector $v$ whose
$k^{th}$ coordinate $v_k$ depends upon the unique block $B_{i(k)}$ of $\pi$
for which $+k$ lies in $w(B_{i(k)})$:  
\begin{enumerate}
\item[$\bullet$]
if $B_{i(k)}$ is the zero block of $\pi$, then $v_k=0$, 
\item[$\bullet$]
if $B_{i(k)}$ is a singleton block of either form $\{+n\}$ or
$\{-n\}$, then $v_k=0$, 
\item[$\bullet$]
if $B_{i(k)}$ is a nonzero block, with positive opener $+j$,
then $v_k=+\omega^j$, 
\item[$\bullet$]
if $B_{i(k)}$ is a nonzero block, with negative opener $-j$,
then $v_k=-\omega^j$.
\end{enumerate}
Continuing the example with $n=7$ and
$[w,X_1],[w,X_2],[w,X_3]$ and $\pi$ from before,
one has parenthesizations corresponding to $X_1,X_2,X_3$ with $\pm 7$ removed
$$
\begin{aligned}
X_1  \leftrightarrow
 &\cdots (-5(-6)+1+2)(+3+4)(+5(+6)-1-2)(-3-4)(-5(-6)+1+2) \cdots \\
X_2  \leftrightarrow
 &\cdots (-5(-6)+1+2)(+3+4)(+5(+6)-1-2)(-3-4)(-5(-6)+1+2) \cdots \\
X_3  \leftrightarrow
 &\cdots -5(-6)+1+2(+3+4)+5(+6)-1-2(-3-4)-5(-6)+1+2 \cdots
\end{aligned}
$$
from which one can compute
$$
\begin{array}{rcccccccl}
     (&v_1,&v_2,&v_3,&v_4,&v_5,&v_6,&v_7&) \\
     & & & & & & & & \\
f([w,X_1])=
      (&+\omega^5,&+\omega^3,&0,&-\omega^6,&+\omega^5,&-\omega^5,&-\omega^3&) \\
f([w,X_2])=
      (&+\omega^5,&+\omega^3,&-\omega^5,&-\omega^6,&+\omega^5,&-\omega^5,&-\omega^3&) \\
f([w,X_3])=
      (&0,&+\omega^3,&0,&-\omega^6,&0,&0,&-\omega^3&). \\
\end{array}
$$
We leave it to the reader to check that $f$ 
is $W \times C$-equivariant. 

\vskip.1in
\noindent
{\sf The inverse bijection $f^{-1}$.}

Given $v$ in $V^\Theta$, we determine $f^{-1}(v)=[w,X]$ based on three cases
for the number of coordinates $j$ with $v_j=0$.

{\sf Case 1.}  There is a unique coordinate $v_{j_0}=0$.

In this case, by the definition of $f$, the partition $\pi$ of $\pm [n]$
corresponding to $X$ has no zero block, and has $\{+n\}, \{-n\}$ as
singleton blocks, with either $w(+n)=+j_0$ or $w(+n)=-j_0$.
Treat the remaining $n-1$ coordinates $\hat{v}$ of $v$, which are all nonzero,
as a type $B_{n-1}$ element of $V^\Theta$, and use this to recover
the other blocks $\hat{\pi}$ beside $\{+n\},\{-n\}$ of $\pi$, along
with their labeling by $\hat{w}$, taking values
in $[n] \setminus \{j_0\}$.  Finally, one has $w(+n)=+j_0$ or $w(+n)=-j_0$
depending upon whether the number of negative values taken by
$\hat{w}$ is even or odd.

{\sf Case 2.}  The number of coordinates $j$ with $v_j=0$ is $z \geq 2$.

In this case, by the definition of $f$, the element $X$ in $NC(W)$
must have a zero block containing $z$ coordinates, and one of them must
be the coordinate $x_n$.
Note that $v$ has length $n$, and its multiplicities $m_j$ of $\pm \omega^j$ 
give a composition $(m_1,m_2,\ldots,m_{n-1})$, whose sum is $n-z \leq n-2$.
Thus these multiplicities 
can be used as in the type $B$ inverse bijection to recover
the blocks of a type $B_{n-1}/C_{n-1}$-noncrossing parking function
and labeling of its blocks.
Each of its nonzero blocks $B$ will be labeled by a set 
of the same cardinality as $B$, but its zero block $B_0 \subseteq \pm [n]$ 
will include only $z-1$ coordinates, and be labeled by a set of coordinates
of size $z$ (namely those $j$ for which $v_j=0$).  Augmenting this zero block $B_0$
to include the extra coordinate $x_n$ recovers the correct $[w,X] \in \Park^{NC}_W$.

{\sf Case 3.}  Every coordinate of $v$ is nonzero.
In this case, by the definition of $f$, the element $X$ in $NC(W)$ must have
only nonzero blocks, one of which will be of size at least two and 
contain $+n$; call it $B$, so that $-B$ will contain $-n$.  
Because the multiplicities $(m_1,\ldots,m_{n-1})$ satisfy
$\sum_{j=1}^n m_j = n$, one knows that, for each $j$ having $m_j \geq 2$, 
the adjusted multiplicities
$(m_1,\ldots,m_{j-1},m_j-1,m_{j+1},\ldots,m_{n-1})$ will sum to $n-1$,
and therefore recover a valid type $B_{n-1}/C_{n-1}$ noncrossing
partition $\hat{X}$.  One can check that {\it exactly one} choice of $j$ (call it $j_0$)
has the property that if one adds in one of the central vertices $\pm n$ to the
block opened by $+j_0$, the picture remains noncrossing for type $D_n$.  Once
we know whether to add $+n$ or $-n$ to this block opened by $+j_0$, 
we will uniquely define an
element $X$ in $NC(W)$, and the actual
coordinates of $v$ let one recover the labeling of the blocks of $X$ to get
$[w,X]$.  But now this determines whether $+n$ or $-n$ should go in the
block opened by $j_0$, since only one of the two choices will make $w$ have
an even number of negative signs.

\section{Proof of Main Conjecture (weak version) in type $A$}
\label{type-A-proof-section}

Recall that the weak version of the Main Conjecture asserts the
isomorphism $\Park^{alg}_W \cong \Park^{NC}_W$
as $W \times C$-representations.
As mentioned in the Introduction,
Proposition~\ref{algebraic-character-prop} lets one rephrase this
as an equality of $W \times C$-character values:
for every $u$ in $W$ and every integer $\ell$,
\begin{equation}
\label{desired-character-equality}
\begin{array}{ccc}
\chi_{\Park^{alg}_W}(u,c^\ell) & = & \chi_{\Park^{NC}_W}(u,c^\ell) \\
\| & & \| \\
(h+1)^{\mathrm{mult}_u(\omega^\ell)} &  & \left| \left( \Park^{NC}_W \right)^{(u,c^\ell)} \right| 
\end{array}
\end{equation}
where $\omega := e^{{2 \pi i}{h}}$ as usual,
and $\mathrm{mult}_u(\omega^\ell)$ is the multiplicity
of $\omega^\ell$ as an eigenvalue of $u$ acting on $V$.

Note that when $c^\ell=1$, but $u$ varies over all of $W$,
equality in \eqref{desired-character-equality} is equivalent to the isomorphism 
$\Park^{alg}_W \cong \Park^{NC}_W$ of $W$-representations,
already known case-by-case, and in particular, known for type $A$; see 
Sections~\ref{coincidence-of-perm-reps-section}, \ref{park-alg-defn-section}.
Consequently, to prove \eqref{desired-character-equality},
it only remains to consider the case where $c^\ell \neq 1$.
In this case, $c^\ell$ has multiplicative order 
$
d:=\frac{n}{\gcd(\ell,n)} \geq 2,
$
and we will denote 
$
\hat{n}:=\frac{n}{d},
$
so that 
$
\langle c^\ell \rangle=\langle c^{\hat{n}} \rangle.
$

Recall that in type $A_{n-1}$, one has $W=\symm_n$, 
and we have chosen as Coxeter element the $n$-cycle $c=(1,2,\ldots,n)$.
For this choice, the combinatorial model 
for the parking functions $[w,X]$ in $\Park^{NC}_W$
was discussed in Example~\ref{type-A-pictorial-example}:
one has $X$ corresponding to a noncrossing set partition $\pi=\{B_1,B_2,\ldots\}$ 
of $[n]$, and $w$ gives a labeling of each of its blocks $B_i$ by a set $w(B_i)$
of the same cardinality $|w(B_i)|=|B_i|$, giving 
another set partition $\pi=\{w(B_1),w(B_2),\ldots\}$ of $[n]$.

It is easily checked that a permutation $u$ acting
on $V$ has $\omega^\ell \neq 1$ as an eigenvalue with multiplicity
$\mathrm{mult}_u(\omega^\ell)$ equal to
$$
r_d(u):= | \{ \text{cycles of }u\text{ whose size is divisible by }d \}|.
$$
Thus the equality \eqref{desired-character-equality}
to be proven can be rephrased as follows:  one must show that,
for every permutation $u$ in $\symm_n$ and integer $\ell$ 
with $c^\ell \neq 1$, one has
\begin{equation}
\label{rephrased-character-equality}
(n+1)^{r_d(u)} =
\left| \left( \Park^{NC}_W \right)^{(u,c^\ell)} \right|.
\end{equation}
Our approach in proving 
\eqref{rephrased-character-equality} 
will be to show both sides count the following objects.

\begin{defn} 
Let $u$ be in $\symm_n$ and $c^\ell \neq 1$, as above.
Extend the action of $c^\ell$ permuting $[n]$
to an action on $[n] \cup \{0\}$ by letting $c^\ell(0)=0$, and
say that a function $f:[n] \rightarrow [n] \cup \{0\}$
is {\sf $(u,c^\ell)$-equivariant} if 
$$
f(u(j)) = c^\ell f(j)
$$
for every $j$ in $[n]$.
Equivalently, whenever $f(j) \neq 0$, then
$f(u(j)) \equiv f(j) \bmod{n}.$
\end{defn}

Seeing that the $(u,c^\ell)$-equivariant $f$ are counted by the left side of 
\eqref{rephrased-character-equality} is easy.

\begin{proposition}
\label{easiest-equivariant-count}
For any $u$ in $\symm_n$ and $c^\ell \neq 1$,
with notation as above, $(n+1)^{r_d(u)}$ counts
the number of $(u,c^\ell)$-equivariant 
functions $f:[n] \rightarrow [n] \cup \{0\}$.
\end{proposition}
\begin{proof}
The $(u,c^\ell)$-equivariance implies that such
a function $f$ is completely determined by its
values on one representative $j$ from each cycle of $u$.
If the cycle has size divisible by $d$, this
value $f(j)$ can be an arbitrary element of $[n] \cup \{0\}$,
giving $n+1$ choices.  If the cycle has size not divisible
by $d$ then this value $f(j)$ must be a fixed point of $c^\ell$,
that is, $f(j)=0$.
\end{proof}

The first step in showing that $(u,c^\ell)$-equivariant functions $f$
are also counted by the right side of \eqref{rephrased-character-equality}
is similar in spirit to the proof of a Stirling number identity
presented in the Twelvefold Way \cite[Chapter 1, eqn. (24d), pp. 34--35]{StanEC1}:
one classifies functions $f$ according to how their fibers $f^{-1}(j)$
partition the domain $[n]$.  To this end, make the following definition.

\begin{defn}
Given $u$ in $\symm_n$ and $d \geq 2$, say that a set partition
$\pi=\{A_1,A_2,\ldots\}$ of $[n]$ is {\sf $(u,d)$-admissible} if 
\begin{enumerate}
\item[$\bullet$]
$\pi$ is $u$-stable in the sense that  $u(\pi) = \{u(A_1), u(A_2), \ldots \} = \pi$,
so for each $A_i$ there exists some $j$ with $u(A_i)=A_j$, and
\item[$\bullet$]
at most one block $A_{i_0}$ is itself $u$-stable in the
sense that $u(A_{i_0}) =A_{i_0}$, and
\item[$\bullet$] 
all other blocks beside the $u$-stable block $A_{i_0}$ (if present)
are permuted by $u$ in orbits of length $d$, that is,
$\{ A_i, u(A_i), u^2(A_i), \ldots, u^{d-1}(A_i) \}$ are distinct,
but $u^d(A_i)=A_i$.
\end{enumerate}
\end{defn}

\begin{proposition}
\label{12-fold-way-prop}
With notation as above, 
the number of $(u,c^\ell)$-equivariant 
functions $f:[n] \rightarrow [n] \cup \{0\}$ is 
$$
\sum_{\pi} n(n-d)(n-2d) \cdots (n-(k_\pi-1)d)
$$
where $k_\pi$ is the number of $u$-orbits of blocks
of $\pi$ having length $d$.
\end{proposition}

\begin{proof}
Associate to each $(u,c^\ell)$-equivariant function $f$
a partition $\pi=\{A_1,A_2,\ldots\}$ 
of the domain $[n]$, by letting the blocks of $\pi$
be the nonempty fibers $f^{-1}(j)$.
Note that equivariance forces $\pi$ to be
$(u,d)$-admissible.  Note also that $\pi$ contains
a $u$-stable block $A_{i_0}=f^{-1}(0)$ 
if and only if $f$ takes on the value $0$.

On the other hand, if one fixes a $(u,d)$-admissible
partition $\pi$ of $[n]$, one can count how many
$(u,c^\ell)$-equivariant $f$ are associated to it as follows.  
If $\pi$ has a $u$-stable block $A_{i_0}$ then
set $f(A_{i_0})=0$;  otherwise $f$ does not take on
the value $0$.  To determine the rest of $f$,
choose representative blocks $A_1,A_2,\ldots,A_{k_\pi}$ 
from each of the $u$-orbits of blocks
of $\pi$ having length $d$.  After choosing 
the value $f(A_1)=j_1$ from the $n$ choices in $[n]$, 
this forces (working modulo $n$) that
$$
\begin{aligned}
f(u(A_1))&=j_1+\ell,\\
f(u^2(A_1))&=j_1+2\ell,\\
\vdots\\
f(u^{d-1}(A_1))&=j_1+(d-1)\ell.
\end{aligned}
$$
This then leaves $n-d$ choices for $f(A_2)=j_2$, forcing
$$
\begin{aligned}
f(u(A_2))&=j_2+\ell,\\
f(u^2(A_2))&=j_2+2\ell,\\
\vdots\\
f(u^{d-1}(A_2))&=j_2+(d-1)\ell.
\end{aligned}
$$
There are then $n-2d$ choices for $f(A_3)=j_3$, etc.
\end{proof}

\noindent
Equality \eqref{rephrased-character-equality} now follows
from Propositions~\ref{easiest-equivariant-count},
\ref{12-fold-way-prop} and the next proposition.

\begin{proposition}
Fix $u$ in $\symm_n$ and $c^\ell \neq 1$, with notations as above.

If $[w,X]$ in $\Park^{NC}_W$ is fixed by $(u,c^\ell)$,
and has 
$
X=\{B_1,B_2,\ldots\},
$ 
then the associated set partition
$
\pi=\{A_1,A_2,\ldots\}
$
of $[n]$ having blocks $A_i:=w(B_i)$ will be $(u,d)$-admissible.

Conversely, fixing a $(u,d)$-admissible partition $\pi$, there will be exactly
$$
n(n-d)(n-2d) \cdots (n-(k_\pi-1)d)
$$
$[w,X]$ in $\Park^{NC}_W$ fixed by $(u,c^\ell)$ that are associated to $\pi$
in this way.
\end{proposition}
\begin{proof}
For the first assertion, note that 
$$
[w,X]=(u,c^\ell)[w,X]=[uwc^{-\ell},c^\ell(X)]
$$
implies, in particular, that $c^\ell(X)=X$, so $X$ is a 
noncrossing partition of $[n]$ that has $d$-fold rotational
symmetry.  This means that $X$ can have at most one $d$-fold
symmetric block $B_{i_0}$ (if any at all), and all of its other
blocks come in $d$-fold rotation orbits of length $d$.
The fact that $uwc^{-\ell} W_X=wW_X$ then means
that the partition $\pi=\{A_1,A_2,\ldots\}$ defined by
$A_i:=w(B_i)$ is $(u,d)$-admissible.  In fact, one has
that the $d$-fold symmetric block of $X$ (if present)
gives the unique $u$-stable block $A_{i_0}=w(B_{i_0})$ of $\pi$
(if present).  Also, if one chooses representatives
$B_1,B_2,\ldots,B_k$ for the $d$-fold rotation orbits
of length $d$ among the blocks of $X$, then
their corresponding $w$-images 
$A_1,A_2,\ldots,A_k$ give representatives for the
$u$-orbits of $\pi$ of length $d$.  In particular, one must
have $k=k_\pi$.

For the second assertion, define the {\sf type} of 
a $(u,d)$-admissible partition $\pi$ of $[n]$ to be the
sequence $\mu = (\mu_1,\mu_2,\ldots,\mu_{\hat{n}})$
where $\pi$ has $\mu_j$ different $u$-orbits of blocks
of length $d$ in which the blocks have size $j$.
Similarly define the type $\mu$ of
a $d$-fold symmetric noncrossing partition $X$ of $[n]$ 
to mean that $X$ has $\mu_j$ different $d$-fold rotation orbits of 
length $d$ consisting of blocks having size $j$.
When $\pi$ is associated to $X$ as above, note that
they have the same type $\mu$.  

Having fixed a $(u,d)$-admissible partition $\pi$,
say of type $\mu$, one can count how many  
$[w,X]$ in $\Park^{NC}_W$ fixed by $(u,c^\ell)$ will be associated to it
in the following way.
A result of Athanasiadis \cite[Theorem 2.3]{Athanasiadis3}
counts the number of
$d$-fold symmetric\footnote{Actually, he counts
the {\bf centrally-}symmetric noncrossing partitions $X$ of type $\mu$,
but these have an easy bijection to those which are $d$-fold
symmetric, for any $d \geq 2$.} 
noncrossing partitions $X$ of $[n]$ having type $\mu$ as
\begin{equation}
\label{Athanasiadis-formula}
\frac{\hat{n} (\hat{n}-1) (\hat{n}-2) \cdots (\hat{n}-(k-1)) }
      {\mu_1 ! \mu_2 ! \cdots \mu_{\hat{n}}!}.
\end{equation}
For each such $X$, and for each fixed block size $j$,
the permutation $w$ must match the $\mu_j$ different
$d$-fold rotation orbits having blocks of size $j$
\begin{equation}
\label{c-orbit-in-X}
(B, c^\ell(B), c^{2\ell}(B),\ldots,c^{(d-1)\ell}(B) )
\end{equation}
with the $\mu_j$ different $u$-orbits of $\pi$ having blocks of
size $j$
\begin{equation}
\label{u-orbit-in-pi}
(A, u(A), u^2(A),\ldots,u^{d-1}(A) ).
\end{equation}
There are $\mu_j !$ ways to do this matching.  Having picked
such a matching, for each of the $k=\mu_1+\mu_2+\cdots+\mu_{\hat{n}}$ 
different matched orbit pairs as in \eqref{c-orbit-in-X},
\eqref{u-orbit-in-pi}, $w$ has $d$ choices for how to align them
cyclically:  $w$ specifies which of the $d$ possible sets 
$\{ A, u(A), u^2(A), \ldots,u^{d-1}(A)\}$ will be the image $w(B)$.

Hence there are $d^k \mu_1 ! \mu_2 ! \cdots \mu_{\hat{n}}!$
ways to choose the image sets $w(B_i)$ after making the
choice of $X=\{B_1,B_2,\ldots\}$.  Together with \eqref{Athanasiadis-formula},
this gives the desired count:
$$
d^k  \mu_1 ! \mu_2 ! \cdots \mu_{\hat{n}}!
\cdot
\frac{\hat{n} (\hat{n}-1) (\hat{n}-2) \cdots (\hat{n}-(k-1)) }
      {\mu_1 ! \mu_2 ! \cdots \mu_{\hat{n}}!} 
=n (n-d) (n-2d) \cdots (n-(k-1)d).
$$
\end{proof}

\section{Narayana and Kirkman polynomials}
\label{Narayana-Kirkman-section}

After proving a statement
equivalent to Corollary~\ref{Narayana-Kirkman-corollary}
on the Kirkman numbers for $W$, we explain how 
calculations of Gyoja, Nishiyama, and Shimura \cite{GNS} give
product formulas for the Kirkman numbers in classical types, and even their graded
$q$-analogues.

\subsection{Proof of Corollary~\ref{Narayana-Kirkman-corollary}}

The following result is an equivalent version of
Corollary~\ref{Narayana-Kirkman-corollary};
see the discussion in Section~\ref{third-consequence-section}.

\begin{proposition}
\label{Narayana-Kirkman-corollary-equivalent}
For $W$ an irreducible real reflection group, letting
$\Park(W)$ denote either of the equivalent $W$-representations
$\Park^{NC}_W \cong \Park^{alg}_W$, one has
$$
\sum_{k=0}^n \langle \,\, \chi_{\wedge^k V} \,\ 
              , \,\, \chi_{\Park(W)} \,\, \rangle_W \cdot (t-1)^k
= \sum_{X \in NC(W)} t^{\dim_\CC X}.
$$
\end{proposition}
\begin{proof}
Note that 
$$
\begin{aligned}
\langle \,\, \chi_{\wedge^k V} \,\ 
              , \,\, \chi_{\Park(W)} \,\, \rangle_W 
&=\sum_{X \in NC(W)} 
\langle \,\, \chi_{\wedge^k V} \,\ 
              , \,\ \Ind_{W_X}^W \triv_{W_X} \,\, \rangle_W  \\
&=\sum_{X \in NC(W)} 
\langle \,\, \Res^W_{W_X} \chi_{\wedge^k V} \,\ 
              , \,\ \triv_{W_X} \,\, \rangle_{W_X}  \\
\end{aligned}
$$
using \eqref{definitional-permutation-reps} and Frobenius Reciprocity.
For purposes of later multiplying this by $(t-1)^k$ and summing on $k$,
note that when $w$ acts on $V$ and on $\wedge^k V$, 
one has
\begin{equation}
\label{exterior-powers-in-characteristic-polynomial}
\sum_{k=0}^n t^k \chi_{\wedge^k V}(w) = \det(1+tw).
\end{equation}
However, note that for $w$ that happen to lie in the pointwise
stabilizer $W_X$ of $X$ inside $W$,
there will be $\dim_\CC X$ extra
$+1$-eigenvalues for $w$ when it is considered as an element acting on $V$, 
rather than as an element $w/X$ of $W_X$ acting on $V/X$.
Hence
$$
\begin{aligned}
\sum_{k=0}^n (t-1)^k \chi_{\wedge^k V}(w) 
&= \det(1+(t-1)w) \\ 
&= (1+(t-1))^{\dim_\CC X} \det\left(1+(t-1)(w/X)\right) \\
&= t^{\dim_\CC X} \sum_{k=0}^n (t-1)^k \chi_{\wedge^k(V/X)}(w/X) 
\end{aligned}
$$
and therefore
$$
\begin{aligned}
&\sum_{k=0}^n (t-1)^k 
    \langle \,\, \chi_{\wedge^k V} \,\ 
              , \,\, \chi_{\Park(W)} \,\, \rangle_W 
=\sum_{X \in NC(W)} 
   \sum_{k=0}^n (t-1)^k 
   \langle \,\, \Res^W_{W_X} \chi_{\wedge^k V} \,\ 
              , \,\ \triv_{W_X} \,\, \rangle_{W_X}  \\
&= \sum_{X \in NC(W)} t^{\dim_\CC X} 
     \sum_{k=0}^n (t-1)^k 
   \langle \,\, \chi_{\wedge^k(V/X)}(w/X) \,\ 
              , \,\ \triv_{W_X} \,\, \rangle_{W_X}  
= \sum_{X \in NC(W)} t^{\dim_\CC X} 
\end{aligned}
$$
where the very last equality uses the fact that
the $W_X$-representations $\{ \wedge^k(W/X) \}$
for $k=0,1,2,\ldots\dim W/X$ are 
inequivalent $W_X$-irreducibles \cite[Chap. 5, \S 2, Exer. 3]{Bourbaki},
with $\wedge^0(W/X) = \triv_{W_X}$.
\end{proof}

\subsection{Formulas for Kirkman and $q$-Kirkman numbers}

Define the {\it Kirkman} and {\it $q$-Kirkman numbers} for a real irreducible
reflection $W$ acting on $V=\CC^n$ and $0 \leq k \leq n$ by
\begin{equation}
\label{q-Kirkman-definition}
\begin{aligned}
\Kirk(W,k)&:=
 \langle \,\, \chi_{\wedge^k V} \,\ 
              , \,\, \chi_{\Park(W)} \,\, \rangle_W  \\
\Kirk(W,k;q)&:=
\sum_{d \geq 0}
 \langle \,\, \chi_{\wedge^k V} \,\ 
              , \,\, \chi_{\left( \CC[V]/(\Theta) \right)_d } \,\, \rangle_W \cdot q^d.
\end{aligned}
\end{equation}
so that $\left[ \Kirk(W,k;q)\right]_{q=1}=\Kirk(W,k)$ has the combinatorial
interpretation from cluster theory given by 
Corollary~\ref{Narayana-Kirkman-corollary}.  One calculates them via
\begin{equation}
\label{q-Kirkman-as-intertwiner}
\begin{aligned}
\sum_{k=0}^n \Kirk(W,k;q) t^k 
  &=  \left\langle \,\, \sum_{k=0}^n \chi_{\wedge^k V} t^k  \,\ 
              , \,\, \sum_{d \geq 0} \chi_{\left( \CC[V]/(\Theta) \right)_d } q^d \,\, \right\rangle_W  \\
&=\frac{1}{|W|} \sum_{w \in W} \frac{\det(1+tw) \det(1-q^{h+1}w)}{\det(1-qw)}
\end{aligned}
\end{equation}
using \eqref{exterior-powers-in-characteristic-polynomial} 
and Proposition~\ref{algebraic-character-prop}.
We wish to record here some more explicit product formulas for these numbers
using character calculations due to Gyoja, Nishiyama, and Shimura \cite{GNS}.  
To state these, first recall the $q$-number
$
[n]_q := 1 + q + \cdots + q^{n-1}, 
$
the $q$-factorial 
$[n]!_q:= [1]_q [2]_q \cdots [n]_q$,
and the $q$-binomial
$
\qbin{n}{k}{q} :=\frac{[n]!_q}{[k]!_q [n-k]!_q}.
$

\begin{proposition}
\label{q-Kirkman-formulas}
For $W$ an irreducible real reflection group acting on $V=\CC^n$, 
with root system $\Phi$, positive roots $\Phi^+$,
fundamental degrees $d_1 \leq \ldots \leq d_n$, and Coxeter number $h:=d_n$, one has
\begin{equation}
\label{det-and-triv-q-Kirkman}
\Kirk(W,n;q) = q^{|\Phi^+|} 
\qquad \text{ and } \qquad
\Kirk(W,0;q) = \prod_{i=1}^n \frac{ [h+d_i]_q }{ [d_i]_q } 
=: \Cat(W,q) 
\end{equation}
as well as these explicit formulas for $\Kirk(W,k;q)$ in the classical
infinite families
\begin{center}
\begin{tabular}{|c|c|}
\hline
$W$ & $\Kirk(W,k;q)$ \\ \hline\hline
    & \\
$A_{n-1} (= \symm_n)$ & $\frac{q^{\binom{k+1}{2}}}{[n]_q} \qbin{n}{k}{q} \qbin{2n-k}{n-k-1}{q}$ \\
    & \\ \hline
    & \\
$B_n = C_n$ & $q^{k^2} \qbin{n}{k}{q^2} \qbin{2n-k}{n-k}{q^2}$ \\ 
    & \\ \hline
    & \\
$D_n$ & $q^{k^2} \qbin{n-1}{k}{q^2} \qbin{2n-k-1}{n-k}{q^2}+
         q^{k^2-2k+n} \qbin{n}{k}{q^2} \qbin{2n-k-2}{n-k}{q^2}$ \\ 
    & \\ \hline
\end{tabular}
\end{center}
and this formula for rank $2$ dihedral groups $W=I_2(m)$:
\begin{equation}
\label{dihedral-q-Kirkman}
\Kirk(I_2(m),1;q) \,\, = \,\, q^1[m]_{q^2}+q^{m-1}[2]_{q^2} 
                 \,\, = \,\, q^1 [m+2]_q \frac{[2]_{q^{m-2}}}{[2]_q}.
\end{equation}
\end{proposition}

\begin{proof}
The first formula in \eqref{det-and-triv-q-Kirkman} can be proven via a character
computation, but one can also argue it as follows.
Since $\wedge^n V$ carries the alternating $W$-character $\det$,
the formula asserts that 
the unique copy of this character $\det$ occurring in $\CC[V]/(\Theta)$,
whose existence is proven in Lemma~\ref{Park-has-det-once-lemma},  occurs 
in degree $|\Phi^+|$.  In other words, it is carried by the smallest nonzero
$W$-alternating polynomial $\prod_{H \in \Cox(\Phi)} \ell_H$, where $\ell_H$ is a linear
form in $V^*$ that vanishes on the reflecting hyperplane $H$.
This follows because the same is well-known to be true 
for the {\it coinvariant algebra} \cite[Ch. 3]{Humphreys}, and the coinvariant algebra
is known \cite[Prop. 5.2]{BessisR} to be a quotient
of $\CC[V]/(\Theta)$.
 
The second formula in \eqref{det-and-triv-q-Kirkman} is simply re-asserting the discussion 
surrounding \eqref{q-Catalan-formula}, since
$\wedge^0 V$ carries the trivial $W$-character.

For the classical Weyl groups $W$ and
for any irreducible $W$-character $\chi$,
Gyoja, Nishiyama, and Shimura \cite{GNS} compute explicit formulas 
for a generating function
\begin{equation}
\label{GNS-generating-function}
\tilde{\tau}_W(\chi;q,u):=
\frac{1}{|W|} \sum_{w \in W} \chi(w) \frac{\det(1+uw)}{\det(1-qw)}.
\end{equation}
Comparing with \eqref{q-Kirkman-as-intertwiner} one sees that  
$$
\Kirk(W,k;q)=\left[ \,\, \tilde{\tau}_W(\chi_{\wedge^k};q,u) \,\, \right]_{u=-q^{h+1}}.
$$
For type $A_{n-1}$, one must divide the result by $[n+1]_q$,
due to the fact that they are not using the irreducible action of $W=\symm_n$
on $V \cong \CC^{n-1}$, but rather the action on $\CC^n$ by permuting
coordinates\footnote{Alternatively, in type $A_{n-1}$, this same calculation 
can be reproduced by combining the {\it $q$-hook-content formula}
for the principal specialization of a Schur function \cite[Thm. 7.21.2]{StanEC2} with
a calculation from Haiman \cite[Prop. 2.5.2]{Haiman}.}.
One needs to know that  the exterior power $\wedge^k V$ corresponds
in their indexing of irreducible characters, to
\begin{enumerate}
\item[$\bullet$]
$\chi^{\alpha}$ with $\alpha=(n-k,1^k)$ in type $A_{n-1}$
\cite[formula (3.2) ]{GNS}, 
\item[$\bullet$]
$\chi^{(\alpha,\beta)}$ with $(\alpha,\beta)=((n-k),(1^k))$ in type $B_n/C_n$
\cite[formula (3.9) ]{GNS}, 
\item[$\bullet$]
$\chi^{\{\alpha,\beta\}}$ with\footnote{One must also remove a stray factor of 
$\frac{1}{2}$ appearing on the right in \cite[formula (3.15) ]{GNS}.} 
$\{ \alpha,\beta\}=\{ (n-k),(1^k)\}$ in type $D_n$
\cite[formula (3.15) ]{GNS}.
\end{enumerate}

Formula \eqref{dihedral-q-Kirkman} can be verified as follows.
The reflection representation $V$ for $W=I_2(m)$ turns out to
be induced from the character $\chi_\omega$ of its rotation subgroup $C=\langle c \rangle$
that sends the generator $c$ to $\omega=e^{\frac{2 \pi i}{m}}$.
Thus Frobenius Reciprocity gives 
$$
\begin{aligned}
\Kirk(W,1;q)
&=\sum_{d \geq 0}
 \langle \,\, \chi_{V} \,\ 
              , \,\, \chi_{\left( \CC[V]/(\Theta) \right)_d } \,\, \rangle_W \cdot q^d\\
&=\sum_{d \geq 0}
 \langle \,\, \chi_\omega \,\ 
              , \,\, \Res^W_C \chi_{\left( \CC[V]/(\Theta) \right)_d } \,\, \rangle_W \cdot q^d.
\end{aligned}
$$
In other words, we wish to compute the graded Hilbert series in the variable $q$
for the $\chi_\omega$-isotypic component of $\CC[V]/(\Theta)$,
considered as a $C$-representation by restriction from $W$.  
If one chooses the convenient hsop $(\Theta)=(x^{m+1},y^{m+1})$
inside $\CC[V]=\CC[x,y]$ then $\CC[V]/(\Theta)$ has $\CC$-basis $\{x^i y^j\}_{i,j=0,1,2,\ldots,m}$.
Hence its $\chi_\omega$ isotypic component has $\CC$-basis
$$
\begin{aligned}
&\{x^i y^j: 0 \leq i,j \leq m, \text{ and }i-j \equiv 1 \mod m\}\\
&=\{ x^1y^0, x^2y^1, x^3y^2, \cdots, x^m y^{m-1} \} \cup \{x^0y^{m-1},x^1 y^m\}
\end{aligned}
$$
and its Hilbert series is
$$
(q^1+q^3+q^5+\cdots+q^{2m-1})+(q^{m-1}+q^{m+1}) =
q^1[m]_{q^2}+q^{m-1}[2]_{q^2}. 
$$
\end{proof}

We close this section with two remarks on the formulas in
Proposition~\ref{q-Kirkman-formulas}.

\begin{remark} \rm \
One can readily check that the 
formula for $\Kirk(W,k;q)$ when $W$ is of type $A_{n-1}$
given above is, up to a power of $q$, the same as 
the {\it $q$-hook-formula} $f^\lambda(q)$ or {\it fake degree polynomial}
for $\lambda=(n-k, n-k, 1^k)$ which $q$-counts
standard Young tableaux of shape $\lambda$ by their major index; see
\cite[Cor. 7.21.5]{StanEC2}.  

At $q=1$, this
coincidence between type $A$ Kirkman numbers and tableaux numbers
is an observation of K. O'Hara and A.V. Zelevinsky;  see
Stanley \cite{Stanley-dissections} for a bijective proof.   
We do not have a good algebraic explanation for the coincidence when
the variable $q$ is unspecialized, even though both polynomials in $q$
have Hilbert series interpretations.
\end{remark}

\begin{remark}\rm \
Other versions of $q$-Kirkman numbers $\Kirk(W,k;q)$ have been
considered recently in the literature in types $A,B/C,D$ and $I_2(m)$,
but defined in ad hoc ways.  They have appeared in conjunction with
{\it cyclic sieving phenomena} ({\it CSP's}) involving the set $X$ of clusters for
$W$ of a fixed cardinality-- recall from Corollary~\ref{Narayana-Kirkman-corollary}
that $\Kirk(W,k)$ counts the clusters of cardinality $n-k$.
This set $X$ carries a natural action of a cyclic group 
$C=\langle \tau \rangle$ of order $h+2$
introduced by Fomin and Zelevinsky \cite{FominZelevinsky}, generated
by their {\it deformed Coxeter element} $\tau$.

We compare here this type-by-type with the
other version of $q$-Kirkman numbers
with a renormalized version of our $\Kirk(W,k;q)$ defined in \eqref{q-Kirkman-definition},
where one divides by the smallest power of $q$ to make the constant term $1$;
call this renormalized polynomial $\Kirk_0(W,k;q)$.  

\vskip.1in
\noindent
{\sf Type A}.
In type $A$, it was shown in \cite[Thm. 7.1]{RStantonWhite}, 
by brute force evaluation and enumeration
that one has a CSP for this
triple $(X,X(q),C)$, with $X$ and $C$ as described above,
and $X(q)=\Kirk_0(W,k;q)$.  

\vskip.1in
\noindent
{\sf Type B/C}.
In type $B/C$, the analogous assertion was proven using
similar methods, in work of Eu and Fu \cite[Thm. 4.1 at $s=1$]{EuFu},
again using $X(q)=\Kirk_0(W,k;q)$. 

\vskip.1in
\noindent
{\sf Type I (dihedral)}.
For type $I_2(m)$, Eu and Fu prove the analogous
assertion in \cite[\S 6, $s=1$]{EuFu}, except that
their choice of polynomials $X(q)$, while
agreeing with $\Kirk_0(W,k;q)$ for $k=0,2$, 
will disagree at $k=1$:  they use instead
$$
X(q)=
\begin{cases}
[a+2]_q & \text{ if }a\text{ is odd,}\\
[a+2]_{q^2} & \text{ if }a\text{ is even.}
\end{cases}
$$
However, one readily checks that 
their choice of 
$X(q)$ is congruent to $\Kirk_0(W,1;q)$ modulo $q^{h+2}-1$.
Thus either polynomial has the same evaluation 
when $q$ is any $(h+2)^{th}$ root-of-unity,
and hence the CSP holds for either choice.

\vskip.1in
\noindent
{\sf Type F}.
Eu and Fu do not suggest $X(q)$ in general for type $F_4$,
but they do give data in \cite[Figure 15]{EuFu} on the orbit
sizes for the action of the deformed Coxeter element $\tau$
on the sets of clusters of a fixed cardinality $n-k$. For each
$k=0,1,2,3,4$, this data shows all orbits of size $7=\frac{h+2}{2}$,
and we checked that $\Kirk_0(W,k;q)$ would predict this correctly
for $k=0,1,2,3,4$.

\vskip.1in
\noindent
{\sf Type D}.
In type $D$, Eu and Fu prove in \cite[Thm. 5.1 at $s=1$]{EuFu} 
an analogous CSP, but
using polynomials $X(q)$ that disagree more fundamentally with 
$\Kirk_0(W,k;q)$. Their formula (16) at $s=1$ (and replacing
$k$ by $n-k$) defines a polynomial
$$
\begin{aligned}
X(q)&= \qbin{2n-k-1}{n-k}{q^2} \left( \qbin{n-1}{n-k}{q^2} +
      q^n \qbin{n-2}{n-k-1}{q^2} \right) \\
    & + \left( \qbin{2n-k-1}{n-k}{q^2} + 
        q^n \qbin{2n-k-2}{n-k}{q^2} \right) \qbin{n-2}{k}{q^2}.
\end{aligned}
$$
which they prove gives a triple $(X,X(q),C)$ as above exhibiting the CSP.

Although this $X(q)$ equals $\Kirk_0(W,k;q)$ whenever $k=0$ or
$k=n$, one can check that, starting with $n=5$ and $k=2,3$,
they are $X(q) \not\equiv \Kirk_0(W,k;q) \bmod{q^{h+2}-1}$.
In fact, $X(q)$ and $\Kirk_0(W,k;q)$ disagree at $q=-1$.
Thus if one assumes the result \cite[Thm. 5.1 at $s=1$]{EuFu} is
correct, which we have not checked, then
one reaches the somewhat surprising conclusion
that $\Kirk_0(W,k;q)$ fails to give a CSP. 

\vskip.1in
\noindent
{\sf Type E}.
We did not check $E_7, E_8$.  However,
in type $E_6$, the data given by Eu and Fu in \cite[Figure 15]{EuFu} 
on the orbit sizes for the action of $\tau$
agrees with the CSP prediction of  $\Kirk_0(W,k;q)$ for
$k=0,1,3,5,6$, but fails to give a CSP for $k=2,4$.
\end{remark}

\section{Inspiration: Nonnesting parking functions label Shi regions}
\label{Shi-chamber-section}

The original motivation for this paper came from the work of the first and third author on the Shi arrangement and the Ish arrangement of hyperplanes \cite{ArmstrongRhoades}.  In type A there is a natural labeling of the regions of the Shi arrangement due to Athanasiadis and Linusson \cite{ALShi} (which is a variant of a description due to Shi \cite{Shi-sign-types}) and these labels are essentially the nonnesting parking functions. Upon generalizing this labeling to other Weyl groups, the authors realized that the same process could be used to define {\bf noncrossing parking functions} and that this notion extends beyond crystallographic types. 

As motivation, we explain here how (for crystallographic $W$) the set of nonnesting parking functions $\Park^{NN}_W$ naturally labels the regions of the Shi arrangement.

Let $W$ be a (by definition crystallographic) Weyl group with positive roots $\Phi^+$ inside the root system $\Phi \subseteq V$,
and for each $\alpha\in\Phi$ and $k\in\RR$ define an affine hyperplane
\begin{equation}
H_{\alpha,k} := \{ v \in V \,:\, \langle v, \alpha \rangle = k \}.
\end{equation}
As a special case one has the linear hyperplanes $H_{\alpha,0} = H_{\alpha}$
from the Coxeter arrangement $\Cox(\Phi)$.  

\begin{defn}
The {\sf Shi arrangement} 
$\Shi(\Phi)$ is the arrangement of affine hyperplanes
\begin{equation}
\Shi(\Phi) := \{ H_{\alpha,k} \,:\, \alpha \in \Phi^+, k = 0,1 \}.
\end{equation}
\end{defn}

This arrangement arose originally in work of Shi \cite{Shi} 
on Kazhdan-Lusztig cells for affine Weyl groups.
He later showed in \cite{Shi-sign-types} that $\Shi(\Phi)$ dissects $V$ 
into $(h+1)^n$ connected components $R$,
which we will call {\it regions};  these correspond to Shi's {\it admissible
sign types}.
Given a region $R$ of $\Shi(\Phi)$, and $\alpha \in \Phi^+$,
say that the hyperplane $H_{\alpha,1}$ is a {\sf ceiling} of $R$ 
if the inequality $ \langle v, \alpha \rangle \leq 1$
is one of the irredundant facet inequalities defining its closure $\overline{R}$ as a
(sometimes unbounded) polyhedron.

Note that since  $\Cox(\Phi) \subseteq \Shi(\Phi)$,
each region $R$ lies in a unique chamber $\Cox(\Phi)$, of the form $wC$
for a uniquely defined $w$ in $W$, 
where $C$ is the {\sf (open) dominant region} of $\Cox(\Phi)$ defined by
$\langle v, \alpha \rangle > 0$ for all $\alpha \in \Phi^+$.
The Shi regions $R$ lying within $C$ correspond to
the {\it $\oplus$-sign types} from \cite{Shi-plus-sign-types}.

The following lemma about the interaction of $\Cox(\Phi)$ and $\Shi(\Phi)$
is surely well-known, but we include the easy proof for the sake of completeness.
\begin{lemma}
\label{Shi-Cox-fact}
For $\beta \in \Phi^+$ a positive root,
its associated affine Shi hyperplane $H_\beta,1$ intersects the (open) chamber $wC$ 
if and only if $\beta = w(\alpha)$ for some $\alpha \in \Phi^+$.
\end{lemma}
\begin{proof}
Fix $\beta$ in $\Phi^+$.
If $\beta = w(\alpha)$ for $\alpha$ in $\Phi^+$,
then choosing $v$ in $C$ with $\langle v,\alpha  \rangle =1$
gives $v':=w(v)$ in $wC$ with
\begin{equation}
\label{affine-hyperplane-calculation}
\langle v', \beta \rangle
=\langle w(v), w(\alpha) \rangle 
= \langle v,\alpha  \rangle 
=1.
\end{equation}
So $H_{\beta,1}$ intersects $wC$.

Conversely, if $H_{\beta,1}$ intersects $wC$, then there exists some $v'$ in $wC$ with
$\langle v', \beta \rangle=1$.  The same calculation \eqref{affine-hyperplane-calculation}
shows that the point $v:=w^{-1}(v')$ lying in $C$, and the root $\alpha \in \Phi$
for which $\beta=w(\alpha)$,
have pairing $\langle v,\alpha  \rangle =1 > 0$.  On the other hand, since
$v$ lies in $C$, it pairs positively with positive roots and negatively with
negative roots.  Thus $\alpha$ must be a positive root, i.e. $\beta=w(\alpha)$
for some $\alpha \in \Phi^+$.
\end{proof}

\begin{proposition}
The labeling map $\lambda: R \longmapsto [w,X]$, where $R \subseteq wC$ and 
$$
X:=\bigcap_{\substack{\text{ceilings }H_{\alpha,1} \\ \text{ of }R}} H_{w^{-1}(\alpha),0},
$$
gives a well-defined bijection from the regions of $\Shi(\Phi)$ to $\Park^{NN}_{\Phi}$.
\end{proposition}

The poset of positive roots $\Phi^+$, along with
the labeling of regions $R$ by $[w,X]$
for $\Shi(A_2), \Shi(B_2)$, are illustrated below.  Here
the $\Cox(\Phi)$ chamber $wC$ containing $R$
is labeled toward the periphery, with the subset of roots
\begin{equation}
\label{R-antichain}
A:=\{ w^{-1}(\alpha) : \text{ ceilings } H_{\alpha,1} \\ \text{ of }R \}
\end{equation}
whose normals intersect in $X$ labeled in the interior of $R$.  This
labeling should be compared with Shi \cite[Figures 1 and 2]{Shi-sign-types}.

\begin{center}
\includegraphics[scale=.8]{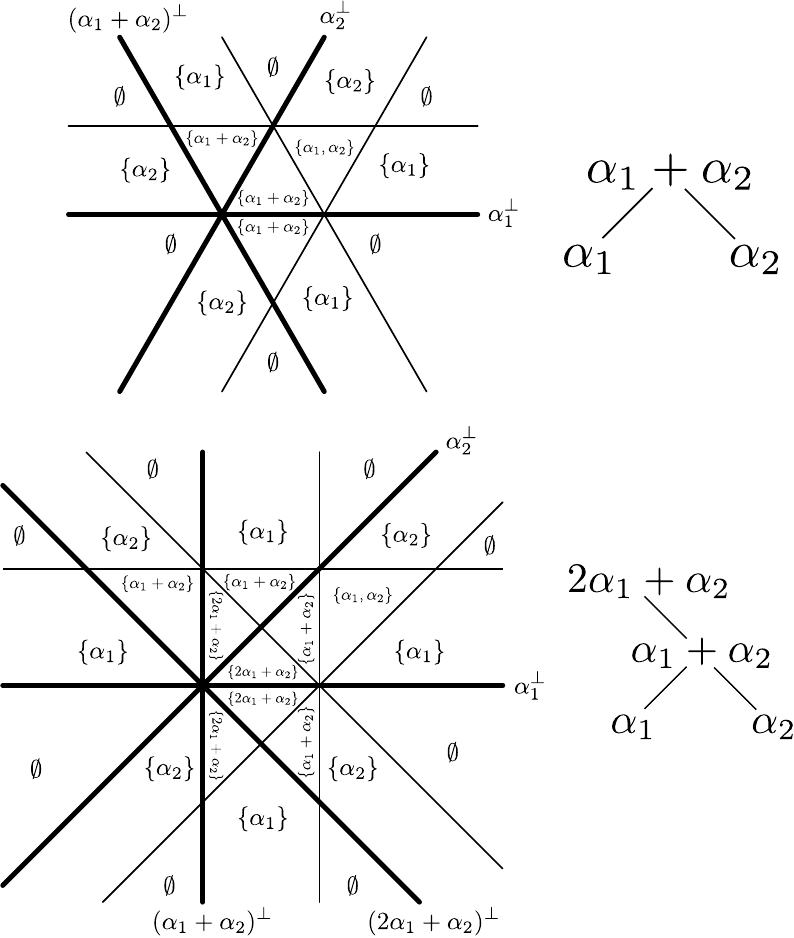}
\end{center}

\begin{proof}
We first check that $\lambda$ is well-defined.  Note that when $R \subseteq wC$,
the subset $A \subseteq \Phi$ defined in \eqref{R-antichain}
actually lies in $\Phi^+$, due to Lemma~\ref{Shi-Cox-fact}:
a ceiling hyperplane $H_{\alpha,1}$ for $R \subseteq wC$ has
$H_{\alpha,1}$ intersecting $wC$, so $w^{-1}(\alpha)$ lies in $\Phi^+$.

Well-definition also requires checking that
$A$ forms an antichain in the root ordering on $\Phi^+$.  Suppose, for the
sake of contradiction, that $A$ contains two positive roots 
$w^{-1}(\alpha) < w^{-1}(\beta)$ in $\Phi^+$ such
that both $H_{\alpha,1}, H_{\beta,1}$ are ceilings for $R$.  
Note that $\beta-\alpha$ has
$w^{-1}(\beta-\alpha)=w^{-1}(\beta) - w^{-1}(\alpha)$,
a nonnegative sum of positive roots.
Therefore 
$
\langle v,w^{-1}(\beta-\alpha) \rangle \geq 0
$ 
is a valid inequality for $v$ in the closure $\overline{C}$ of the fundamental chamber.
Then  
\begin{equation}
\label{valid-chamber-inequality}
\langle v', \beta-\alpha \rangle \geq 0
\end{equation}
is valid for $v'$ in the closed chamber $\overline{wC}$, and 
hence also on the closed region $\overline{R}$.  But then the valid
inequality $\langle v', \alpha \rangle \leq 1$ on $\overline{R}$
is already implied by \eqref{valid-chamber-inequality}
together with the valid inequality $\langle v', \beta \rangle \leq 0$
on $\overline{R}$.  Thus one cannot have both $H_{\alpha,1}$ and
$H_{\beta,1}$ as ceilings for $R$.

Since one knows that the number of regions $R$ in $\Shi(\Phi)$
is the same as the cardinality of $\Park^{NN}_W$, namely
$(h+1)^n$, it only remains to show that the map 
$\lambda: R \mapsto [w,X]$ defined above is injective.
To this end, define a map backward $\mu: [w,X] \mapsto R$
as follows.  Given an equivalence class $[w,X]$ in $\Park^{NN}_W$, 
pick the unique coset representative $w$ for $wW_X$
having $w(\alpha) \in \Phi^+$ for all $\alpha$ in $\Phi^+$ 
with $H_{\alpha,0} \supset X$.  Then $\mu$ maps
$[w,X]$ to the intersection $R$ of the chamber $wC$ with 
the open half-spaces defined by $\langle v', w(\alpha) \rangle < 1$ as $\alpha$ 
runs through the (unique) 
antichain $A$ of positive roots having $\bigcap_{\alpha \in A} H_{\alpha,0}=X$.

To see that $\mu \circ \lambda$ is the identity map on Shi regions,
we first claim that $\mu(\lambda(R))$ and $R$ lie in the same chamber $wC$.
This requires a fact due to Sommers \cite[\S2]{Sommers-B-stable}:  
an antichain $A$ in the positive root poset $\Phi^+$ is always the $W$-image of
a subset of some choice of {\it simple roots} for $W$.  This implies that
the antichain $A$ of positive roots defined by $R$ as in \eqref{R-antichain} above
will always form a system of simple roots for the subgroup $W_X$.
Since Lemma~\ref{Shi-Cox-fact} implies $w(\alpha) \in \Phi^+$
for all $\alpha$ in $A$, this then implies that $w(\alpha) \in \Phi^+$ for
all $\alpha$ in $\Phi^+$ with $H_\alpha \supset X$.  Thus if $R \subseteq wC$
and $\lambda(R)=[w,X]$, then $w$ {\it is} always the unique coset representative 
for $wW_X$ chosen by the map $\mu$ when applied to $\lambda(R)$.

Finally, note that a Shi region $R$ is completely determined by knowing
which chamber $wC$ contains it along with its set of ceilings $H_{\alpha,1}$, 
so $\mu(\lambda(R))=R$.
\end{proof}

\section{Open Problems}
\label{open-problems-section}

\subsection{Two basic problems}

\begin{problem}
Prove the strong or intermediate version of the Main Conjecture
in a case-free fashion.

Short of this, check that the weakest version of the Main Conjecture
holds via computer calculation in the remaining types $E_7, E_8$.  
\end{problem}

\begin{problem}
Extend the Main Conjecture to well-generated complex reflection groups,
or even all complex reflection groups.
\end{problem}

\noindent
In this regard the authors do not know, for example, how far
the statement and proof of Etingof's Theorem~\ref{Etingof-reduced} below 
generalizes.

\subsection{Nilpotent orbits, $q$-Kreweras and $q$-Narayana numbers}

Recall from the Introduction that the Kirkman numbers and 
Narayana numbers for $W$
give the $f$-vector and $h$-vector of the Fomin-Zelevinsky cluster complex.
In Section  \ref{Narayana-Kirkman-section} we related them to
the multiplicities of the $W$-irreducible exterior powers $\wedge^k V$ 
in the $W$-parking space, and the graded multiplicities allowed one
to define $q$-Kirkman numbers.  

In recent work on geometry of nilpotent orbits, 
E. Sommers \cite{Sommers-Narayana} has suggested how to
define $q$-analogues for Weyl groups $W$ of the {\it Kreweras numbers}, which
in type $A$ count elements of $NC(W)$ according to their parabolic
type;  they are the polynomials
$\left[ f_{e,\phi}(q,t) \right]_{t=h+1}$,
in the notation of \cite[\S 5.3]{Sommers-Narayana}.  
Sommers and the second author have observed that these can be grouped into 
$q$-analogues of {\it Narayana numbers}.

\begin{problem}
What is the relation, if any, between these $q$-Narayana numbers 
suggested by Sommers's work, and our $q$-Kirkman numbers $\Kirk(W,k;q)$?
Is there a well-behaved $q$-analogue of the $f$-vector to $h$-vector
transformation\footnote{While this article was under review, an affirmative answer to this question was found by Sommers and the second author (in preparation).}?
\end{problem}

\subsection{The Fuss parameter}

We discuss here some problems in the
known direction of generalization for Catalan
objects that replaces the parameter $p=h+1$
with its {\it Fuss} analogue $p=mh+1$, or
more generally, with a parameter $p$ assumed to 
satisfy $\gcd(p,h)=1$ (or weakenings of this assumption).  
Under these hypotheses, the ungraded and graded
virtual $W$-characters 
$$
\chi(w)=p^{\dim V^w}
\qquad \text{ and } \qquad
\chi(w; q)=
\frac{ \det(1-q^{p}w) }{  \det(1 - qw)    } \\
$$
actually come from genuine $W$-representations that one might call
{\it $p$-parking spaces}.
The ungraded character corresponds to
the $W$-action permuting $Q/pQ$, as observed by 
Haiman \cite[Prop. 7.4.1]{Haiman}.  The
graded character corresponds (via Koszul complex calculations as
in  \cite[Theorem 1.11]{BEG}, \cite[\S 4]{BessisR}, \cite[\S 5]{Gordon})
to the quotient $\CC[V]/(\Theta)$
where $\Theta$ is a homogeneous system of parameters of degree $p$ carrying
a certain Galois twist of the reflection representation $V$,
and whose existence is provided by the rational Cherednik theory 
(see, e.g., Etingof \cite{Etingof}).  

One can start with product formulas for the
cases of \eqref{GNS-generating-function} corresponding to the trivial
and determinant characters (calculated as usual via 
Solomon's theorem \cite{Solomon})
\begin{equation}
\label{GNS-boundary-case-formulas}
\tilde{\tau}_W(\triv;q,u)=\prod_{i=1}^n \frac{1+q^{e_i}u}{1-q^{e_i+1}} 
\qquad \text{ and } \qquad
\tilde{\tau}_W(\det;q,u)=\prod_{i=1}^n \frac{u+q^{e_i}}{1-q^{e_i+1}}.
\end{equation}
From these one deduces, by setting $u=-q^p$, that the 
Hilbert series of the $W$-fixed spaces and
the $W$-det-isotypic spaces within 
$\CC[V]/(\Theta)$ have Hilbert series
$$
\Cat^{(p)}(W;q):=
\prod_{i=1}^n \frac{[e_i+p]_q}{[e_i+1]_q} 
\qquad \text{ and } \qquad
q^{|\Phi^+|} \prod_{i=1}^n \frac{[p-e_i]_q}{[e_i+1]_q},
$$
generalizing the extreme cases \eqref{det-and-triv-q-Kirkman}
of the $q$-Kirkman number formulas.
The $q=1$ specializations then give \cite[Thm. 7.4.2]{Haiman} the
total number of $W$-orbits and the number of $W$-regular
orbits on $Q/pQ$.  
For $p=mh+1$, one has that $\Cat^{(mh+1)}(W;q)$ is the 
$q$-{\it Fuss-Catalan number}, whose $q=1$ specialization
counts $m$-element multichains in the 
noncrossing partitions $NC(W)$, 
dominant regions in the $m$-extended Shi arrangement, and 
facets in the $m$-generalized cluster complex; see \cite[Chapter 5]{Armstrong}.

One can again define Narayana and Kirkman polynomials with this parameter $p$,
which have interpretations when $p=mh+1$.  For example, one can define Narayana
numbers at parameter $p$ that count the $W$-orbits $x$ in $Q/pQ$ according to the parabolic 
corank\footnote{In the usual case $p=h+1$, it would be immaterial whether
one uses parabolic rank or corank, as the Narayana polynomial will have symmetric coefficient sequence
in $t$.} of their $W$-stabilizer $W_x$:
 \begin{equation}
 \label{eq:hpoly}
\sum_{x \in W\backslash Q/pQ} t^{n-\rank(W_x)} =
 \frac{1}{|W|} \sum_{w\in W} \det(t+(1-t)w) \cdot p^{\dim V^w} 
 \end{equation}
One finds in the crystallographic case that when $p=mh+1$, 
this is the $h$-{\it polynomial} of the $m$-generalized cluster complex of Fomin and Reading \cite{FominReading},
that is, replacing $t$ by $t+1$ in  \eqref{eq:hpoly}
gives the generating function for the face numbers of this simplicial complex.

\begin{problem}
Given an irreducible real reflection group $W$ and
positive integer $p$ with $\gcd(p,h)=1$, find
definitions of noncrossing and nonnesting partitions and noncrossing
parking functions, and generalize the Main Conjecture to this context.
\end{problem}

In the `Fuss case' $p = mh + 1$, the third author defined Fuss analogs 
$\Park^{NC}_W(m)$ and $\Park^{alg}_W(m)$ of $\Park^{NC}_W$ and $\Park^{alg}_W$,
as well as a Fuss 
analog of the nonnesting parking space when $W$ is crystallographic \cite{Rhoades}.
Fuss analogs of the strong, intermediate, and weak versions of the Main Conjecture 
are presented and are proven in nearly the generality of the $m = 1$ case.
On the other hand, when $\gcd(p,h)=1$ but $p \not\equiv 1 \bmod{h}$, one
encounters the complication that the hsop guaranteed by the rational Cherednik algebra theory need not carry the representation $V^*$, but rather 
some Galois conjugate of this representation.

\subsection{The near boundary cases of Kirkman numbers}

As mentioned above, the product formulas \eqref{GNS-boundary-case-formulas} 
for \eqref{GNS-generating-function} when $\chi$ is either
the character of the trivial and determinant
representations $\wedge^0 V, \wedge^n V$ of $W$ yield
the product formulas \eqref{det-and-triv-q-Kirkman}
for the boundary cases $k=0,n$ of the $q$-Kirkman numbers $\Kirk(W,k;q)$,
and even allowing the general parameter $p$, not just $p=h+1$.

We discuss here a conjectural\footnote{While this article was under review, this conjecture (Conjecture~\ref{near-boundary-Kirkman-conjecture} or 
\ref{near-boundary-Kirkman-conjecture}' below) was resolved affirmatively in joint work of A. Shepler and the second author (in preparation).} formula for 
\eqref{GNS-generating-function} in the {\it near-boundary}
cases where $\chi$ comes either from $\wedge^1 V=V$ or 
$\wedge^{n-1}V \cong \det \otimes V$, along
with its consequences for the ($q$-)Kirkman numbers $\Kirk(W,k;q)$, 
even with the general parameter $p$.
For the sake of stating the conjecture, define
a different $q$-analogue of the rank $n$ by
$$
n_q:=\sum_{i=1}^n q^{e_i-1}.
$$ 
In other words, $n_q$ is 
the generating function for the {\it codegrees} of $W$.

\begin{conjecture}
\label{near-boundary-Kirkman-conjecture}
For a real reflection group $W$ acting irreducibly on $V=\CC^n$, with
exponents $(e_1,\ldots,e_n)$ indexed so that the Coxeter number $h=e_n+1$,
one has the following two equivalent\footnote{The equivalence of the 
two formulas comes from the observation (see \cite[eqn. (1.24)]{GNS})
that $\tilde{\tau}(\det \cdot \chi;q,u)=u^n \, \tilde{\tau}(\chi;q,u^{-1})$,
combined with $V^* \cong V$ and 
$\wedge^{n-1} V \cong V^* \otimes \wedge^n V \cong \det \otimes V$.}
 formulas:
$$
\begin{aligned}
\tilde{\tau}(\chi_V;q,u)
&=n_q \cdot \frac{q+u}{1+q^{h-1}u} \cdot
   \prod_{i=1}^n \frac{1+q^{e_i}u}{1-q^{e_i+1}} \\
\tilde{\tau}(\chi_{\wedge^{n-1}V};q,u)
&= n_q \cdot \frac{1+qu}{u+q^{h-1}} \cdot
   \prod_{i=1}^n \frac{u+q^{e_i} }{ 1-q^{e_i+1} }.
\end{aligned}
$$
\end{conjecture}
\noindent
Setting $u=-q^p$ for a positive integer $p$ would imply that the graded
character 
$\chi(w;q)=\frac{ \det(1-q^{p}w) }{  \det(1 - qw)    }$
has its graded multiplicities of
the irreducibles $V$ and $\wedge^{n-1}V$ given by, respectively,
\begin{equation}
\label{near-boundary-p-Fuss-graded-chars}
\begin{aligned}
q^1 \cdot & n_q \cdot \frac{ [p-1]_q}{ [p+h-1]_q } \cdot \Cat^{(p)}(W;q)\\
q^{|\Phi^+|-h+1} \cdot & n_q \cdot \frac{ [p+1]_q}{ [p-h+1]_q } \cdot \prod_{i=1}^{n}\frac{ [p-e_i]_q }{ [e_i+1]_q }.
\end{aligned}
\end{equation}
This would have two interesting consequences. 

Firstly, setting $p=h+1$ in \eqref{near-boundary-p-Fuss-graded-chars} would give
\begin{equation}
\label{near-boundary-q-Kirkmans}
\begin{aligned}
\Kirk(W,1,q) &=q^1 \cdot n_q \cdot \frac{ [h]_q}{ [2h]_q } \cdot \Cat(W;q) \\
\Kirk(W,n-1,q) &= q^{|\Phi^+|-h+1} \cdot n_q \cdot \frac{ [h+2]_q}{ [2]_q } 
\end{aligned}
\end{equation}
where the second equation used the fact that $h-e_i=e_{n-i}$.  

Secondly, equation \eqref{near-boundary-p-Fuss-graded-chars}
at $q=1$ says that the ratio of the multiplicities of the
reflection and trivial characters in
the character $\chi(w)=p^{\dim V^w}$ is 
$
\frac{n(p-1)}{p+h-1}.
$
This turns out to be equivalent, by taking $\frac{d}{dt}$ in
\eqref{eq:hpoly}, to the following curious statement:
a $W$-orbit $x$ in $W \backslash Q/pQ$ chosen uniformly at
random has the expected value for the parabolic corank
of the $W$-stabilizer subgroup $W_x$ equal to
$
\frac{n(p-1)}{p+h-1},
$
or equivalently, its expected rank is 
$
\frac{nh}{p+h-1}.
$

Lastly, we mention the algebraic interpretation of 
Conjecture~\ref{near-boundary-Kirkman-conjecture}.
Denote the symmetric and exterior algebras of $V^*$ by
$S=\CC[V]=\Sym(V^*)$ and $\wedge=\wedge V^*$.
Their tensor product $S \otimes V$
becomes a $W$-representation, bigraded by polynomial and exterior degree.
Its $W$-intertwiner space with $V^*$ 
$$
M:=\Hom_W( V^*, S \otimes \wedge ) \cong (S \otimes \wedge \otimes V)^W
$$
will have bigraded Hilbert series given by the left side of the conjecture,
using the variables, $q,u$ for the polynomial, exterior gradings, respectively.

Furthermore, $S \otimes \wedge$ becomes a bigraded module over
the invariant subalgebra $S^W$, via multiplication in the left tensor factor $S$.
Since $S^W=\CC[f_1,\ldots,f_n]$ is a polynomial algebra, $S$ will be a 
free $S^W$-module.  Hence $S \otimes \wedge$ is also $S^W$-free,
and the same holds for $M$.
Conjecture~\ref{near-boundary-Kirkman-conjecture} 
then predicts the bidegrees for 
any choice of bihomogeneous $S^W$-basis elements of $M$, or equivalently, 
the Hilbert series of the quotient $M/S^W_+ M$.  
Using the fact that $1/\Hilb(S^W,q) = \prod_{i=1}^{\ell} (1-q^{e_i+1})$,
Conjecture~\ref{near-boundary-Kirkman-conjecture} 
becomes equivalent to the following.

\vskip.1in
\noindent
{\bf Conjecture~\ref{near-boundary-Kirkman-conjecture}${}^\prime$.}
$$
\Hilb(\,\, M/S^W_+M \,\, ; q,u) \,\, = \,\, 
n_q \cdot (q+u) \cdot \prod_{i=1}^{\ell-1} (1+q^{e_i}u).
$$

\section{Appendix:  Etingof's Proof of Reducedness for a Certain hsop}

Let $W$ be an irreducible real reflection group of rank $n$, with 
(complexified) reflection representation $V$, Coxeter number $h$,
and let $p$ be any positive integer coprime to $h$.
We present here a uniform argument due to P. Etingof \cite{EtingofComm} 
showing the existence of a hsop of degree $p$ carrying $V^*$
which satisfies the reducedness condition of the strong or intermediate
versions of the Main Conjecture from \S\ref{main-conjecture-subsection}.

\begin{theorem} \label{Etingof-reduced} (Etingof)
There exists an hsop $\Theta = (\theta_1, \dots, \theta_n)$ of degree $p$ carrying $V^*$
such that the subvariety $V^{\Theta} \subseteq V$ consists of $p^n$ distinct points.
\end{theorem}

The choice of hsop in Theorem~\ref{Etingof-reduced} arises from the theory of 
rational Cherednik algebras, and to explain it requires some notation.  
Let $\langle -, - \rangle$ be a $W$-invariant positive definite form on $V$, 
unique up to scaling.  For any $v \in V$, let $v^{\vee} \in V^*$ 
be the linear functional given by
$\langle -, v \rangle: V \rightarrow \CC$.
Also let $\partial_v : \CC[V] \rightarrow \CC[V]$ be the partial derivative operator 
in the direction of $v$.  
The {\sf Dunkl operator} $D_v : \CC[V] \rightarrow \CC[V]$ with parameter $\frac{p}{h}$
is given by
\begin{equation}
D_v := \partial_v - \frac{p}{h} \sum_{t \in T} \frac{\langle \alpha_t, v \rangle}{\alpha_t^{\vee}}(1-t),
\end{equation}
where $\{ \alpha_t \,:\, t \in T \}$ is the root system associated to $W$, normalized to satisfy
$\langle \alpha_t, \alpha_t \rangle = 2$ for all $t$ in $T$.  For any $t$ in $T$ 
and  $f \in \CC[V]$, one has $(1-t).f$ divisible by $\alpha_t^{\vee}$ in $\CC[V]$, so the 
operator $D_v$ maps polynomials to polynomials.

It follows from Gordon's work\footnote{Although there are subtleties to this story realized only later than \cite{Gordon}-- see Gordon and Griffeth \cite[\S 1.5, 2.11, 2.12]{GordonGriffeth} for how this is circumvented via results of Malle and of Rouquier.} on the rational Cherednik algebras \cite{Gordon} that there exists
a $W$-equivariant injective linear map $\widehat{\Theta}: V^* \hookrightarrow \CC[V]_p$ whose
image $\widehat{\Theta}(V^*)$
\begin{enumerate}
\item[(i)]
is annihilated by all Dunkl operators $\{ D_b \,:\, b \in V \}$, and
\item[(ii)]
generates an ideal $I$ in $\CC[V]$ whose quotient $\CC[V]/I$ carries
a $p^n$-dimensional simple module for the rational Cherednik algebra
at parameter $\frac{p}{h}$.
\end{enumerate}
In particular, since the quotient  $\CC[V]/I$ is finite-dimensional,
for any basis $x_1,\ldots,x_n$ of $V^*$, the elements $\theta_i=\widehat{\Theta}(x_i)$ for
$i=1,2,\ldots,n$ form an hsop  $\Theta=(\theta_1,\ldots,\theta_n)$ 
as in the setup \eqref{strong-hsop-setup} for the Main Conjecture.
One then has an associated degree $p$ polynomial map $\Theta: V \rightarrow V$ 
characterized by
$
a(\Theta(v)) = (\widehat{\Theta}(a))(v) 
$ 
for $v \in V, a \in V^*$, with fixed point locus 
$$
V^{\Theta}=\{x \in V \,:\, \Theta(x) = x \} 
=\{x \in V \,: a(x) = (\widehat{\Theta}(a))(x) \text{ for all }a \in V^* \}.
$$

\begin{proof}[Proof of Theorem~\ref{Etingof-reduced}]
To avoid trivialities, assume $p \neq 1$ without loss of generality.
Let $x \in V^{\Theta}$ and let $d \Theta_x : V \rightarrow V$ be the linear map given by the 
differential of the polynomial map $\Theta$ at $x$.  Since the generators $\theta_i-x_i$ for
the ideal $(\Theta-\xx)$ that cuts out $V^\Theta$ have Jacobian $d\Theta_x - 1_V$,
it is enough to show that $d \Theta_x$ does not have the eigenvalue $1$.  
To do this, we compute the map $d \Theta_x$ explicitly.
We will make frequent use of the 
fact that for $a, b, x \in V$ one has  
\begin{equation}
\label{partial-versus-d-fact}
\begin{aligned}
\langle a, d \Theta_x(b) \rangle 
= (\partial_b \widehat{\Theta}(a^{\vee})) (x) 
&=\frac{p}{h} \sum_{t \in T} \langle \alpha_t, b \rangle
\left( \frac{(1-t)\widehat{\Theta}(a^{\vee})}{\alpha_t^{\vee}} \right)(x)
\\
&=\frac{p}{h} \sum_{t \in T} \langle \alpha_t, b \rangle \langle \alpha_t, a \rangle
\left( \frac{\widehat{\Theta}(\alpha_t^{\vee})}{\alpha_t^{\vee}} \right)(x)
\end{aligned}
\end{equation}
where the second-to-last equality used the fact that $\widehat{\Theta}$ is annihilated by the
Dunkl operator $D_b$, and the last equality derives from $\widehat{\Theta}$ being 
linear and $W$-equivariant.

Given $x$ in $V^\Theta$, let $W_x \subseteq W$ be its isotropy subgroup,
which is itself a reflection group having reflection representation 
$V / V^{W_x}$.  Decompose 
$$
W_x = W_1 \times \dots \times W_m
$$
as a product of irreducible reflection groups.  
For $1 \leq i \leq m$, let $T_i \subseteq T$ be the reflections in $W_i$ and let
$T_0 = T - (T_1 \cup \dots \cup T_m)$, giving a disjoint union of sets
$$
T = T_0 \sqcup T_1 \sqcup \dots \sqcup T_m.
$$
One also has a $W_x$-stable orthogonal direct sum decomposition
$$
V = V_0 \oplus V_1 \oplus \dots \oplus V_m
$$ 
in which $V_0 = V^{W_x}$, and $V_i$ is the (complexified) 
irreducible reflection representation of $W_i$ for $1 \leq i \leq m$.  
Since $d \Theta_x$ is a $W_x$-invariant 
operator, it preserves these direct summands.  
In fact, we will show that this is an 
eigenspace decomposition:  since $W_i$ acts irreducibly on $V_i$ for $i=1,2,\ldots,m$,
by Schur's Lemma, the restriction $d \Theta_x|_{V_i}$ acts as a scalar $c_i$, 
(to be determined in Case 2 below), and we will show in Case 1 below that 
$d \Theta_x|_{V_0}$ acts as the scalar $p > 1$.  The following fact will be useful.

\begin{lemma}
\label{quotient-evaluation-fact}
For any $x$ in $V^\Theta$ and $t$ in $T$, one has
\begin{equation}
\left( \frac{\widehat{\Theta}(\alpha_t^{\vee})}{\alpha_t^{\vee}}\right)(x)
=\begin{cases}
1 & \text{ if } t \in T_0,\\
c_i &  \text{ if } t \in T_i \text{ with }i\in \{1,2,\ldots,m\}.
\end{cases}
\end{equation}
\end{lemma}
\begin{proof}
Since $x$ lies in $V^\Theta$, one has 
$
\widehat{\Theta}(\alpha_t^{\vee})(x) = \alpha_t^{\vee}(\Theta(x))=\alpha_t^{\vee}(x).
$
Thus the first case in the lemma follows, as $\alpha_t^{\vee}(x) \neq 0$ if $t$ lies in $T_0$, 
since such a $t$ does not lie in $W_x$.  To derive the second case, write
$\widehat{\Theta}(\alpha_t^{\vee})=\alpha_t^\vee \cdot f$ with $f$ in $\CC[V]$.  Then one has
$$
\begin{aligned}
\partial_{\alpha_t} \widehat{\Theta}(\alpha_t^{\vee})
&=\partial_{\alpha_t} ( \alpha_t^\vee \cdot f) \\
&=\partial_{\alpha_t}(\alpha_t^{\vee}) \cdot f 
+ \alpha_t^{\vee} \cdot \partial_{\alpha_t}(f) \\
&= 2f + \alpha_t^{\vee} \cdot \partial_{\alpha_t}(f).
\end{aligned} 
$$
Hence 
$
\partial_{\alpha_t} \widehat{\Theta}(\alpha_t^{\vee}(x))=2f(x),
$
as $\alpha_t^{\vee}(x)=0$ when $t$ lies in $T_i$ for $i \geq 1$, and therefore 
\begin{equation}
\label{c-as-partial}
\begin{aligned}
c_i &=  \frac{1}{2} \langle \alpha_t, c_i \alpha_t \rangle 
=\frac{1}{2} \langle \alpha_t, d\Theta_x (\alpha_t) \rangle 
=\frac{1}{2} \partial_{\alpha_t} \widehat{\Theta}(\alpha_t^{\vee})(x) \\
&=\frac{1}{2} (2f(x)) 
=f(x)
= \frac{\widehat{\Theta}(\alpha_t^{\vee})}{\alpha_t^{\vee}}(x).
\end{aligned}
\end{equation}
\end{proof}

We can now determine the operators $d \Theta_x|_{V_i}$ for $i=0,1,2,\ldots,m$.
\vskip.1in
\noindent
{\sf Case 1.} $i = 0$.  We claim that for any $a, b \in V_0$, one has
$
\langle a, d \Theta_x(b) \rangle
=p \cdot \langle a, b \rangle
$
and hence that $d \Theta_x|_{V_0}$ acts as the scalar $p>1$.
To see this, start with \eqref{partial-versus-d-fact}:
\begin{equation}
\label{case-one-expression}
\begin{aligned}
\langle a, d \Theta_x(b) \rangle 
&=\frac{p}{h} \sum_{t \in T} \langle \alpha_t, b \rangle \langle \alpha_t, a \rangle
\left( \frac{\widehat{\Theta}(\alpha_t^{\vee})}{\alpha_t^{\vee}} \right)(x) \\
&=\frac{p}{h} \sum_{t \in T_0} \langle \alpha_t, b \rangle \langle \alpha_t, a \rangle
\left( \frac{\widehat{\Theta}(\alpha_t^{\vee})}{\alpha_t^{\vee}} \right)(x)\\
&=\frac{p}{h} \sum_{t \in T_0} \langle \alpha_t, b \rangle \langle \alpha_t, a \rangle
=\frac{p}{h} \sum_{t \in T} \langle \alpha_t, b \rangle \langle \alpha_t, a \rangle.
\end{aligned}
\end{equation}
Here the second and fourth equality used the fact that any
$t$ in $T \setminus T_0=\cup_{i=1}^m T_i$ fixes $a$ in $V_0$, and hence
has $\langle \alpha_t, a \rangle = 0$, while the third equality used 
Lemma~\ref{quotient-evaluation-fact}.
The claim then follows from this lemma.

\begin{lemma}
\label{Etingof's-sum-lemma}
For $W$ an irreducible real reflection group with Coxeter number $h$ and reflections $T$
acting on $V$,  any $a, b \in V$ will satisfy
\begin{equation}
\label{reflection-sum}
\sum_{t \in T} \langle a, \alpha_t \rangle \langle b, \alpha_t \rangle = h \langle a, b \rangle.
\end{equation}
\end{lemma}
\begin{proof}
The left side is a $W$-invariant bilinear function of 
$a,b$, and hence equal to some scalar multiple $\lambda$ of $\langle a, b \rangle$.
Letting $a = b = e_i$ for any orthonormal basis $e_1, \dots, e_n$ of $V$ and 
summing over $i$, one has 
$
\sum_{t \in T} \langle \alpha_t, \alpha_t \rangle = \lambda n,
$
so that $2 |T| = \lambda n$.  However, it is well-known \cite[\S3.18]{Humphreys} 
that $2 |T| = hn$, so $\lambda = h$.
\end{proof}

\vskip.1in
\noindent
{\sf Case 2.} $i$ lies in $\{1,2,\ldots,m\}$.
We will compute the scalar $c_i$ by which  $d \Theta_x|_{V_i}$ acts.
Assume $a,b$ lie in $V_i$, and start with \eqref{partial-versus-d-fact}:
\begin{equation}
\label{case-two-identity}
\begin{aligned}
\langle a, d \Theta_x(b) \rangle 
&= 
\frac{p}{h} \sum_{t \in T}
\langle \alpha_t, a \rangle \langle \alpha_t, b \rangle 
\frac{\widehat{\Theta}(\alpha_t^{\vee})}{\alpha_t^{\vee}}(x)\\
&= 
\frac{p}{h} 
\left( 
\sum_{t \in T}
\langle \alpha_t, a \rangle \langle \alpha_t, b \rangle 
+
\sum_{t \in T}
\langle \alpha_t, a \rangle \langle \alpha_t, b \rangle 
\left( \frac{\widehat{\Theta}(\alpha_t^{\vee})}{\alpha_t^{\vee}}(x)-1 \right)
\right)
\\
&= 
\frac{p}{h} 
\left( h \langle a,b \rangle +
(c_i-1) \sum_{t \in T_i}
\langle \alpha_t, a \rangle \langle \alpha_t, b \rangle 
\right)\\
&= 
\frac{p}{h} 
\left( h \langle a,b \rangle +
(c_i-1) h_i \langle a , b \rangle 
\right)
\end{aligned}
\end{equation}
where $h_i$ is the Coxeter number for $W_i$.
Here the third equality comes from applying Lemma~\ref{Etingof's-sum-lemma} to $W$
in the first sum, and  using  Lemma~\ref{quotient-evaluation-fact} in the second sum,
along with the fact that $\langle \alpha_t,a \rangle=0$ if $t \not\in T_0 \cup T_i$.
The fourth equality applies Lemma~\ref{Etingof's-sum-lemma} to $W_i$ in the second sum.

Now given $t \in T_i$, specializing \eqref{case-two-identity} at $a = b = \alpha_t$ gives
$$
\begin{aligned}
2 c_i =\langle \alpha_t , c_i \alpha_t \rangle  
      = \langle \alpha_t ,d \Theta_x(\alpha_t) \rangle 
     &=p \cdot \langle \alpha_t, \alpha_t\rangle + 
         \frac{p}{h} (c_i - 1) h_i \cdot \langle \alpha_t,\alpha_t \rangle\\
&= 2p + 2\frac{p}{h} (c_i - 1) h_i
\end{aligned}
$$
yielding $c_i = \frac{ph - ph_i}{h - p h_i}$.  Therefore $c_i \neq 1$, since $p,h > 1$.  
\end{proof}

\section*{Acknowledgments}
The authors are grateful to Alex Miller, Soichi Okada, 
Alex Postnikov, Anne Shepler, Eric Sommers, 
Christian Stump, and an anonymous referee
for helpful suggestions, corrections, conversations, references
and computations, and to Pavel Etingof for his permission to
include Theorem~\ref{Etingof-reduced} and its proof here.

\end{document}